\documentclass{article}
\usepackage{amssymb}
\usepackage{amsmath}
\usepackage{graphicx}
\usepackage{epstopdf}
\usepackage{subfigure}
\usepackage{subfloat}

\newtheorem{theorem}{Theorem}

\newtheorem{definition}[theorem]{Definition}

\newtheorem{lemma}[theorem]{Lemma}

\newenvironment{proof}[1][Proof]{\noindent\textbf{#1.} }{\ \rule{0.5em}{0.5em}}
\input{tcilatex}
\begin{document}

\title{Approximate continuous-discrete filters for the estimation of
diffusion processes from partial and noisy observations}
\author{J.C. Jimenez \\
%EndAName
{\small Instituto de Cibern\'{e}tica, Matem\'{a}tica y F\'{\i}sica}\\
{\small Departamento de Matem\'{a}tica Interdisciplinaria}\\
{\small Calle 15, no. 551, Vedado, La Habana, Cuba}\\
{\small e-mail: jcarlos@icimaf.cu}} \maketitle

\begin{abstract}
In this paper, an alternative approximation to the innovation method
is introduced for the parameter estimation of diffusion processes
from partial and noisy observations. This is based on a convergent
approximation to the first two conditional moments of the innovation
process through approximate continuous-discrete filters of minimum
variance. It is shown that, for finite samples, the resulting
approximate estimators converge to the exact one when the error of
the approximate filters decreases. For an increasing number of
observations, the estimators are asymptotically normal distributed
and their bias decreases when the above mentioned error does it. A
simulation study is provided to illustrate the performance of the
new estimators. The results show that, with respect to the
conventional approximate estimators, the new ones significantly
enhance the parameter estimation of the test equations. The proposed
estimators are intended for the recurrent practical situation where
a nonlinear stochastic system should be identified from a reduced
number of partial and noisy observations distant in time.
\end{abstract}

\section{Introduction}

The statistical inference for diffusion processes described by
Stochastic Differential Equations (SDEs) is currently a subject of
intensive researches. A basic difficulty of this statistical problem
is that, except for a few simple examples, the joint distribution of
the discrete-time observations of the process has unknown
closed-form. In addition, if only some components of the diffusion
process contaminated with noise are observed, then an extra
complication arises. Typically, in this situation, the statistical
problem under consideration is reformulated in the framework of
continuous-discrete state space models, where the SDE to be
estimated defines the continuous state equation and the given
observations are described in terms of an discrete observation
equation. For such class of models, a number of estimators based on
analytical and simulated approximations have been developed in the
last four decades. See, for instance, Nielsen et. al (2000a) and
Jimenez et al. (2006) for a review.

In particular, the present paper deals with the class of innovation
estimators for the parameter estimation of SDEs given a time series
of partial and noisy observations. These are the estimators obtained
by maximizing a normal log-likelihood function of the discrete-time
innovations associated with the underlying continuous-discrete state
space model. Approximations to this class of estimators have been
derived by approximating the discrete-time innovations by means of
inexact filters. With this purpose, approximate continuous-discrete
filters like the Local Linearization (Ozaki 1994, Shoji 1998,
Jimenez \& Ozaki 2006), the extended Kalman (Nielsen \& Madsen 2001,
Singer 2002), and the second order (Nielsen et al. 2000b, Singer
2002) filters have been used, as well as, discrete-discrete filters
after the discretization of the SDE by means of a numerical scheme
(Ozaki \& Iino 2001, and Peng et al. 2002). The approximate
innovation estimators obtained in this way have been useful for the
identification, from actual data, of a variety of
neurophysiological, financial and molecular models among others
(see, e.g., Calderon 2009, Chiarella et al. 2009, Jimenez et al.
2006, Riera et al. 2004, Valdes et al. 1999). However, a common
feature of the approximate innovation estimators mentioned above is
that, once the observations are given, the error between the
approximate and the exact innovations is fixed and completely
determined by the distance between observations. Clearly, this fixes
the bias of the approximate estimators for finite samples and
obstructs its asymptotic correction when the number of observations
increases.

In this paper, an alternative approximation to the innovation
estimator for diffusion processes is introduced, which is oriented
to reduce and control the estimation bias. This is based on a
recursive approximation of the first two conditional moments of the
innovation process through approximate filters that converge to the
linear one of minimum variance. It is shown that, for finite
samples, the resulting approximate estimators converge to the exact
one when the error of the approximate filters decreases. For an
increasing number of observations, they are asymptotically normal
distributed and their bias decreases when the above mentioned error
does it. As a particular instance, the approximate innovation
estimators designed with the order-$\beta$ Local Linearization
filters are presented. Their convergence, practical algorithms and
performance in simulations are also considered in detail. The
simulations show that, with respect to the conventional
approximations to the innovation estimators, the new approximate
estimators significantly enhance the parameter estimation of the
test equations given a reduced number of partial and noisy
observations distant in time, which is a typical situation in many
practical inference problems.

The paper is organized as follows. In section 2, basic notations and
definitions are presented. In section 3, the new approximate
estimators are defined and some of their properties studied. As a
particular instance, the order-$\beta$ innovation estimator based on
convergent Local Linearization filters is presented in Section 4, as
well as algorithms for its practical implementation. In the last
section, the performance of the new estimators is illustrated with
various examples.

\section{Notation and preliminary\label{Inn Method Section}}

Let $(\Omega ,\mathcal{F},P)$ be the underlying complete probability
space and $\{\mathcal{F}_{t},$ $t\geq t_{0}\}$ be an increasing
right continuous family of complete sub $\sigma $-algebras of
$\mathcal{F}$, and $\mathbf{x}$ be a $d$-dimensional diffusion
process defined by the stochastic differential equation
\begin{equation}
d\mathbf{x}(t)=\mathbf{f}(t,\mathbf{x}(t);\mathbf{\theta }%
)dt+\sum\limits_{i=1}^{m}\mathbf{g}_{i}(t,\mathbf{x}(t);\mathbf{\theta })d%
\mathbf{w}^{i}(t)  \label{SS1}
\end{equation}%
for $t\geq t_{0}\in
%TCIMACRO{\U{211d} }%
%BeginExpansion
\mathbb{R}
%EndExpansion
$, where $\mathbf{f}$ and $\mathbf{g}_{i}$ are differentiable functions, $%
\mathbf{w=(\mathbf{w}}^{1},..,\mathbf{w}^{m}\mathbf{)}$ is an $m$%
-dimensional $\mathcal{F}_{t}$-adapted standard Wiener process, $\mathbf{%
\theta }\in \mathcal{D}_{\theta }$ is a vector of parameters, and $\mathcal{D%
}_{\theta }\subset
%TCIMACRO{\U{211d} }%
%BeginExpansion
\mathbb{R}
%EndExpansion
^{p}$ is a compact set. Linear growth, uniform Lipschitz and
smoothness conditions on the functions $\mathbf{f}$ and
$\mathbf{g}_{i}$ that ensure the existence and uniqueness of a
strong solution of (\ref{SS1}) with bounded moments are assumed for
all $\mathbf{\theta }\in \mathcal{D}_{\theta }$.

Let us consider the state space model defined by the continuous
state equation (\ref{SS1}) and the discrete observation equation
\begin{equation}
\mathbf{z}(t_{k})=\mathbf{Cx}(t_{k})+\mathbf{e}_{t_{k}},\text{ for }%
k=0,1,..,M-1,  \label{SS2}
\end{equation}%
where $\{\mathbf{e}_{t_{k}}:\mathbf{e}_{t_{k}}\thicksim \mathcal{N}(0,%
\mathbf{\Pi }_{t_{k}}),$ $k=0,..,M-1\}$ is a sequence of
$r$-dimensional
i.i.d. Gaussian random vectors independent of $\mathbf{w}$, $\mathbf{\Pi }%
_{t_{k}}$ an $r\times r$ positive semi-definite matrix, and $\mathbf{C}$ an $%
r\times d$ matrix. Here, it is assumed that the $M$ time instants
$t_{k}$
define an increasing sequence $\{t\}_{M}=\{t_{k}:t_{k}<t_{k+1}$, $%
k=0,1,..,M-1\}$.

Suppose that,  through (\ref{SS2}), $M$ partial and noisy
observations of the diffusion process $\mathbf{x}$ defined by
(\ref{SS1}) with $\mathbf{\theta =\theta }_{0}\in
\mathcal{D}_{\theta }$
are given on $\{t\}_{M}$. In particular, denote by $Z=\{%
\mathbf{z}_{t_{0}},..,\mathbf{z}_{t_{M-1}}\}$ the sequence of these
observations, where $\mathbf{z}_{t_{k}}$ denotes the observation at
$t_{k}$ for all $t_{k}\in \{t\}_{M}$.

The inference problem to be consider here is the estimation of the
parameter $\mathbf{\theta }_{0}$ of the SDE (\ref{SS1}) given the
time series $Z$. Specifically, let us consider the innovation
estimator defined as follows.

\begin{definition}
(Ozaki, 1994) Given $M$ observations $Z$ of the continuous-discrete state space model (\ref%
{SS1})-(\ref{SS2}) with $\mathbf{\theta =\theta }_{0}\in
\mathcal{D}_{\theta }$ on $\{t\}_{M}$,
\begin{equation}
\widehat{\mathbf{\theta } }_{M}=\arg \{\underset{\mathbf{\theta } }{\mathbf{\min }}U_{M}(%
\mathbf{\theta },Z)\}  \label{Innovation estimator}
\end{equation}%
defines the innovation estimator of $\mathbf{\theta }_{0}$, where
\begin{equation*}
U_{M}(\mathbf{\theta },Z)=(M-1)\ln (2\pi )+\sum\limits_{k=1}^{M-1}\ln (\det (%
\mathbf{\Sigma }_{t_{k}}))+\mathbf{\nu }_{t_{k}}^{\top }(\mathbf{\Sigma }%
_{t_{k}})^{-1}\mathbf{\nu }_{t_{k}},
\end{equation*}%
$\mathbf{\nu }_{t_{k}}$ is the discrete innovations of the model (\ref{SS1}%
)-(\ref{SS2}) and $\mathbf{\Sigma }_{t_{k}}$ the innovation variance
for all $t_{k}\in \{t\}_{M}$.
\end{definition}

In the above definition,
\begin{equation*}
\mathbf{\nu
}_{t_{k}}=\mathbf{z}_{t_{k}}-\mathbf{Cx}_{t_{k}/t_{k-1}}\text{ \
\ \ \ \ and \ \ \ \ \ }\mathbf{\Sigma }_{t_{k}}=\mathbf{CU}_{t_{k}/t_{k-1}}%
\mathbf{C}^{\intercal }+\mathbf{\Pi }_{t_{k}},
\end{equation*}%
where $\mathbf{x}_{t_{k}/t_{k-1}}=E(\mathbf{x}(t_{k})|Z_{t_{k-1}})$ and $%
\mathbf{U}_{t_{k}/t_{k-1}}=E(\mathbf{x}(t_{k})\mathbf{x}^{\intercal
}(t_{k})|Z_{t_{k-1}})-\mathbf{x}_{t_{k}/t_{k-1}}\mathbf{x}%
_{t_{k}/t_{k-1}}^{\intercal }$ denote the conditional mean and
variance of
the diffusion process $\mathbf{x}$ at\ $t_{k}$ given the observations $%
Z_{t_{k-1}}=\{\mathbf{z}_{t_{0}},..,\mathbf{z}_{t_{k-1}}\}$ for all $%
t_{k-1},t_{k}\in \{t\}_{M}$ and $\mathbf{\theta }\in
\mathcal{D}_{\theta }$.
Here, the predictions $E(\mathbf{x}(t_{k})|Z_{t_{k-1}})$ and $E(\mathbf{x}%
(t_{k})\mathbf{x}^{\intercal }(t_{k})|Z_{t_{k-1}})$ are recursively
computed
through the Linear Minimum Variance filter for the model (\ref{SS1})-(\ref%
{SS2}). Because the first two conditional moments of $\mathbf{x}$
are correctly specified, Theorem 1 in Ljung \& Caines (1979) for
prediction error estimators implies the consistent and the
asymptotic normality of the innovation estimator (\ref{Innovation
estimator}) under conventional regularity conditions (Ozaki 1994,
Nolsoe et al. 2000).

In general, since the conditional mean and variance of equation
(\ref{SS1})
have not explicit formulas, approximations to them are needed. If $%
\widetilde{\mathbf{x}}_{t_{k}/t_{k-1}}$ and $\widetilde{\mathbf{U}}%
_{t_{k}/t_{k-1}}$ are approximations to $\mathbf{x}_{t_{k}/t_{k-1}}$ and $%
\mathbf{U}_{t_{k}/t_{k-1}}$, then the estimator
\begin{equation*}
\widehat{\mathbf{\vartheta }}_{M}=\arg \{\underset{\mathbf{\theta }}{\mathbf{%
\min }}\widetilde{U}_{M}\mathbf{(\theta },Z)\},
\end{equation*}%
with%
\begin{equation*}
\widetilde{U}_{M}(\mathbf{\theta },Z)=(M-1)\ln (2\pi
)+\sum\limits_{k=1}^{M-1}\ln (\det (\widetilde{\mathbf{\Sigma }}_{t_{k}}))+%
\widetilde{\mathbf{\nu }}_{t_{k}}^{\top }(\widetilde{\mathbf{\Sigma }}%
_{t_{k}})_{t_{k}}^{-1}\widetilde{\mathbf{\nu }}_{t_{k}}
\end{equation*}%
provides an approximation to the innovation estimator
(\ref{Innovation estimator}), where
\begin{equation*}
\widetilde{\mathbf{\nu }}_{t_{k}}=\mathbf{z}_{t_{k}}-\mathbf{C}\widetilde{%
\mathbf{x}}_{t_{k}/t_{k-1}}\text{ \ \ \ \ \ and \ \ \ \ \ }\widetilde{%
\mathbf{\Sigma }}_{t_{k}}=\mathbf{C}\widetilde{\mathbf{U}}_{t_{k}/t_{k-1}}%
\mathbf{C}^{\intercal }+\mathbf{\Pi }_{t_{k}}
\end{equation*}%
are approximations to $\mathbf{\nu }_{t_{k}}$ and $\mathbf{\Sigma
}_{t_{k}}$.

Approximate estimators of this type have early been considered in a
number of papers. Approximate continuous-discrete filters like Local
Linearization filters (Ozaki 1994, Shoji 1998, Jimenez \& Ozaki
2006), extended Kalman filter (Nielsen \& Madsen 2001, Singer 2002),
and second order filters (Nielsen et al. 2000b, Singer 2002) have
been used for approximating the values of $\mathbf{\nu }_{t_{k}}$
and $\mathbf{\Sigma }_{t_{k}}$. On the other hand, in Ozaki \& Iino
(2001) and Peng et. al (2002), discrete-discrete filters have also
been used after the discretization of the equation (\ref{SS1}) by
means of a numerical scheme. In all these
approximations, once the data $Z$ are given (and so the time partition $%
\{t\}_{M}$ is specified), the error between $\widetilde{\mathbf{\nu }}%
_{t_{k}}$ and $\mathbf{\nu }_{t_{k}}$ is completely settled by $%
t_{k}-t_{k-1} $ and can not be reduced. In this way, the difference
between the approximate innovation estimator
$\widehat{\mathbf{\vartheta }}_{M}$ and the exact one
$\widehat{\mathbf{\theta }}_{M}$ can not be reduced neither.
Clearly, this is a important limitation of these approximate
estimators. Nevertheless, in a number of practical situations (see
Jimenez \& Ozaki 2006, Jimenez et al. (2006), and references
therein) the bias the approximate innovation estimators is
negligible. Therefore, these estimators has been useful for the
identification, from actual data, of a variety of
neurophysiological, financial and molecular models among others as
it was mentioned above. Further, in a simulation study with the
Cox-Ingersoll-Ross model of short-term interest rate, approximate
innovation methods have provided similar or better results than
those obtained by prediction-based estimating functions but with
much lower computational cost (Nolsoe et al., 2000). Similar results
have been reported in a comparative study with the approximate
likelihood via simulation method (Singer, 2002).

Denote by $\mathcal{C}_{P}^{l}(\mathbb{R}^{d},\mathbb{R})$ the space
of $l$ time continuously differentiable functions
$g:\mathbb{R}^{d}\rightarrow \mathbb{R}$ for which $g$ and all its
partial derivatives up to order $l$ have polynomial growth.

\section{Order-$\protect\beta $ innovation estimator}

Let $\left( \tau \right) _{h>0}=\{\tau _{n}:\tau _{n+1}-\tau _{n}\leq h,$ $%
n=0,1,\ldots ,N\}$ be a time discretization of $[t_{0},t_{M-1}]$ such that $%
\left( \tau \right) _{h}\supset \{t\}_{M}$, and $\mathbf{y}_{n}$ be
the approximate value of $\mathbf{x}(\tau _{n})$ obtained from a
discretization of the equation (\ref{SS1}) for all $\tau _{n}\in
\left( \tau \right) _{h}$.
Let us consider the continuous time approximation $\mathbf{y}=\{\mathbf{y}%
(t),$ $t\in \lbrack t_{0},t_{M-1}]:\mathbf{y}(\tau
_{n})=\mathbf{y}_{n}$ for all $\tau _{n}\in \left( \tau \right)
_{h}\}$ of $\mathbf{x}$ with initial
conditions%
\begin{equation}
E\left( \mathbf{y}(t_{0})\text{{\LARGE \TEXTsymbol{\vert}}}\mathcal{F}%
_{t_{0}}\right) =E\left( \mathbf{x}(t_{0})\text{{\LARGE \TEXTsymbol{\vert}}}%
\mathcal{F}_{t_{0}}\right) \text{ \ \ and \ }E\left( \mathbf{y}(t_{0})%
\mathbf{y}^{\intercal }(t_{0})\text{{\LARGE \TEXTsymbol{\vert}}}\mathcal{F}%
_{t_{0}}\right) =E\left( \mathbf{x}(t_{0})\mathbf{x}^{\intercal }(t_{0})%
\text{{\LARGE \TEXTsymbol{\vert}}}\mathcal{F}_{t_{0}}\right) ;\text{
} \label{Inn LMVF5}
\end{equation}%
satisfying the bound condition
\begin{equation}
E\left( \left\vert \mathbf{y}(t)\right\vert ^{2q}\text{{\LARGE \TEXTsymbol{%
\vert}}}\mathcal{F}_{t_{0}}\right) \leq L  \label{Inn LMVF6}
\end{equation}%
for all $t\in \lbrack t_{0},t_{M-1}]$; and the weak convergence
criteria
\begin{equation}
\underset{t_{k}\leq t\leq t_{k+1}}{\sup }\left\vert E\left( g(\mathbf{x}(t))%
\text{{\LARGE \TEXTsymbol{\vert}}}\mathcal{F}_{t_{k}}\right) -E\left( g(%
\mathbf{y}(t))\text{{\LARGE
\TEXTsymbol{\vert}}}\mathcal{F}_{t_{k}}\right) \right\vert \leq
L_{k}h^{\beta }  \label{Inn LMVF7}
\end{equation}%
for all $t_{k},t_{k+1}\in \{t\}_{M}$ and $\mathbf{\theta }\in \mathcal{D}%
_{\theta }$, where $g\in \mathcal{C}_{P}^{2(\beta +1)}(\mathbb{R}^{d},%
\mathbb{R})$, $L$ and $L_{k}$ are positive constants, $\beta \in
%TCIMACRO{\U{2115} }%
%BeginExpansion
\mathbb{N}
%EndExpansion
_{+}$, and $q=1,2...$. The process $\mathbf{y}$ defined in this way is typically called order-$%
\beta $ approximation to $\mathbf{x}$ in weak sense (Kloeden \&
Platen, 1999). The second conditional moment of $\mathbf{y}$ is also
assumed to be positive definite and continuous for all $\mathbf{\theta }\in \mathcal{%
D}_{\theta }$.

In addition, let us consider the following approximation to the
Linear Minimum Variance (LMV) filter of the model
(\ref{SS1})-(\ref{SS2}).

\begin{definition}
(Jimenez 2012b) Given a time discretization $\left( \tau \right)
_{h}\supset \{t\}_{M}$, the order-$\beta $ Linear Minimum Variance
filter for the state space model (\ref{SS1})-(\ref{SS2}) is defined,
between observations, by
\begin{equation}
\mathbf{y}_{t/t}=E(\mathbf{y(}t)/Z_{t})\text{ \ \ \ \ \ \ and \ \ \ }\mathbf{%
V}_{t/t}=E(\mathbf{y(}t)\mathbf{y}^{\intercal }(t)/Z_{t})-\mathbf{y}_{t/t}%
\mathbf{y}_{t/t}^{\intercal }  \label{LMVF12 App}
\end{equation}%
$\ $for all $t\in (t_{k},t_{k+1})$, and by
\begin{equation}
\mathbf{y}_{t_{k+1}/t_{k+1}}=\mathbf{y}_{t_{k+1}/t_{k}}+\mathbf{K}_{t_{k+1}}%
\mathbf{(\mathbf{z}}_{t_{k+1}}-\mathbf{\mathbf{C}y}_{t_{k+1}/t_{k}}\mathbf{)}%
,  \label{LMVF3 App}
\end{equation}%
\begin{equation}
\mathbf{V}_{t_{k+1}/t_{k+1}}=\mathbf{V}_{t_{k+1}/t_{k}}-\mathbf{K}_{t_{k+1}}%
\mathbf{CV}_{t_{k+1}/t_{k}},  \label{LMVF4 App}
\end{equation}%
for each observation at $t_{k+1}$, with filter gain%
\begin{equation}
\mathbf{K}_{t_{k+1}}=\mathbf{V}_{t_{k+1}/t_{k}}\mathbf{C}^{\intercal }%
{\Large (}\mathbf{CV}_{t_{k+1}/t_{k}}\mathbf{C}^{\intercal
}+\mathbf{\Pi }_{t_{k+1}})^{-1}  \label{LMVF5 App}
\end{equation}%
for all $t_{k},t_{k+1}\in \{t\}_{M}$, where $\mathbf{y}$ is an
order-$\beta $
approximation to the solution of (\ref{SS1}) in weak sense, and $Z_{t}=\{%
\mathbf{z(}t_{k}):t_{k}\leq t,$ $t_{k}\in \{t\}_{M}\}$ are given
observations of (\ref{SS1})-(\ref{SS2}) until the time instant $t$.
The
predictions $\mathbf{y}_{t/t_{k}}=E(\mathbf{y(}t)/Z_{t_{k}})$ and $\mathbf{V}%
_{t/t_{k}}=E(\mathbf{y(}t)\mathbf{y}^{\intercal }(t)/Z_{t_{k}})-\mathbf{y}%
_{t/t_{k}}\mathbf{y}_{t/t_{k}}^{\intercal }$, with initial conditions $%
\mathbf{y}_{t_{k}/t_{k}}$ and $\mathbf{V}_{t_{k}/t_{k}}$, are
defined for all $t\in (t_{k},t_{k+1}]$ and $t_{k},t_{k+1}\in
\{t\}_{M}$.
\end{definition}

Once an order-$\beta $ approximation to the solution of equation
(\ref{SS1}) is chosen, and so an order-$\beta $ LMV filter is
specified, the following approximate innovation estimator can
naturally be defined.

\begin{definition} \label{Definition order-B Inn
estimator} Given $M$ observations $Z$ of the continuous-discrete
state space model (\ref {SS1})-(\ref{SS2}) with $\mathbf{\theta
=\theta }_{0}\in \mathcal{D}_{\theta }$ on $\{t\}_{M}$, the
order-$\beta $ innovation estimator for the parameters of
(\ref{SS1}) is defined by
\begin{equation}
\widehat{\mathbf{\theta }}_{M}(h)=\arg \{\underset{\mathbf{\theta
}}{\mathbf{ \min }}U_{M,h}\mathbf{(\theta },Z)\},  \label{order-B
innovation estimator}
\end{equation}%
where
\[
U_{M,h}(\mathbf{\theta },Z)=(M-1)\ln (2\pi
)+\sum\limits_{k=1}^{M-1}\ln (\det ( \mathbf{\Sigma
}_{h,t_{k}}))+\mathbf{\nu }_{h,t_{k}}^{\intercal }(\mathbf{ \Sigma
}_{h,t_{k}})^{-1}\mathbf{\nu }_{h,t_{k}},
\]
with $\mathbf{\nu
}_{h,t_{k}}=\mathbf{z}_{t_{k}}-\mathbf{Cy}_{t_{k}/t_{k-1}}$ and
$\mathbf{\Sigma }_{h,t_{k}}=\mathbf{CV}_{t_{k}/t_{k-1}}\mathbf{C}
^{\intercal }+\mathbf{\Pi }_{t_{k}}$, being
$\mathbf{y}_{t_{k}/t_{k-1}}$ and $\mathbf{V}_{t_{k}/t_{k-1}}$ the
prediction mean and variance of an order-$\beta $ LMV filter for the
model (\ref{SS1})-(\ref{SS2}), and $h$ the maximum stepsize of the
time discretization $\left( \tau \right) _{h}\supset \{t\}_{M}$
associated to the filter.
\end{definition}

In principle, according to the above definitions, any kind of
approximation $ \mathbf{y}$ converging to $\mathbf{x}$ in a weak
sense can be used to construct an approximate order-$\beta $ LMV
filter and so an approximate order-$\beta $ innovation estimator. In
this way, the Euler-Maruyama, the Local Linearization and any high
order numerical scheme for SDEs as those considered in Kloeden \&
Platen (1999) might be used as well. However, the approximations
$\mathbf{\nu }_{h,t_{k}}$ and $\mathbf{\Sigma }_{h,t_{k}}$ to
$\mathbf{\nu }_{t_{k}}$ and $\mathbf{\Sigma }_{t_{k}}$ in
(\ref{Innovation estimator}) at each $t_{k}$ will be now derived
from the predictions of approximate LMV filter after various
iterations with stepsizes lower than $ t_{k}-t_{k-1}$. Note that,
when $\left( \tau \right) _{h}\equiv \{t\}_{M}$, an order-$\beta $
LMV filter might reduce to some one of the conventional
approximation to the exact LMV filter. In this situation, the
corresponding order-$\beta $ innovation estimator reduces to some
one of the approximate innovation estimator mentioned in Section
\ref{Inn Method Section}. In particular, to those considered in
Ozaki (1994), Shoji (1998), Jimenez \& Ozaki (2006) when Local
Linearization schemes are used to define order-$ \beta $ LMV
filters.

Note that the goodness of the approximation $\mathbf{y}$ to
$\mathbf{x}$ is measured (in weak sense) by the left hand side of
(\ref{Inn LMVF7}). Thus, the inequality (\ref{Inn LMVF7}) gives a
bound for the errors of the approximation $\mathbf{y }$ to
$\mathbf{x}$, for all $t\in \lbrack t_{k},t_{k+1}]$ and all pair of
consecutive observations $t_{k},t_{k+1}\in \{t\}_{M}$. Moreover,
this inequality states the convergence (in weak sense and with rate
$\beta $) of the approximation $\mathbf{y}$ to $\mathbf{x}$ as the
maximum stepsize $h$ of the time discretization $(\tau )_{h}\supset
\{t\}_{M}$ goes to zero. Clearly this includes, as particular case,
the convergence of the first two conditional moments of $\mathbf{y}$
to those of $\mathbf{x}$, which implies the convergence of
order-$\beta$ LMV filter (\ref{LMVF12 App})-(\ref{LMVF5 App}) to the
exact LMV filter stated by Theorem 5 in Jimenez (2012b). Since the
approximate innovation estimator (\ref{order-B innovation
estimator}) is designed in terms of the order-$\beta$ LMV filter
(\ref{LMVF12 App})-(\ref{LMVF5 App}), the weak convergence of
$\mathbf{y}$ to $\mathbf{x}$ should then imply the convergence of
the approximate innovation estimator (\ref{order-B innovation
estimator}) to the exact one (\ref{Innovation estimator}) and the
similarity of their asymptotic properties, as $h$ goes to zero. Next
results deal with these matters.

\subsection{Convergence}

For a finite sample $Z$ of $M$ observation of the state space model (\ref%
{SS1})-(\ref{SS2}), Theorem 5 in Jimenez (2012b) states the
convergence of the order-$\beta $ LMV filters to the exact LMV one
when $h$ decreases. Therefore, the convergence of the order-$\beta $
innovation estimator to the exact innovation estimator is
predictable when $h$ goes to zero.

\begin{theorem}
\label{Innovation convergence theorem}Let $Z$ be a time series of
$M$ observations of the state space model (\ref{SS1})-(\ref{SS2})
with $\mathbf{ \theta =\theta }_{0}$ on the time partition
$\{t\}_{M}$. Let $\widehat{ \mathbf{\theta }}_{M}$ and
$\widehat{\mathbf{\theta }}_{M}(h)$ be, respectively, the innovation
and an order-$\beta $ innovation estimator for the parameters of
(\ref{SS1}) given $Z$. Then
\begin{equation*}
\left\vert \widehat{\mathbf{\theta }}_{M}(h)-\widehat{\mathbf{\theta
}} _{M}\right\vert \rightarrow 0
\end{equation*}%
as $h\rightarrow 0$. Moreover,
\[
E(\left\vert \widehat{\mathbf{\theta
}}_{M}(h)-\widehat{\mathbf{\theta }}_{M}\right\vert )\rightarrow 0
\]%
as $h\rightarrow 0$, where the expectation is with respect to the
measure on the underlying probability space generating the
realizations of the model (\ref{SS1})-(\ref{SS2}) with $\mathbf{
\theta =\theta }_{0}$.
\end{theorem}

\begin{proof}
Defining $\Delta \mathbf{\Sigma }_{h,t_{k}}=\mathbf{\Sigma }_{t_{k}}-%
\mathbf{\Sigma }_{h,t_{k}}$, it follows that
\begin{eqnarray}
\det (\mathbf{\Sigma }_{h,t_{k}}) &=&\det (\mathbf{\Sigma
}_{t_{k}}-\Delta
\mathbf{\Sigma }_{h,t_{k}})  \notag \\
&=&\det (\mathbf{\Sigma }_{t_{k}})\det (\mathbf{I}-\mathbf{\Sigma }%
_{t_{k}}^{-1}\Delta \mathbf{\Sigma }_{h,t_{k}})  \label{PLK Ind 1}
\end{eqnarray}%
and%
\begin{eqnarray}
\mathbf{\Sigma }_{h,t_{k}}^{-1} &=&(\mathbf{\Sigma }_{t_{k}}-\Delta \mathbf{%
\Sigma }_{h,t_{k}})^{-1}  \notag \\
&=&\mathbf{\Sigma }_{t_{k}}^{-1}+\mathbf{\Sigma }_{t_{k}}^{-1}\Delta \mathbf{%
\Sigma }_{h,t_{k}}(\mathbf{I}-\mathbf{\Sigma }_{t_{k}}^{-1}\Delta \mathbf{%
\Sigma }_{h,t_{k}})^{-1}\mathbf{\Sigma }_{t_{k}}^{-1}.  \label{PLK
Ind 2}
\end{eqnarray}%
By using these two identities and the identity
\begin{eqnarray}
(\mathbf{z}_{t_{k}}-\mathbf{\mu }_{h,t_{k}})^{\intercal }(\mathbf{\Sigma }%
_{h,t_{k}})^{-1}(\mathbf{z}_{t_{k}}-\mathbf{\mu }_{h,t_{k}}) &=&(\mathbf{z}%
_{t_{k}}-\mathbf{\mu }_{t_{k}})^{\intercal }(\mathbf{\Sigma }%
_{h,t_{k}})^{-1}(\mathbf{z}_{t_{k}}-\mathbf{\mu }_{t_{k}})  \notag \\
&&+(\mathbf{z}_{t_{k}}-\mathbf{\mu }_{t_{k}})^{\intercal }(\mathbf{\Sigma }%
_{h,t_{k}})^{-1}(\mathbf{\mu }_{t_{k}}-\mathbf{\mu }_{h,t_{k}})  \notag \\
&&+(\mathbf{\mu }_{t_{k}}-\mathbf{\mu }_{h,t_{k}})^{\intercal }(\mathbf{%
\Sigma }_{h,t_{k}})^{-1}(\mathbf{z}_{t_{k}}-\mathbf{\mu }_{t_{k}})  \notag \\
&&+(\mathbf{\mu }_{t_{k}}-\mathbf{\mu }_{h,t_{k}})^{\intercal }(\mathbf{%
\Sigma }_{h,t_{k}})^{-1}(\mathbf{\mu }_{t_{k}}-\mathbf{\mu
}_{h,t_{k}}), \label{PLK Ind 3}
\end{eqnarray}%
with $\mathbf{\mu }_{t_{k}}=\mathbf{Cx}_{t_{k}/t_{k-1}}$ and $\mathbf{\mu }%
_{h,t_{k}}=\mathbf{Cy}_{t_{k}/t_{k-1}}$, it is obtained that
\begin{equation}
U_{M,h}(\mathbf{\theta },Z)=U_{M}(\mathbf{\theta
},Z)+R_{M,h}(\mathbf{\theta }),  \label{PLK Ind 4}
\end{equation}%
where $U_{M}$ and $U_{M,h}$ are defined in (\ref{Innovation estimator}) and (%
\ref{order-B innovation estimator}), respectively, and
\begin{eqnarray*}
R_{M,h}(\mathbf{\theta }) &=&\sum\limits_{k=1}^{M-1}\ln (\det (\mathbf{I}-%
\mathbf{\Sigma }_{t_{k}}^{-1}\Delta \mathbf{\Sigma }_{h,t_{k}}))+(\mathbf{z}%
_{t_{k}}-\mathbf{\mu }_{t_{k}})^{\intercal }\mathbf{M}_{h,t_{k}}\mathbf{(z}%
_{t_{k}}-\mathbf{\mu }_{t_{k}}) \\
&&+(\mathbf{z}_{t_{k}}-\mathbf{\mu }_{t_{k}})^{\intercal }(\mathbf{\Sigma }%
_{h,t_{k}})^{-1}(\mathbf{\mu }_{t_{k}}-\mathbf{\mu
}_{h,t_{k}})+(\mathbf{\mu
}_{t_{k}}-\mathbf{\mu }_{h,t_{k}})^{\intercal }(\mathbf{\Sigma }%
_{h,t_{k}})^{-1}(\mathbf{z}_{t_{k}}-\mathbf{\mu }_{t_{k}}) \\
&&+(\mathbf{\mu }_{t_{k}}-\mathbf{\mu }_{h,t_{k}})^{\intercal }(\mathbf{%
\Sigma }_{h,t_{k}})^{-1}(\mathbf{\mu }_{t_{k}}-\mathbf{\mu
}_{h,t_{k}})
\end{eqnarray*}%
with $\mathbf{M}_{h,t_{k}}=\mathbf{\Sigma }_{t_{k}}^{-1}\Delta \mathbf{%
\Sigma }_{h,t_{k}}(\mathbf{I-\Sigma }_{t_{k}}^{-1}\Delta \mathbf{\Sigma }%
_{h,t_{k}})^{-1}\mathbf{\Sigma }_{t_{k}}^{-1}$.

Theorem 5 in Jimenez (2012b) deals with the convergence of the order-$%
\beta $ filters to the exact LMV one. In particular, for the
predictions, it states that
\begin{equation}
\left\vert
\mathbf{x}_{t_{k}/t_{k-1}}-\mathbf{y}_{t_{k}/t_{k-1}}\right\vert
\leq Kh^{\beta }\text{ \ \ \ \ and \ \ }\left\vert \mathbf{U}%
_{t_{k}/t_{k-1}}-\mathbf{V}_{t_{k}/t_{k-1}}\right\vert \leq
Kh^{\beta } \label{Inn LMVF10}
\end{equation}%
for all $t_{k},t_{k+1}\in \{t\}_{M}$, where $K$ is a positive
constant.
Here, we recall that $\mathbf{x}_{t_{k}/t_{k-1}}$ and $\mathbf{U}%
_{t_{k}/t_{k-1}}$ are the predictions of the exact LMV filter for the model (%
\ref{SS1})-(\ref{SS2}), whereas $\mathbf{y}_{t_{k}/t_{k-1}}$ and $\mathbf{V}%
_{t_{k}/t_{k-1}}$ are the predictions of the order-$\beta $ filter.
From
this and taking into account that $\mathbf{\mu }_{t_{k}}-\mathbf{\mu }%
_{h,t_{k}}=\mathbf{C(\mathbf{x}}_{t_{k}/t_{k-1}}-\mathbf{\mathbf{y}}%
_{t_{k}/t_{k-1}}\mathbf{)}$ and $\Delta \mathbf{\Sigma }_{h,t_{k}}=$\ $%
\mathbf{C(U}_{t_{k}/t_{k-1}}-\mathbf{V}_{t_{k}/t_{k-1}})\mathbf{C}%
^{\intercal }$ follows that%
\begin{equation*}
\left\vert \mathbf{\mu }_{t_{k}}-\mathbf{\mu }_{h,t_{k}}\right\vert
\rightarrow \mathbf{0}\text{ \ \ \ \ \ \ and \ \ \ \ \ \ \
}\left\vert \mathbf{\Sigma }_{t_{k}}-\mathbf{\Sigma
}_{h,t_{k}}\right\vert \rightarrow \mathbf{0}
\end{equation*}%
as $h\rightarrow 0$ for all $\mathbf{\theta }\in \mathcal{D}_{\theta }$ and $%
k=1,..,M-1$. This and the finite bound for the first two conditional
moments
of $\mathbf{x}$ and $\mathbf{y}$ imply that $R_{M,h}(\mathbf{\theta }%
)\rightarrow \mathbf{0}$ as well with $h$. From this and (\ref{PLK
Ind 4}),
\begin{equation}
\left\vert \widehat{\mathbf{\theta }}_{M}(h)-\widehat{\mathbf{\theta }}%
_{M}\right\vert =\left\vert \underset{\mathbf{\theta }}{\arg \{\mathbf{\min }%
}\{U_{M}(\mathbf{\theta },Z)+R_{M,h}(\mathbf{\theta })\}\}-\arg \{\underset{%
\mathbf{\theta }}{\text{ }\mathbf{\min }}U_{M}(\mathbf{\theta }%
,Z)\}\right\vert \rightarrow 0  \label{Inn LMVF11}
\end{equation}%
as $h\rightarrow 0$, which implies the first assertion of the
theorem.

On the other hand, since the constant $K$ in (\ref{Inn LMVF10}) does
not depend of a specific realization of the model
(\ref{SS1})-(\ref{SS2}), from these inequalities follows that
\begin{equation*}
E(\left\vert
\mathbf{x}_{t_{k}/t_{k-1}}-\mathbf{y}_{t_{k}/t_{k-1}}\right\vert)
\leq Kh^{\beta }\text{ \ \ \ \ and \ \ }E(\left\vert \mathbf{U}
_{t_{k}/t_{k-1}}-\mathbf{V}_{t_{k}/t_{k-1}}\right\vert) \leq
Kh^{\beta },
\end{equation*}%
where the new expectation here is with respect to the measure on the
underlying probability space generating the realizations of the
model (\ref{SS1})-(\ref{SS2}) with $\mathbf{ \theta =\theta }_{0}$.
From this and (\ref{Inn LMVF11}) follows that $E(\left\vert
\widehat{\mathbf{\theta }}_{M}(h)-\widehat{ \mathbf{\theta
}}_{M}\right\vert )\rightarrow 0$ as $h\rightarrow 0$, which
concludes the proof.
\end{proof}

The first assertion of this theorem states that, for each given data
$Z$, the order-$\beta $ innovation estimator
$\widehat{\mathbf{\theta }}_{M}(h)$ converges to the exact one
$\widehat{\mathbf{\theta }}_{M}$ as $h$ goes to zero. Because $h$
controls the weak convergence criteria (\ref{Inn LMVF7}) is then
clear that the order-$\beta $ innovation estimator (\ref{order-B
innovation estimator}) converges to the exact one (\ref{Innovation
estimator}) when the error (in weak sense) of the order-$\beta $
approximation $\mathbf{y}$ to $\mathbf{x}$ decreases or,
equivalently, when the error between the order-$\beta $ filter and
the exact LMV filter decreases. On the other hand, the second
assertion implies that the average of the errors $\left\vert
\widehat{\mathbf{\theta }}_{M}(h)-\widehat{\mathbf{ \theta
}}_{M}\right\vert $ corresponding to different realizations of the
model (\ref{SS1})-(\ref{SS2}) decreases when $h$ does.

Next theorem deals with error between the averages of the estimators
$ \widehat{\mathbf{\theta }}_{M}(h)$ and $\widehat{\mathbf{\theta
}}_{M}$ computed for different realizations of the state space
model.

\begin{theorem}
\label{Inn week convergence}Let $Z$ be a time series of $M$
observations of the state space model (\ref{SS1})-(\ref{SS2}) with
$\mathbf{ \theta =\theta }_{0}$ on the time partition $\{t\}_{M}$.
Let $\widehat{ \mathbf{\theta }}_{M}$ and $\widehat{\mathbf{\theta
}}_{M}(h)$ be, respectively, the innovation and an order-$\beta $
innovation estimator for the parameters of (\ref{SS1}) given $Z$.
Then,
\[
\left\vert E(\widehat{\mathbf{\theta
}}_{M}(h))-E(\widehat{\mathbf{\theta }} _{M})\right\vert \rightarrow
0
\]%
as $h\rightarrow 0$, where the expectation is with respect to the
measure on the underlying probability space generating the
realizations of the model (\ref{SS1})-(\ref{SS2}) with $\mathbf{
\theta =\theta }_{0}$.
\end{theorem}

\begin{proof}
Trivially,
\begin{eqnarray*}
\left\vert E(\widehat{\mathbf{\theta
}}_{M}(h))-E(\widehat{\mathbf{\theta }} _{M})\right\vert
&=&\left\vert E(\widehat{\mathbf{\theta }}_{M}(h)-\widehat{
\mathbf{\theta }}_{M})\right\vert  \\
&\leq &E(\left\vert \widehat{\mathbf{\theta
}}_{M}(h)-\widehat{\mathbf{ \theta }}_{M}\right\vert ),
\end{eqnarray*}%
where the expectation here is taken with respect to the measure on
the underlying probability space generating the realizations of the
model (\ref{SS1})-(\ref{SS2}) with $\mathbf{ \theta =\theta }_{0}$.
From this and the second assertion of Theorem \ref{Innovation
convergence theorem}, the proof is completed.
\end{proof}

Here, it is worth to remak that the conventional approximate
innovation estimators mentioned in Section \ref{Inn Method Section}
do not have the desired convergence properties stated in the
theorems above for the order-$\beta $ innovation estimator. Further
note that, either in Definition \ref{Definition order-B Inn
estimator} nor in Theorems \ref{Innovation convergence theorem} and
\ref{Inn week convergence} some restriction on the time partition
$\{t\}_{M}$ for the data has been assumed. Thus, there are not
specific constraints about the time distance between two consecutive
observations, which allows the application of the order-$\beta $
innovation estimator in a variety of practical problems with a
reduced number of not close observations in time, with sequential
random measurements, or with multiple missing data. Neither there
are restrictions on the time discretization $(\tau )_{h}$ $ \supset
\{t\}_{M}$ on which the order-$\beta $ innovation estimator is
defined. Thus, $(\tau )_{h}$ can be set by the user by taking into
account some specifications or previous knowledge on the inference
problem under consideration, or automatically designed by an
adaptive strategy as it will be shown in the section concerning the
numerical simulations.

\subsection{Asymptotic properties\label{PLK Theoretical Section}}

In this section, asymptotic properties of the approximate innovation
estimator $\widehat{\mathbf{\theta }}_{M}(h)$ will be studied by
using a general result obtained in Ljung and Caines (1979) for
prediction error
estimators. According to that, the relation between the estimator $\widehat{%
\mathbf{\theta }}_{M}(h)$ and the global minimum $\mathbf{\theta
}_{M}^{\ast }$ of the function
\begin{equation}
W_{M}(\mathbf{\theta })=E(U_{M}(\mathbf{\theta },Z))\text{ with }\mathbf{%
\theta }\in \mathcal{D}_{\theta }  \label{PLKW}
\end{equation}%
should be considered, where $U_{M}$ is defined in (\ref{Innovation estimator}%
) and the expectation is taken with respect to the measure on the
underlying probability space generating the realizations of the
state space model (\ref{SS1})-(\ref{SS2}). Here, it is worth to
remark that $\mathbf{\theta }_{M}^{\ast }$ is not an estimator of
$\mathbf{\theta }$ since the function $W_{M}$ does not depend of a
given data $Z$. In fact, $\mathbf{\theta }_{M}^{\ast }$ indexes the
best predictor, in the sense that the average prediction error loss
function $W_{M}$ is minimized at this parameter (Ljung \& Caines,
1979).

In what follows, regularity conditions for the unique
identifiability of the state space model (\ref{SS1})-(\ref{SS2}) are
assumed, which are typically satisfied by stationary and ergodic
diffusion processes (see, e.g., Ljung \& Caines, 1979).

\begin{lemma}
\label{Innovation Lemma}If $\mathbf{\Sigma }_{t_{k}}$ is positive
definite for all $k=1,..,M-1$, then the function
$W_{M}(\mathbf{\theta })$ defined in (\ref{PLKW}) has an unique
minimum and
\begin{equation}
\arg \{\underset{\mathbf{\theta \in }\mathcal{D}_{\theta }}{\mathbf{\min }}%
W_{M}(\mathbf{\theta })\}=\mathbf{\theta }_{0}.  \label{M-estimator}
\end{equation}
\end{lemma}

\begin{proof}
Since $\mathbf{\Sigma }_{t_{k}}$ is positive definite for all
$k=1,..,M-1$,
Lemma A.2 in Bollerslev \& Wooldridge\textbf{\ }(1992) ensures that $\mathbf{%
\theta }_{0}$ is the unique minimum of the function%
\begin{equation*}
l_{k}(\mathbf{\theta })=E(\ln (\det (\mathbf{\Sigma }_{t_{k}}))+\mathbf{\nu }%
_{t_{k}}^{\top }(\mathbf{\Sigma }_{t_{k}})^{-1}\mathbf{\nu }%
_{t_{k}}|Z_{t_{k-1}})
\end{equation*}%
on $\mathcal{D}_{\theta }$ for all $k$, where $\mathbf{\nu }_{t_{k}}=\mathbf{%
z}_{t_{k}}-\mathbf{Cx}_{t_{k}/t_{k-1}}$\ and $\mathbf{\Sigma }_{t_{k}}=%
\mathbf{CU}_{t_{k}/t_{k-1}}\mathbf{C}^{\intercal }+\mathbf{\Pi
}_{t_{k}}$.
Consequently and under the assumed unique identifiability of the model (\ref%
{SS1})-(\ref{SS2}), $\mathbf{\theta }_{0}$ is then the unique
minimum of
\begin{equation*}
W_{M}(\mathbf{\theta })=(M-1)\ln (2\pi )+\sum\limits_{k=1}^{M-1}E(l_{k}(\mathbf{%
\theta }))
\end{equation*}%
on $\mathcal{D}_{\theta }.$
\end{proof}

Denote by $U_{M,h}^{\prime }$ the derivative of $U_{M,h}$ with respect to $%
\mathbf{\theta }$, and by $W_{M}^{^{\prime \prime }}$ the second
derivative of $W_{M}$ with respect to $\mathbf{\theta }$.

\begin{theorem}
\label{Innovation main theorem}Let $Z$ be a time series of $M$
observations of the state space model (\ref{SS1})-(\ref{SS2}) with
$\mathbf{\theta
=\theta }_{0}$ on the time partition $\{t\}_{M}$. \ Let $\widehat{\mathbf{%
\theta }}_{M}(h)$ be an order-$\beta $ innovation estimator for the
parameters of (\ref{SS1}) given $Z$. Then
\begin{equation}
\widehat{\mathbf{\theta }}_{M}(h)-\mathbf{\theta }_{0}\rightarrow
\Delta \mathbf{\theta }_{M}(h)  \label{PLK1}
\end{equation}%
w.p.1 as $M\rightarrow \infty $, where $\Delta \mathbf{\theta }%
_{M}(h)\rightarrow 0$ as $h\rightarrow 0$. Moreover, if for some
$M_{0}\in
%TCIMACRO{\U{2115} }%
%BeginExpansion
\mathbb{N}
%EndExpansion
$ there exists $\epsilon >0$ such that%
\begin{equation}
W_{M}^{^{\prime \prime }}(\mathbf{\theta })>\epsilon
\mathbf{I}\text{ \ \ \
and \ \ }\mathbf{H}_{M,h}(\mathbf{\theta )}=ME(U_{M,h}^{\prime }(\mathbf{%
\theta },Z\mathbf{)(}U_{M,h}^{\prime }(\mathbf{\theta },Z\mathbf{))}%
^{\intercal })>\epsilon \mathbf{I}  \label{PLK4}
\end{equation}%
for all $M>M_{0}$ and $\mathbf{\theta }\in \mathcal{D}_{\theta }$, then%
\begin{equation}
\sqrt{M}\mathbf{P}_{M,h}^{-1/2}(\widehat{\mathbf{\theta }}_{M}(h)-\mathbf{%
\theta }_{0})\sim \mathcal{N}(\Delta \mathbf{\theta
}_{M}(h),\mathbf{I}) \label{PLK2}
\end{equation}%
as $M\rightarrow \infty $, where $\mathbf{P}_{M,h}=(W_{M}^{^{\prime
\prime
}}(\mathbf{\theta }_{0}+\Delta \mathbf{\theta }_{M}(h)))^{-1}\mathbf{H}%
_{M,h}(\mathbf{\theta }_{0}+\Delta \mathbf{\theta
}_{M}(h))(W_{M}^{^{\prime \prime }}(\mathbf{\theta }_{0}+\Delta
\mathbf{\theta }_{M}(h)))^{-1}+\Delta
\mathbf{P}_{M,h}$ with $\Delta \mathbf{P}_{M,h}\rightarrow \mathbf{0}$ as $%
h\rightarrow 0$.
\end{theorem}

\begin{proof}
Let $W_{M,h}(\mathbf{\theta })=E(U_{M,h}(\mathbf{\theta },Z))$ and $\mathbf{%
\alpha }_{M}(h)=\arg \{\underset{\mathbf{\theta \in }\mathcal{D}_{\theta }}{%
\mathbf{\min }}W_{M,h}(\mathbf{\theta })\}$, where $U_{M,h}$ is defined in (%
\ref{order-B innovation estimator}).

For a $h$ fixed, Theorem 1 in Ljung \& Caines (1979) implies that%
\begin{equation}
\widehat{\mathbf{\theta }}_{M}(h)-\mathbf{\alpha }_{M}(h)\rightarrow
0 \label{PLK7}
\end{equation}%
w.p.1 as $M\rightarrow \infty $; and
\begin{equation}
\sqrt{M}\mathbf{P}_{M,h}^{-1/2}(\mathbf{\alpha }_{M}(h))(\widehat{\mathbf{%
\theta }}_{M}(h)-\mathbf{\alpha }_{M}(h))\sim
\mathcal{N}(0,\mathbf{I}) \label{PLK8}
\end{equation}%
as $M\rightarrow \infty $, where
\begin{equation*}
\mathbf{P}_{M,h}(\mathbf{\theta })=(W_{M,h}^{\prime \prime }(\mathbf{\theta }%
))^{-1}\text{ }\mathbf{H}_{M,h}(\mathbf{\theta })\text{
}(W_{M,h}^{\prime \prime }(\mathbf{\theta }))^{-1}
\end{equation*}%
with $\mathbf{H}_{M,h}(\mathbf{\theta })=ME(U_{M,h}^{\prime
}(\mathbf{\theta },Z\mathbf{)(}U_{M,h}^{\prime }(\mathbf{\theta
},Z\mathbf{))}^{\intercal })$.

By using the identities (\ref{PLK Ind 1})-(\ref{PLK Ind 3}), the
function
\begin{equation*}
W_{M,h}(\mathbf{\theta })=(M-1)\ln (2\pi )+\sum\limits_{k=1}^{M-1}E(\ln (\det (%
\mathbf{\Sigma }_{h,t_{k}}))+(\mathbf{z}_{t_{k}}-\mathbf{\mu }%
_{h,t_{k}})^{\intercal }(\mathbf{\Sigma }_{h,t_{k}})^{-1}(\mathbf{z}_{t_{k}}-%
\mathbf{\mu }_{h,t_{k}})),
\end{equation*}%
with $\mathbf{\mu }_{h,t_{k}}=\mathbf{Cy}_{t_{k}/t_{k-1}}$, can be
written as
\begin{equation}
W_{M,h}(\mathbf{\theta })=W_{M}(\mathbf{\theta })+E(R_{M,h}(\mathbf{\theta }%
)),  \label{PLK3}
\end{equation}%
where $W_{M}$ is defined in (\ref{PLKW}) and%
\begin{eqnarray*}
R_{M,h}(\mathbf{\theta }) &=&\sum\limits_{k=1}^{M-1}E(\ln (\det (\mathbf{I}-%
\mathbf{\Sigma }_{t_{k}}^{-1}\Delta \mathbf{\Sigma }_{h,t_{k}}))|\mathcal{F}%
_{t_{k-1}})+E((\mathbf{z}_{t_{k}}-\mathbf{\mu }_{t_{k}})^{\intercal }\mathbf{%
M}_{h,t_{k}}\mathbf{(z}_{t_{k}}-\mathbf{\mu }_{t_{k}})|\mathcal{F}_{t_{k-1}}) \\
&&+E((\mathbf{z}_{t_{k}}-\mathbf{\mu }_{t_{k}})^{\intercal }(\mathbf{\Sigma }%
_{h,t_{k}})^{-1}(\mathbf{\mu }_{t_{k}}-\mathbf{\mu }_{h,t_{k}})|\mathcal{F}%
_{t_{k-1}})+E((\mathbf{\mu }_{t_{k}}-\mathbf{\mu }_{h,t_{k}})^{\intercal }(%
\mathbf{\Sigma }_{h,t_{k}})^{-1}(\mathbf{z}_{t_{k}}-\mathbf{\mu }_{t_{k}})|%
\mathcal{F}_{t_{k-1}}) \\
&&+E((\mathbf{\mu }_{t_{k}}-\mathbf{\mu }_{h,t_{k}})^{\intercal }(\mathbf{%
\Sigma }_{h,t_{k}})^{-1}(\mathbf{\mu }_{t_{k}}-\mathbf{\mu }_{h,t_{k}})|%
\mathcal{F}_{t_{k-1}})
\end{eqnarray*}%
with $\mathbf{M}_{h,t_{k}}=\mathbf{\Sigma }_{t_{k}}^{-1}\Delta \mathbf{%
\Sigma }_{h,t_{k}}(\mathbf{I-\Sigma }_{t_{k}}^{-1}\Delta \mathbf{\Sigma }%
_{h,t_{k}})^{-1}\mathbf{\Sigma }_{t_{k}}^{-1}$, $\mathbf{\mu }_{t_{k}}=%
\mathbf{Cx}_{t_{k}/t_{k-1}}$ and $\Delta \mathbf{\Sigma }_{h,t_{k}}=%
\mathbf{\Sigma }_{t_{k}}-\mathbf{\Sigma }_{h,t_{k}}$.

Denote by $W_{M,h}^{\prime \prime }$ and $R_{M,h}^{\prime \prime }$
the second derivative of $W_{M,h}$ and $R_{M,h}$ with respect to
$\mathbf{\theta }$.

Taking into account that
\begin{eqnarray*}
(W_{M,h}^{\prime \prime }(\mathbf{\theta }))^{-1} &=&(W_{M}^{\prime \prime }(%
\mathbf{\theta })+E(R_{M,h}^{\prime \prime }(\mathbf{\theta })))^{-1} \\
&=&(W_{M}^{\prime \prime }(\mathbf{\theta }))^{-1}+\mathbf{K}_{M,h}(\mathbf{%
\theta })
\end{eqnarray*}%
with
\begin{equation*}
\mathbf{K}_{M,h}(\mathbf{\theta })=-(W_{M}^{\prime \prime }(\mathbf{\theta }%
))^{-1}E(R_{M,h}^{\prime \prime }(\mathbf{\theta }))(\mathbf{I}%
+(W_{M}^{\prime \prime }(\mathbf{\theta }))^{-1}E(R_{M,h}^{\prime \prime }(%
\mathbf{\theta })))^{-1}(W_{M}^{\prime \prime }(\mathbf{\theta
}))^{-1},
\end{equation*}%
it is obtained that%
\begin{equation}
\mathbf{P}_{M,h}(\mathbf{\theta })=(W_{M}^{^{\prime \prime
}}(\mathbf{\theta
}))^{-1}\mathbf{H}_{M,h}(\mathbf{\theta })(W_{M}^{^{\prime \prime }}(\mathbf{%
\theta }))^{-1}+\Delta \mathbf{P}_{M,h}(\mathbf{\theta }),
\label{PLK12}
\end{equation}%
where%
\begin{equation*}
\Delta \mathbf{P}_{M,h}(\mathbf{\theta })=\mathbf{K}_{M,h}(\mathbf{\theta })%
\mathbf{H}_{M,h}(\mathbf{\theta })(W_{M}^{^{\prime \prime }}(\mathbf{\theta }%
))^{-1}+(W_{M}^{^{\prime \prime }}(\mathbf{\theta }))^{-1}\mathbf{H}_{M,h}(%
\mathbf{\theta })\mathbf{K}_{M,h}(\mathbf{\theta })+\mathbf{K}_{M,h}(\mathbf{%
\theta })\mathbf{H}_{M,h}(\mathbf{\theta })\mathbf{K}_{M,h}(\mathbf{\theta }%
).
\end{equation*}

Theorem 5 in Jimenez (2012b) deals with the convergence of the order-$%
\beta $ filters to the exact LMV one. In particular, for the
predictions, it states that
\begin{equation*}
\left\vert
\mathbf{x}_{t_{k}/t_{k-1}}-\mathbf{y}_{t_{k}/t_{k-1}}\right\vert
\leq Kh^{\beta }\text{ \ \ \ \ and \ \ }\left\vert \mathbf{U}%
_{t_{k}/t_{k-1}}-\mathbf{V}_{t_{k}/t_{k-1}}\right\vert \leq
Kh^{\beta }
\end{equation*}%
for all $t_{k},t_{k+1}\in \{t\}_{M}$, where $K$ is a positive
constant.
Here, we recall that $\mathbf{x}_{t_{k}/t_{k-1}}$ and $\mathbf{U}%
_{t_{k}/t_{k-1}}$ are the predictions of the exact LMV filter for the model (%
\ref{SS1})-(\ref{SS2}), whereas $\mathbf{y}_{t_{k}/t_{k-1}}$ and $\mathbf{V}%
_{t_{k}/t_{k-1}}$ are the predictions of the order-$\beta $ filter.
From
this and taking into account that $\mathbf{\mu }_{t_{k}}-\mathbf{\mu }%
_{h,t_{k}}=\mathbf{C(\mathbf{x}}_{t_{k}/t_{k-1}}-\mathbf{\mathbf{y}}%
_{t_{k}/t_{k-1}}\mathbf{)}$ and $\Delta \mathbf{\Sigma }_{h,t_{k}}=$\ $%
\mathbf{C(U}_{t_{k}/t_{k-1}}-\mathbf{V}_{t_{k}/t_{k-1}})\mathbf{C}%
^{\intercal }$ follows that%
\begin{equation*}
\left\vert \mathbf{\mu }_{t_{k}}-\mathbf{\mu }_{h,t_{k}}\right\vert
\rightarrow \mathbf{0}\text{ \ \ \ \ \ \ and \ \ \ \ \ }\left\vert \mathbf{%
\Sigma }_{t_{k}}-\mathbf{\Sigma }_{h,t_{k}}\right\vert \rightarrow
\mathbf{0}
\end{equation*}%
as $h\rightarrow 0$ for all $\mathbf{\theta }\in \mathcal{D}_{\theta }$ and $%
k=1,..,M-1$. This and the finite bound for the first two conditional
moments
of $\mathbf{x}$ and $\mathbf{y}$ imply that $\left\vert R_{M,h}(\mathbf{%
\theta },Z)\right\vert \rightarrow 0$ and $\left\vert
R_{M,h}^{\prime \prime }(\mathbf{\theta },Z)\right\vert \rightarrow
0$ as well with $h$. From this and (\ref{PLK3}), it is obtained that
\begin{equation}
W_{M,h}(\mathbf{\theta })\rightarrow W_{M}(\mathbf{\theta })\text{ \
\ \ \ and \ \ \ \ }W_{M,h}^{\prime \prime }(\mathbf{\theta
})\rightarrow
W_{M}^{\prime \prime }(\mathbf{\theta })\text{ \ \ \ as \ \ \ \ \ }%
h\rightarrow 0.  \label{PLK9}
\end{equation}%
In addition, left (\ref{PLK9}) and Lemma \ref{Innovation Lemma}
imply that
\begin{equation}
\Delta \mathbf{\theta }_{M}(h)=\mathbf{\alpha }_{M}(h)-\mathbf{\theta }%
_{0}=\arg \{\underset{\mathbf{\theta \in }\mathcal{D}_{\theta }}{\mathbf{%
\min }}W_{M,h}(\mathbf{\theta })\}-\arg \{\underset{\mathbf{\theta \in }%
\mathcal{D}_{\theta }}{\mathbf{\min }}W_{M}(\mathbf{\theta })\}\rightarrow 0%
\text{ \ \ \ as \ \ \ \ \ }h\rightarrow 0,  \label{PLK11}
\end{equation}%
whereas from right (\ref{PLK9}) follows that%
\begin{equation}
\Delta \mathbf{P}_{M,h}(\mathbf{\theta })\rightarrow 0\text{ \ \ \
as \ \ \ \ \ }h\rightarrow 0.  \label{PLK10}
\end{equation}

Finally, (\ref{PLK11})-(\ref{PLK10}) together (\ref{PLK7}),
(\ref{PLK8}) and (\ref{PLK12}) imply that (\ref{PLK1}) and
(\ref{PLK2}) hold, which completes the proof.
\end{proof}

Theorem \ref{Innovation main theorem} states that, for an increasing
number of observations, the order-$\beta $ innovation estimator
$\widehat{\mathbf{ \theta }}_{M}(h)$ is asymptotically normal
distributed and its bias decreases when $h$ goes to zeros. This is a
predictable result due to the asymptotic properties of the exact
innovation estimator $\widehat{\mathbf{ \theta }}_{M}$ derived from
Theorem 1 in Ljung \& Caines (1979) and the convergence of the
approximate estimator $\widehat{\mathbf{\theta }}_{M}(h)$ to
$\widehat{\mathbf{\theta }}_{M}$ given by Theorem \ref{Innovation
convergence theorem} when $h$ goes to zero. Further note that, when
$h=0$, the Theorem \ref{Innovation main theorem} reduces to Theorem
1 in Ljung \& Caines (1979) for the exact innovation estimator
$\widehat{\mathbf{\theta }} _{M}$. This is other expected result
since the order-$\beta $ innovation estimator
$\widehat{\mathbf{\theta }}_{M}(h)$ reduces to the exact one $
\widehat{\mathbf{\theta }}_{M}$ when $h=0$. Further note that,
neither in Theorem \ref{Innovation main theorem} there are
restrictions on the time partition $\{t\}_{M}$ for the data or on
the time discretization $(\tau )_{h}$ $ \supset \{t\}_{M}$ on which
the approximate estimator is defined. Therefore, the comments about
them at the end of the previous subsection are valid here as well.

\subsection{Models with nonlinear observation equation}

Previous definitions and results have been stated for models with
linear observation equation. However, by following the procedure
proposed in Jimenez and Ozaki (2006), they can be easily applied as
well to state space models with nonlinear observation equation.

For illustrate this, let us consider the state space model defined
by the continuous state equation (\ref{SS1}) and the discrete
observation equation

\begin{equation}
\mathbf{z}(t_{k})=\mathbf{h}(t_{k},\mathbf{x}(t_{k}))+\mathbf{e}_{t_{k}},%
\text{ for }k=0,1,..,M-1,  \label{SS3}
\end{equation}%
where $\mathbf{e}_{t_{k}}$ is defined as in (\ref{SS2}) and $\mathbf{h}:$ $%
\mathbb{R}\times \mathbb{R}^{d}\rightarrow \mathbb{R}^{r}$ is a
twice differentiable function. By using the Ito formula,
\begin{align*}
d\mathbf{h}^{j}& =\mathbf{\{}\frac{\partial \mathbf{h}^{j}}{\partial t}%
+\sum\limits_{k=1}^{d}f^{k}\frac{\partial \mathbf{h}^{j}}{\partial \mathbf{x}%
^{k}}+\frac{1}{2}\sum\limits_{s=1}^{m}\sum\limits_{k,l=1}^{d}\mathbf{g}%
_{s}^{l}\mathbf{g}_{s}^{k}\frac{\partial
^{2}\mathbf{h}^{j}}{\partial
\mathbf{x}^{l}\partial \mathbf{x}^{k}}\mathbf{\}}dt+\sum\limits_{s=1}^{m}%
\sum\limits_{l=1}^{d}\mathbf{g}_{s}^{l}\frac{\partial \mathbf{h}^{j}}{%
\partial \mathbf{x}^{l}}d\mathbf{w}^{s} \\
& =\mathbf{\rho }^{j}dt+\sum\limits_{s=1}^{m}\mathbf{\sigma }_{s}^{j}d%
\mathbf{w}^{s}
\end{align*}%
with $j=1,..,r$. Hence, the state space model
(\ref{SS1})+(\ref{SS3}) is transformed to the following
higher-dimensional state space model with linear observation
\begin{equation*}
d\mathbf{v}(t)=\mathbf{a}(t,\mathbf{v}(t))dt+\sum\limits_{i=1}^{m}\mathbf{b}%
_{i}(t,\mathbf{v}(t))d\mathbf{w}^{i}(t),
\end{equation*}%
\begin{equation*}
\mathbf{z}(t_{k})=\mathbf{Cv}(t_{k})+\mathbf{e}_{t_{k}},\text{ for }%
k=0,1,..,M-1,
\end{equation*}%
where
\begin{equation*}
\mathbf{v}=\left[
\begin{array}{l}
\mathbf{x} \\
\mathbf{h}%
\end{array}%
\right] ,\text{ }\mathbf{a}=\left[
\begin{array}{l}
\mathbf{f} \\
\mathbf{\rho }%
\end{array}%
\right] ,\text{ }\mathbf{b}_{i}=\left[
\begin{array}{l}
\mathbf{g}_{i} \\
\mathbf{\sigma }_{i}%
\end{array}%
\right]
\end{equation*}%
and the matrix $\mathbf{C}$ is such that $\mathbf{h}(t_{k},\mathbf{x}%
(t_{k}))=\mathbf{Cv}(t_{k})$.

In this way, the state space model (\ref{SS1})+(\ref{SS3}) is
transformed to the form of the model (\ref{SS1})-(\ref{SS2}), and so
the previous definition and results related to the order-$\beta $
innovation estimator can be applied. Further, note that if the
nonlinear function $h$ depends of unknown parameters, they can be
estimated as well by the approximate innovation method.

\subsection{Models with noise free complete observations\label{Section on
complete observations}}

This section deals with the particular case that the observation
noise is zero and all components of the diffusion process defined in
(\ref{SS1}) are
discretely observed. That is, when $\mathbf{C}\equiv \mathbf{I}$ and $%
\mathbf{\Pi }_{t_{k}}=\mathbf{0}$ in (\ref{SS2}) for all $k$, where $\mathbf{%
I}$ denotes the $d$-dimensional identity matrix. Hence, the
inference problem under consideration in this paper reduces then to
the well known problem of parameter estimation of diffusion
processes from complete observations. In this situation, it is easy
to realize that the innovation estimator (\ref{Innovation
estimator}) reduces to the well known quasi-maximum likelihood (QML)
estimator for SDEs, and that the approximate order-$\beta $
innovation estimator (\ref{order-B innovation estimator}) reduces to
the approximate order-$\beta $ QML estimator introduced in Jimenez
(2012c) for the estimation of SDEs from complete observations. For
the same reason, Theorems \ref{Innovation convergence theorem},
\ref{Inn week convergence} and \ref{Innovation main theorem} reduce
to those corresponding in Jimenez (2012c) concerning the convergence
and asymptotic properties of the approximate order-$\beta $ QML
estimator.

\section{Order-$\protect\beta $ innovation estimator based on LL filters}

Since, in principle, any approximate filter converging to LMV filter
of the model (\ref{SS1})-(\ref{SS2}) can be used to construct an
order-$\beta $ innovation estimator, some additional criterions
could be considered for the selection of one of them. For instance,
high order of convergence, efficiency of the algorithm from
computational viewpoint, and so on. In this paper, we elected the
order-$\beta $ Local Linearization (LL) filters proposed in Jimenez
(2012b) for the following reasons: 1) their predictions have simple
explicit formulas that can be computed by means of efficient
algorithm (including high dimensional equations); 2) their
predictions are exact for linear SDEs in all the possible variants
(with additive and/or
multiplicative noise, autonomous or not); 3) they have an adequate order $%
\beta =1,2$ of convergence; and 4) the better performance of the
approximate innovation estimators based on conventional LL filters
(see, e.g., Ozaki, 1994; Shoji, 1998; Singer, 2002).

According to Jimenez (2012b), the order-$\beta $ LL filter is defined on $%
\left( \tau \right) _{h}\supset \{t\}_{M}$ in terms of the
order-$\beta $
Local Linear approximation $\mathbf{y}$ that satisfies the conditions (\ref%
{Inn LMVF5})-(\ref{Inn LMVF7}). Denote by $\mathbf{y}_{\tau _{n}/t_{k}}=E(%
\mathbf{y}(\tau _{n})|Z_{t_{k}})$ and $\mathbf{P}_{\tau _{n}/t_{k}}=E(%
\mathbf{y(}\tau _{n})\mathbf{y}^{\intercal }(\tau _{n})|Z_{t_{k}})$
the first two conditional moment of $\mathbf{y}$ at $\tau _{n}$
given the
observations $Z_{t_{k}}$, for all $\tau _{n}\in \{\left( \tau \right) _{h}$ $%
\cap $ $[t_{k},t_{k+1}]\}$ and $k=0,..,M-2$.

Starting with the initial filter values $\mathbf{y}_{t_{0}/t_{0}}=\mathbf{x}%
_{t_{0}/t_{0}}$ and
$\mathbf{P}_{t_{0}/t_{0}}=\mathbf{Q}_{t_{0}/t_{0}}$, the LL filter
algorithm performs the recursive computation of :

\begin{enumerate}
\item the predictions $\mathbf{y}_{\tau _{n}/t_{k}}$and $\mathbf{P}_{\tau
_{n}/t_{k}}$ for all $\tau _{n}\in \{\left( \tau \right) _{h}$ $\cap $ $%
(t_{k},t_{k+1}]\}$ by means of the recursive formulas (\ref{ALLF8}%
)-(\ref{ALLF9}) given in the Appendix, and the prediction variance
\begin{equation*}
\mathbf{V}_{t_{k+1}/t_{k}}=\mathbf{P}_{t_{k+1}/t_{k}}-\mathbf{y}%
_{t_{k+1}/t_{k}}\mathbf{y}_{t_{k+1}/t_{k}}^{\intercal };
\end{equation*}

\item the filters
\begin{align*}
\mathbf{y}_{t_{k+1}/t_{k+1}}& =\mathbf{y}_{t_{k+1}/t_{k}}+\mathbf{K}%
_{t_{k+1}}\mathbf{(\mathbf{z}}_{t_{k+1}}-\mathbf{\mathbf{C}y}_{t_{k+1}/t_{k}}%
\mathbf{)}, \\
\mathbf{V}_{t_{k+1}/t_{k+1}}& =\mathbf{V}_{t_{k+1}/t_{k}}-\mathbf{K}%
_{t_{k+1}}\mathbf{CV}_{t_{k+1}/t_{k}}, \\
\mathbf{P}_{t_{k+1}/t_{k+1}}& =\mathbf{V}_{t_{k+1}/t_{k+1}}+\mathbf{y}%
_{t_{k+1}/t_{k+1}}\mathbf{y}_{t_{k+1}/t_{k+1}}^{\intercal },
\end{align*}%
with filter gain%
\begin{equation*}
\mathbf{K}_{t_{k+1}}=\mathbf{V}_{t_{k+1}/t_{k}}\mathbf{C}^{\intercal }%
{\Large (}\mathbf{CV}_{t_{k+1}/t_{k}}\mathbf{C}^{\intercal
}+\mathbf{\Pi }_{t_{k+1}})^{-1};
\end{equation*}
\end{enumerate}

\noindent for each $k$, with $k=0,1,\ldots ,M-2$. \indent

Under general conditions, the convergence of the order-$\beta $ LL
filter to exact LMV filter when $h$ goes to zero has been stated by
Theorem 10 in Jimenez (2012b). Hence, Theorem \ref{Innovation
convergence theorem} implies that the LL-based innovation estimator
\begin{equation}
\widehat{\mathbf{\theta }}_{M}(h)=\arg \{\underset{\mathbf{\theta }}{\mathbf{%
\min }}U_{M,h}\mathbf{(\theta },Z)\},  \label{LL-based innovation
estimator}
\end{equation}%
with%
\begin{equation*}
U_{M,h}(\mathbf{\theta },Z)=(M-1)\ln (2\pi )+\sum\limits_{k=1}^{M-1}\ln (\det (\mathbf{\Sigma }_{t_{k}/t_{k-1}}))+(\mathbf{z}_{t_{k}}-\mathbf{Cy}%
_{t_{k}/t_{k-1}})^{\intercal }(\mathbf{\Sigma }_{t_{k}/t_{k-1}})^{-1}(\mathbf{z}%
_{t_{k}}-\mathbf{Cy}_{t_{k}/t_{k-1}}),
\end{equation*}%
converges to the exact one (\ref{Innovation estimator}) as $h$ goes
to zero for all given $Z$, where $\mathbf{\Sigma
}_{t_{k}/t_{k-1}}=\mathbf{CV}_{t_{k}/t_{k-1}}\mathbf{C}^{\intercal
}+\mathbf{\Pi }_{t_{k}}$. For the same reason, this order-$\beta $
innovation estimator has the asymptotic properties stated in Theorem
\ref{Innovation main theorem}, and the average of their values for
different realizations of the state space model
(\ref{SS1})-(\ref{SS2}) satisfies the convergence property of
Theorem \ref{Inn week convergence}.

Note that, when $\left( \tau \right) _{h}\equiv \{t\}_{M}$, the
order-$\beta $ LL filter reduces to the conventional LL filter. In
this situation, the order-$\beta $ innovation estimator
(\ref{LL-based innovation estimator}) reduces to the conventional
innovation estimators of Ozaki (1994) or Shoji (1998) for SDEs with
additive noise, and to that of Jimenez and Ozaki (2006) for SDEs
with multiplicative noise. It is worth to emphasize here that, for
each data $\mathbf{z}_{k}$, the formulas (\ref{ALLF8})-(\ref{ALLF9})
for the predictions are recursively evaluated at all the time
instants $\tau _{n}\in \{\left( \tau \right) _{h}$ $\cap $
$(t_{k},t_{k+1}]\}$ for the order-$\beta$ estimator, whereas they
are evaluated only at $t_{k+1}=\{\{t\}_{M}$ $\cap $ $
(t_{k},t_{k+1}]\}$ for the conventional ones. In addition, since the
predictions of the order-$\beta $ LL filter are exact for linear
SDEs, the order-$\beta $ innovation estimator (\ref{LL-based
innovation estimator}) reduces to the maximum likelihood estimator
of Schweppe (1965) for linear equations with additive noise.

In practical situations, it is convenient to write a code that
automatically determines the time discretization $\left( \tau
\right) _{h}$ for achieving
a prescribed absolute ($atol_{\mathbf{y}},atol_{\mathbf{P}})$ and relative ($%
rtol_{\mathbf{y}},rtol_{\mathbf{P}})$ error tolerance in the computation of $%
\mathbf{y}_{t_{k+1}/t_{k}}$ and $\mathbf{P}_{t_{k+1}/t_{k}}$. With
this purpose the adaptive strategy proposed in Jimenez (2012b) is
useful.

\section{Simulation study}

In this section, the performance of the new approximate estimators
is illustrated, by means of simulations, with four test SDEs. To do
so, four
types of innovation estimators are computed and compared: 1) the exact one (%
\ref{Innovation estimator}), when it is possible; 2) the
conventional one based on the LL filter. That is, the estimator
defined by (\ref{LL-based innovation estimator}) with $\left( \tau
\right) _{h}\equiv \{t\}_{M}$ and $\beta =1$; 3) the order-$1$
innovation estimator (\ref{LL-based innovation estimator}) with
various uniform time
discretizations $\left( \tau \right) _{h,T}^{u}$; and 4) the adaptive order-$%
1$ innovation estimator (\ref{LL-based innovation estimator}) with
the adaptive selection of time discretizations $\left( \tau \right)
_{\cdot ,T}$ proposed in Jimenez (2012b). For each example,
histograms and confidence limits for the estimators are computed
from various sets of discrete and noisy observations taken with
different time distances (sampling periods) on time intervals with
distinct lengths.

\subsection{Test models\label{Section Test models}}

\textbf{Example 1.} State equation with multiplicative noise%
\begin{equation}
dx=\alpha txdt+\sigma \sqrt{t}xdw_{1}  \label{SE EJ1}
\end{equation}%
and observation equation
\begin{equation}
z_{t_{k}}=x(t_{k})+e_{t_{k}},\text{ for }k=0,1,..,M-1,  \label{OE
EJ1}
\end{equation}%
with $\alpha =-0.1$, $\sigma =0.1$ and observation noise variance
$\Pi =0.0001$. For this state equation, the predictions for the
first two conditional moments are
\[
x_{t_{k+1}/t_{k}}=x_{t_{k}/t_{k}}e^{\alpha
(t_{k+1}^{2}-t_{k}^{2})/2}\text{ \ and\ \ \ \
}Q_{t_{k+1}/t_{k}}=Q_{t_{k}/t_{k}}e^{(\alpha +\sigma
^{2}/2)(t_{k+1}^{2}-t_{k}^{2})},
\]%
where the filters $x_{t_{k}/t_{k}}$ and $Q_{t_{k}/t_{k}}$ are
obtained from the well-known formulas of the exact LMV filter for
all $k=0,1,..,M-2$, with initial values $x_{t_{0}/t_{0}}=1$ and
$Q_{t_{0}/t_{0}}=1$ at $t_{0}=0.5$.

\textbf{Example 2.} State equation with two additive noise%
\begin{equation}
dx=\alpha txdt+\sigma t^{2}e^{\alpha t^{2}/2}dw_{1}+\rho
\sqrt{t}dw_{2} \label{SE EJ2}
\end{equation}%
and observation equation
\begin{equation}
z_{t_{k}}=x(t_{k})+e_{t_{k}},\text{ for }k=0,1,..,M-1,  \label{OE
EJ2}
\end{equation}%
with $\alpha =-0.25$, $\sigma =5$, $\rho =0.1$  and observation
noise variance $\Pi =0.0001$. For this state equation, the
predictions for the
first two conditional moments are%
\[
x_{t_{k+1}/t_{k}}=x_{t_{k}/t_{k}}e^{\alpha
(t_{k+1}^{2}-t_{k}^{2})/2}\text{ }
\]%
and
\[
Q_{t_{k+1}/t_{k}}=(Q_{t_{k}/t_{k}}+\frac{\rho ^{2}}{2\alpha
})e^{\alpha
(t_{k+1}^{2}-t_{k}^{2})}+\frac{\sigma ^{2}}{5}(t_{k+1}^{5}-t_{k}^{5})e^{%
\alpha t_{k+1}^{2}}-\frac{\rho ^{2}}{2\alpha },
\]%
where the filters $x_{t_{k}/t_{k}}$ and $Q_{t_{k}/t_{k}}$ are
obtained from the formulas of the exact LMV filter for all
$k=0,1,..,M-2$, with initial values $x_{t_{0}/t_{0}}=10$ and
$Q_{t_{0}/t_{0}}=100$ at $t_{0}=0.01$.

\textbf{Example 3.} Van der Pool oscillator with random input
(Gitterman,
2005)%
\begin{align}
dx_{1}& =x_{2}dt  \label{SEa EJ3} \\
dx_{2}& =(-(x_{1}^{2}-1)x_{2}-x_{1}+\alpha )dt+\sigma dw  \label{SEb
EJ3}
\end{align}%
and observation equation
\begin{equation}
z_{t_{k}}=x_{1}(t_{k})+e_{t_{k}},\text{ for }k=0,1,..,M-1, \label{OE
EJ3}
\end{equation}%
where $\alpha =0.5$ and $\sigma ^{2}=(0.75)^{2}$ are the intensity
and the variance of the random input, respectively. In addition,
$\Pi =0.001$ is the
observation noise variance, and $\mathbf{x}_{t_{0}/t_{0}}^{\intercal }=[1$ $%
1]$ and $\mathbf{Q}_{t_{0}/t_{0}}=\mathbf{x}_{t_{0}/t_{0}}\mathbf{x}%
_{t_{0}/t_{0}}^{\intercal }$ are the initial filter values at
$t_{0}=0$.

\textbf{Example 4.} Van der Pool oscillator with random frequency
(Gitterman, 2005)%
\begin{align}
dx_{1}& =x_{2}dt  \label{SEa EJ4} \\
dx_{2}& =(-(x_{1}^{2}-1)x_{2}-\alpha x_{1})dt+\sigma x_{1}dw
\label{SEb EJ4}
\end{align}%
and observation equation
\begin{equation}
z_{t_{k}}=x_{1}(t_{k})+e_{t_{k}},\text{ for }k=0,1,..,M-1, \label{OE
EJ4}
\end{equation}%
where $\alpha =1$ and $\sigma ^{2}=1$ are the frequency mean value
and
variance, respectively. $\Pi =0.001$ is the observation noise variance, and $%
\mathbf{x}_{t_{0}/t_{0}}^{\intercal }=[1$ $1]$ and $\mathbf{Q}_{t_{0}/t_{0}}=%
\mathbf{x}_{t_{0}/t_{0}}\mathbf{x}_{t_{0}/t_{0}}^{\intercal }$ are
the initial filter values at $t_{0}=0$.

In these examples, autonomous or non autonomous, linear or
nonlinear, one or two dimensional SDEs with additive or
multiplicative noise are considered for the estimation of two or
three parameters. Note that, since the first two conditional moments
of the SDEs in Examples 1 and 2 have explicit expressions, the exact
innovation estimator (\ref{Innovation estimator}) can be computed.

These four state space models have previously been used in Jimenez
(2012b) to illustrate the convergence of the order-$\beta $ LL
filter by means of simulations. Tables with the errors between the
approximate moments and the exact ones as a function of $h$ were
given for the Examples 1 and 2. Tables with the estimated rate of
convergence were provided for the fours examples.

\subsection{Simulations with one-dimensional state equations}

For the first two examples, $100$ realizations of the state equation
solution were computed by means of the Euler (Kloeden \& Platen,
1999) or the Local Linearization scheme (Jimenez et al., 1999) for
the equation with multiplicative or additive noise, respectively.
For each example, the realizations where computed over the thin time
partition $ \{t_{0}+10^{-4}n:n=0,..,30\times 10^{4}\}$ to guarantee
a precise simulation of the stochastic solutions on the time
interval $[t_{0},t_{0}+30]$. Twelve
subsamples of each realization at the time instants $\{t\}_{M,T}=%
\{t_{k}=t_{0}+kT/M:$ $k=0,..,M-1\}$ were taken for evaluating the
corresponding observation equation with various values of $M$ and
$T$. In particular, the values $T=10,20,30$ and $M=T/\delta $ with
$\delta =1,0.1,0.01,0.001$ were used. In this way, twelve sets of
$100$ time series $ Z_{\delta ,T}^{i}=\{z_{t_{k}}^{i}:t_{k}\in
\{t\}_{M,T},$ $M=T/\delta \}$, with $i=1,..,100$, of $M$
observations $z_{t_{k}}^{i}$ each one were finally available for
both state space models to make inference. This will allow us to
explore and compare the performance of each estimator from
observations taken with different sampling periods $\delta $\ on
time intervals with distinct lengths $T$.

Figure 1 shows the histograms and the confidence limits for both,
the exact ( $\widehat{\alpha }_{\delta ,T}^{E}$) and the
conventional ($\widehat{\alpha } _{\delta ,T}$) innovation
estimators of $\alpha $ computed from the twelve sets of $100$ time
series $Z_{\delta ,T}^{i}$ available for the example 1. Figure 2
shows the same but, for the exact ($\widehat{\sigma }_{\delta
,T}^{E}$) and the conventional ($\widehat{\sigma }_{\delta ,T}$)
innovation estimators of $\sigma $. As it was expected, for the
samples $Z_{\delta ,T}^{i}$ with largest sampling periods, the
parameter estimation is distorted by the well-known lowpass filter
effect of signals sampling (see, e.g., Oppenheim \& Schafer, 2010).
This is the reason of the under estimation of the variance
$\widehat{\sigma }_{\delta ,T}^{E}$ from the samples $Z_{\delta
,T}^{i}$, with $\delta =1$ and $T=10,20,30$, when the parameter
$\alpha $ in the drift coefficient of (\ref{SE EJ1}) is better
estimated by $\widehat{\alpha }_{\delta ,T}^{E}$. Contrarily, from
these samples, the conventional innovation estimators
$\widehat{\alpha }_{\delta ,T}$ can not provided a good
approximation to $\alpha $, and so the whole unexplained component
of the drift coefficient of (\ref{SE EJ1}) included in the samples
is interpreted as noise by the conventional estimators. For this
reason, $\widehat{\sigma }_{\delta ,T}$ over estimates the value of
the parameter $\sigma $. Further, note that when the sampling period
$\delta $ decreases, the difference between the exact
($\widehat{\alpha }_{\delta ,T}^{E},\widehat{\sigma }_{\delta
,T}^{E}$) and the conventional ($\widehat{ \alpha }_{\delta
,T},\widehat{\sigma }_{\delta ,T}$) innovation estimators decreases,
as well as the bias of both estimators. This is also other expected
result. Here, the bias is estimated by the difference between the
parameter value and the estimator average, whereas the difference
between estimators refers to the histogram shape and confidence
limits.

For the data of (\ref{SE EJ1}) with largest sampling period $\delta
=1$, the order-$1$ innovation estimators ($\widehat{\alpha
}_{h,\delta ,T}^{u},\widehat{ \sigma }_{h,\delta ,T}^{u}$) and
($\widehat{\alpha }_{\cdot ,\delta ,T}, \widehat{\sigma }_{\cdot
,\delta ,T}$) on uniform $\left( \tau \right) _{h,T}^{u}=\{\tau
_{n}=t_{0}+nh:$ $n=0,..,T/h\}\supset \{t\}_{T/\delta ,T}$ and
adaptive $\left( \tau \right) _{\cdot ,T}\supset \{t\}_{T/\delta
,T}$ time discretizations, respectively, were computed with
$h=\delta /2,\delta /8,\delta /32$ and tolerances
$rtol_{\mathbf{y}}=rtol_{\mathbf{P}}=5\times 10^{-6}$ and
$atol_{\mathbf{y}}=5\times 10^{-9}$, $atol_{\mathbf{P}}=5\times
10^{-12}$. For each data $Z_{\delta ,T}^{i}$, with $ i=1,..,100$,
the errors
\begin{equation*}
\varepsilon _{i}(\alpha ,h,\delta ,T)=\left\vert \widehat{\alpha
}_{\delta ,T}^{E}-\widehat{\alpha }_{h,\delta ,T}^{u}\right\vert
\text{ and } \varepsilon _{i}(\sigma ,h,\delta ,T)=\left\vert
\widehat{\sigma }_{\delta ,T}^{E}-\widehat{\sigma }_{h,\delta
,T}^{u}\right\vert
\end{equation*}%
between the exact ($\widehat{\alpha }_{\delta
,T}^{E},\widehat{\sigma } _{\delta ,T}^{E}$) and the approximate
($\widehat{\alpha }_{h,\delta ,T}^{u}, \widehat{\sigma }_{h,\delta
,T}^{u}$) innovation estimators were computed. Average and standard
deviation of these $100$ errors were calculated for each set of
values $h,\delta ,T$ specified above, which are summarized in Table
I. Note as, for fixed $T$, the average of the errors decreases as
$h$ does it. This clearly illustrates the convergence of the
order-$1$ innovation estimators to the exact one stated in Theorem
\ref{Innovation convergence theorem} when $h$ goes to zero. In
addition, Figure 3 shows the histograms and the confidence limits
for the order-$1$ innovation estimators ($\widehat{ \alpha
}_{h,\delta ,T}^{u},\widehat{\sigma }_{h,\delta ,T}^{u}$) and ($
\widehat{\alpha }_{\cdot ,\delta ,T},\widehat{\sigma }_{\cdot
,\delta ,T}$) for each set of values $h,\delta ,T$. By comparing the
results of this figure with the corresponding in the previous ones,
the decreasing difference between the order-$1$ innovation
estimators ($\widehat{\alpha } _{h,\delta ,T}^{u},\widehat{\sigma
}_{h,\delta ,T}^{u}$) and the exact one ($ \widehat{\alpha }_{\delta
,T}^{E},\widehat{\sigma }_{\delta ,T}^{E}$) is observed as $h$
decreases, which is consistent with the convergence results of Table
I. These findings are more precisely summarized in Table II, which
shows the difference between the averages of the exact and the
approximate innovation estimators. Further, note the small
difference between the adaptive estimators ($\widehat{\alpha
}_{\cdot ,\delta ,T},\widehat{\sigma }_{\cdot ,\delta ,T}$) and the
exact ones ($\widehat{\alpha }_{\delta ,T}^{E}, \widehat{\sigma
}_{\delta ,T}^{E}$), which illustrates the usefulness of the
adaptive strategy for improving the innovation\ parameter estimation
for finite samples with large sampling periods. The number of
accepted and fail steps of the adaptive innovation estimators at
each $t_{k}\in \{t\}_{T/\delta ,T}$ are shown in Figure 4. Further,
note that the results of Table II illustrate the convergence
findings of Theorem \ref{Inn week convergence}.
\newline
\newline

\begin{tabular}{|l|c|c|c|c|}
\hline $\delta =1$ & $h=\delta $ & $h=\delta /2$ & $h=\delta /8$ &
$h=\delta /32$
\\ \hline
\begin{tabular}{ll}
& $T=10$ \\
$\alpha$ & $T=20$ \\
& $T=30$%
\end{tabular}
& \multicolumn{1}{|l|}{%
\begin{tabular}{l}
$7.5\pm 5.5\times 10^{-3}$ \\ \hline $7.7\pm 8.0\times 10^{-3}$ \\
\hline
$7.1\pm 5.2\times 10^{-3}$%
\end{tabular}%
} & \multicolumn{1}{|l|}{%
\begin{tabular}{l}
$1.8\pm 1.2\times 10^{-3}$ \\ \hline $1.7\pm 1.2\times 10^{-3}$ \\
\hline
$1.7\pm 1.2\times 10^{-3}$%
\end{tabular}%
} & \multicolumn{1}{|l|}{%
\begin{tabular}{l}
$2.9\pm 2.3\times 10^{-4}$ \\ \hline $2.7\pm 2.2\times 10^{-4}$ \\
\hline
$2.7\pm 2.2\times 10^{-4}$%
\end{tabular}%
} & \multicolumn{1}{|l|}{%
\begin{tabular}{l}
$6.8\pm 5.6\times 10^{-5}$ \\ \hline $6.4\pm 5.3\times 10^{-5}$ \\
\hline
$6.3\pm 5.3\times 10^{-5}$%
\end{tabular}%
} \\ \hline
\begin{tabular}{ll}
& $T=10$ \\
$\sigma $ & $T=20$ \\
& $T=30$%
\end{tabular}
& \multicolumn{1}{|l|}{%
\begin{tabular}{l}
$3.2\pm 1.9\times 10^{-2}$ \\ \hline $3.2\pm 1.9\times 10^{-2}$ \\
\hline
$3.2\pm 1.9\times 10^{-2}$%
\end{tabular}%
} & \multicolumn{1}{|l|}{%
\begin{tabular}{l}
$1.0\pm 0.6\times 10^{-2}$ \\ \hline $1.0\pm 0.6\times 10^{-2}$ \\
\hline
$1.0\pm 0.6\times 10^{-2}$%
\end{tabular}%
} & \multicolumn{1}{|l|}{%
\begin{tabular}{l}
$2.1\pm 1.1\times 10^{-3}$ \\ \hline $2.1\pm 1.1\times 10^{-3}$ \\
\hline
$2.1\pm 1.1\times 10^{-3}$%
\end{tabular}%
} & \multicolumn{1}{|l|}{%
\begin{tabular}{l}
$5.1\pm 2.6\times 10^{-4}$ \\ \hline $5.1\pm 2.6\times 10^{-4}$ \\
\hline
$5.1\pm 2.6\times 10^{-4}$%
\end{tabular}%
} \\ \hline
\end{tabular}

{\small Table I. Confidence limits for the error between the exact
and the approximate innovation estimators of the equation (\ref{SE
EJ1}). }$h=\delta
${\small , for the conventional; and }$h=\delta /2,\delta /8,\delta /32,$%
{\small \ for the order-$1$ on }$\left( \tau \right) _{h,T}^{u}${\small .%
\newline \newline
}

\begin{tabular}{|c|c|c|}
\hline $\delta =1$ & $\alpha$ & $\sigma $ \\ \hline
\begin{tabular}{c}
$h$ \\ \hline $\delta $ \\ \hline $\delta /2$ \\ \hline $\delta /8$
\\ \hline $\delta /32$ \\ \hline
$\cdot $%
\end{tabular}
&
\begin{tabular}{ccc}
$T=10$ & $T=20$ & $T=30$ \\ \hline \multicolumn{1}{r}{$-0.00403$} &
\multicolumn{1}{r}{$-0.00433$} & \multicolumn{1}{r}{$-0.00373$} \\
\hline \multicolumn{1}{r}{$-0.00083$} &
\multicolumn{1}{r}{$-0.00077$} & \multicolumn{1}{r}{$-0.00077$} \\
\hline \multicolumn{1}{r}{$-0.00004$} &
\multicolumn{1}{r}{$-0.00002$} & \multicolumn{1}{r}{$-0.00002$} \\
\hline
\multicolumn{1}{r}{$0$} & \multicolumn{1}{r}{$0.00001$} & \multicolumn{1}{r}{%
$0$} \\ \hline \multicolumn{1}{r}{$-0.00010$} &
\multicolumn{1}{r}{$-0.00014$} &
\multicolumn{1}{r}{$-0.00010$}%
\end{tabular}
&
\begin{tabular}{ccc}
$T=10$ & $T=20$ & $T=30$ \\ \hline \multicolumn{1}{r}{$-0.0321$} &
\multicolumn{1}{r}{$-0.0321$} & \multicolumn{1}{r}{$-0.0321$} \\
\hline \multicolumn{1}{r}{$-0.0107$} & \multicolumn{1}{r}{$-0.0106$}
& \multicolumn{1}{r}{$-0.0106$} \\ \hline
\multicolumn{1}{r}{$-0.0021$} & \multicolumn{1}{r}{$-0.0021$} &
\multicolumn{1}{r}{$-0.0021$} \\ \hline
\multicolumn{1}{r}{$-0.0005$} & \multicolumn{1}{r}{$-0.0005$} &
\multicolumn{1}{r}{$-0.0005$} \\ \hline
\multicolumn{1}{r}{$-0.0003$} & \multicolumn{1}{r}{$-0.0002$} &
\multicolumn{1}{r}{$-0.0003$}%
\end{tabular}
\\ \hline
\end{tabular}

{\small Table II: Difference between the averages of the exact and
the
approximate innovation estimators for the equation (\ref{SE EJ1}). }$%
h=\delta ${\small , for the conventional; }$h=\delta /2,\delta
/8,\delta
/32, ${\small \ for the order-$1$ on }$\left( \tau \right) _{h,T}^{u}$%
{\small ; and }$h=\cdot ,${\small \ for the adaptive order-$1$ on
}$\left( \tau \right) _{\cdot ,T}${\small.
\newline \newline }

\resizebox{\textwidth}{!}{
\begin{tabular}{|c|c|c|c|}
\hline $\delta =0.1$ & $\alpha$ & $\sigma $ & $\rho $ \\ \hline
\begin{tabular}{c}
$h$ \\ \hline $\delta $ \\ \hline $\delta /2$ \\ \hline $\delta /4$
\\ \hline $\delta /8$ \\ \hline
$\cdot $%
\end{tabular}
&
\begin{tabular}{ccc}
$T=10$ & $T=20$ & $T=30$ \\ \hline \multicolumn{1}{r}{$0.00039$} &
\multicolumn{1}{r}{$0.00031$} & \multicolumn{1}{r}{$0.00029$} \\
\hline \multicolumn{1}{r}{$0.00010$} & \multicolumn{1}{r}{$0.00007$}
& \multicolumn{1}{r}{$0.00007$} \\ \hline
\multicolumn{1}{r}{$0.00003$} & \multicolumn{1}{r}{$0.00002$} &
\multicolumn{1}{r}{$0.00001$} \\ \hline
\multicolumn{1}{r}{$0.00001$} & \multicolumn{1}{r}{$0$} & \multicolumn{1}{r}{%
$0$} \\ \hline \multicolumn{1}{r}{$-0.00005$} &
\multicolumn{1}{r}{$0.00002$} &
\multicolumn{1}{r}{$0.00009$}%
\end{tabular}
&
\begin{tabular}{ccc}
$T=10$ & $T=20$ & $T=30$ \\ \hline \multicolumn{1}{r}{$-0.0311$} &
\multicolumn{1}{r}{$-0.0291$} & \multicolumn{1}{r}{$-0.0287$} \\
\hline \multicolumn{1}{r}{$-0.0067$} & \multicolumn{1}{r}{$-0.0059$}
& \multicolumn{1}{r}{$-0.0060$} \\ \hline
\multicolumn{1}{r}{$-0.0013$} & \multicolumn{1}{r}{$-0.0012$} &
\multicolumn{1}{r}{$-0.0012$} \\ \hline
\multicolumn{1}{r}{$-0.0002$} & \multicolumn{1}{r}{$-0.0001$} &
\multicolumn{1}{r}{$-0.0001$} \\ \hline
\multicolumn{1}{r}{$0$} & \multicolumn{1}{r}{$-0.0023$} & \multicolumn{1}{r}{%
$-0.0106$}%
\end{tabular}
&
\begin{tabular}{ccc}
$T=10$ & $T=20$ & $T=30$ \\ \hline \multicolumn{1}{r}{$-2.13\times
10^{-4}$} & $2.4\times 10^{-5}$ & $3.3\times 10^{-5}$ \\ \hline
\multicolumn{1}{r}{$-0.54\times 10^{-4}$} & $1.4\times 10^{-5}$ &
$1.8\times 10^{-5}$ \\ \hline \multicolumn{1}{r}{$-0.05\times
10^{-4}$} & $1.0\times 10^{-5}$ & $1.1\times 10^{-5}$ \\ \hline
\multicolumn{1}{r}{$0.03\times 10^{-4}$} & $0.6\times 10^{-5}$ &
$0.6\times 10^{-5}$ \\ \hline \multicolumn{1}{r}{$2.14\times
10^{-4}$} & $2.8\times 10^{-5}$ & $9.4\times
10^{-5}$%
\end{tabular}
\\ \hline
\end{tabular}
}

{\small Table III: Difference between the averages of the exact and
the
approximate innovation estimators for the equation (\ref{SE EJ2}). }$%
h=\delta ${\small , for the conventional; }$h=\delta /2,\delta /4,\delta /8,$%
{\small \ for the order-$1$ on }$\left( \tau \right)
_{h,T}^{u}${\small ; and }$h=\cdot ,${\small \ for the adaptive
order-$1$ on }$\left( \tau \right) _{\cdot ,T}${\small.
\newline
\newline
}

\begin{tabular}{|l|c|c|c|c|}
\hline $\delta =0.1$ & $h=\delta $ & $h=\delta /2$ & $h=\delta /4$ &
$h=\delta /8$
\\ \hline
\begin{tabular}{ll}
& $T=10$ \\
$\alpha$ & $T=20$ \\
& $T=30$%
\end{tabular}
& \multicolumn{1}{|l|}{%
\begin{tabular}{l}
$5.2\pm 4.0\times 10^{-4}$ \\ \hline $5.5\pm 4.0\times 10^{-4}$ \\
\hline
$5.4\pm 3.9\times 10^{-4}$%
\end{tabular}%
} & \multicolumn{1}{|l|}{%
\begin{tabular}{l}
$1.1\pm 0.9\times 10^{-4}$ \\ \hline $1.2\pm 0.8\times 10^{-4}$ \\
\hline
$1.1\pm 0.8\times 10^{-4}$%
\end{tabular}%
} & \multicolumn{1}{|l|}{%
\begin{tabular}{l}
$2.7\pm 2.0\times 10^{-5}$ \\ \hline $2.6\pm 1.8\times 10^{-5}$ \\
\hline
$2.6\pm 1.8\times 10^{-5}$%
\end{tabular}%
} & \multicolumn{1}{|l|}{%
\begin{tabular}{l}
$7.4\pm 5.6\times 10^{-6}$ \\ \hline $7.1\pm 5.9\times 10^{-6}$ \\
\hline
$6.7\pm 5.4\times 10^{-6}$%
\end{tabular}%
} \\ \hline
\begin{tabular}{ll}
& $T=10$ \\
$\sigma $ & $T=20$ \\
& $T=30$%
\end{tabular}
& \multicolumn{1}{|l|}{%
\begin{tabular}{l}
$4.8\pm 3.5\times 10^{-2}$ \\ \hline $4.9\pm 3.5\times 10^{-2}$ \\
\hline
$4.9\pm 3.4\times 10^{-2}$%
\end{tabular}%
} & \multicolumn{1}{|l|}{%
\begin{tabular}{l}
$9.7\pm 6.6\times 10^{-3}$ \\ \hline $1.0\pm 0.7\times 10^{-2}$ \\
\hline
$1.0\pm 0.7\times 10^{-2}$%
\end{tabular}%
} & \multicolumn{1}{|l|}{%
\begin{tabular}{l}
$1.8\pm 1.4\times 10^{-3}$ \\ \hline $1.9\pm 1.5\times 10^{-3}$ \\
\hline
$1.9\pm 1.4\times 10^{-3}$%
\end{tabular}%
} & \multicolumn{1}{|l|}{%
\begin{tabular}{l}
$3.6\pm 3.6\times 10^{-4}$ \\ \hline $3.9\pm 4.7\times 10^{-4}$ \\
\hline
$3.7\pm 3.8\times 10^{-4}$%
\end{tabular}%
} \\ \hline
\begin{tabular}{ll}
& $T=10$ \\
$\rho $ & $T=20$ \\
& $T=30$%
\end{tabular}
& \multicolumn{1}{|l|}{%
\begin{tabular}{l}
$0.8\pm 1.2\times 10^{-3}$ \\ \hline $1.3\pm 1.2\times 10^{-4}$ \\
\hline
$7.5\pm 5.7\times 10^{-5}$%
\end{tabular}%
} & \multicolumn{1}{|l|}{%
\begin{tabular}{l}
$1.9\pm 2.6\times 10^{-4}$ \\ \hline $3.7\pm 3.0\times 10^{-5}$ \\
\hline
$2.5\pm 1.2\times 10^{-5}$%
\end{tabular}%
} & \multicolumn{1}{|l|}{%
\begin{tabular}{l}
$3.9\pm 5.0\times 10^{-5}$ \\ \hline $1.3\pm 0.5\times 10^{-5}$ \\
\hline
$1.1\pm 0.3\times 10^{-5}$%
\end{tabular}%
} & \multicolumn{1}{|l|}{%
\begin{tabular}{l}
$9.9\pm 8.9\times 10^{-6}$ \\ \hline $6.6\pm 2.0\times 10^{-6}$ \\
\hline
$6.0\pm 1.3\times 10^{-6}$%
\end{tabular}%
} \\ \hline
\end{tabular}

{\small Table IV: Confidence limits for the error between the exact
and the approximate innovation estimators of the equation (\ref{SE
EJ2}). }$h=\delta
${\small , for the conventional; and }$h=\delta /2,\delta /4,\delta /8,$%
{\small \ for the order-$1$ on }$\left( \tau \right)
_{h,T}^{u}${\small. \newline}

Figure 5 shows the histograms and the confidence limits for both,
the exact ( $\widehat{\alpha }_{\delta ,T}^{E}$) and the
conventional ($\widehat{\alpha } _{\delta ,T}$) innovation
estimators of $\alpha $ computed from the twelve sets of $100$ time
series $Z_{\delta ,T}^{i}$ available for the example 2. Figure 6
shows the same but, for the exact ($\widehat{\sigma }_{\delta
,T}^{E}$) and the conventional ($\widehat{\sigma }_{\delta ,T}$)
innovation estimators of $\sigma $, whereas Figure 7 does it for the
estimators $ \widehat{\rho }_{\delta ,T}^{E}$ and $\widehat{\rho
}_{\delta ,T}$ of $\rho $ . Note that, for this example, the
diffusion parameters $\sigma $ and $\rho $ can not be estimated from
the samples $Z_{\delta ,T}^{i}$ with the largest sampling period
$\delta =1$. From the other data with sampling period $\delta <1$,
the tree parameters can be estimated and, the bias of the exact and
the conventional innovation estimators is not so large as in the
previous example. Nevertheless, in this extreme situation of low
information in the data, the order-$1$ innovation estimators is able
to improve the accuracy of the parameter estimation when $h$
decreases. This is shown in Figure 8 for the samples $Z_{\delta
,T}^{i}$ with $\delta =0.1$ and $T=10,20,30$, and summarized in
Table III. The order-$1$ innovation estimators ($\widehat{ \alpha
}_{h,\delta ,T}^{u},\widehat{\sigma }_{h,\delta
,T}^{u},\widehat{\rho }_{h,\delta ,T}^{u}$) and ($\widehat{\alpha
}_{\cdot ,\delta ,T},\widehat{ \sigma }_{\cdot ,\delta
,T},\widehat{\rho }_{\cdot ,\delta ,T}$) are again computed on
uniform $\left( \tau \right)
_{h,T}^{u}\supset \{t\}_{T/\delta ,T}$ and adaptive $%
\left( \tau \right) _{\cdot ,T}\supset \{t\}_{T/\delta ,T}$ time
discretizations, respectively, with $T=10,20,30$, $h=\delta
/2,\delta /4,\delta /8$ and tolerances $rtol_{
\mathbf{y}}=rtol_{\mathbf{P}}=5\times 10^{-7}$ and $atol_{\mathbf{y}
}=5\times 10^{-10}$, $atol_{\mathbf{P}}=5\times 10^{-13}$. The
average of accepted and fail steps of the adaptive innovation
estimators at each $ t_{k}\in \{t\}_{T/\delta ,T}$ are shown in
Figure 4. Observe in Table III the higher difference between the
averages of the exact and the adaptive estimators for the three
parameters when $T=30$. The reason is that, for $ t_{k}>200$, the
mean and variance of the diffusion process (\ref{SE EJ2}) becomes
almost indistinguishable of zero in such a way that the signal noise
ratio is very small. This is so small that the adaptive strategy
computes inaccurate estimates of the integration errors for the
predictions and so less accurate estimators for the parameters of
the SDE (\ref{SE EJ2}). For this example, the convergence of the
order-$1$ innovation estimators to the exact one is shown in Table
IV, which gives the confidence limits for the error between theses
estimators for different values of $h$. Note that, Table III and IV
illustrate the convergence results of Theorems \ref{Inn week
convergence} and \ref{Innovation convergence theorem}, respectively.

\subsection{Simulations with two-dimensional state equations}

For the examples 3 and 4, $100$ realizations of the state equation
solution were similarly computed by means of the Local Linearization
and the Euler scheme, respectively. For each example, the
realizations where computed over the thin time partition
$\{t_{0}+10^{-4}n:n=0,..,30\times 10^{4}\}$ for guarantee a precise
simulation of the stochastic solutions on the time interval
$[t_{0},t_{0}+30]$. Two subsamples of each realization at the time
instants $\{t\}_{M,T}=\{t_{k}=t_{0}+kT/M:$ $k=0,..,M-1\}$ were taken
for evaluating the corresponding observation equation, with $T=30$
and two values of $M$. In particular, $M=30,300$ were used, which
correspond to the
sampling periods $\delta =1,0.1$. In this way, two sets of $100$ time series $%
Z_{\delta ,T}^{i}=\{z_{t_{k}}^{i}:t_{k}\in \{t\}_{M,T},$ $M=T/\delta
\}$, with $i=1,..,100$, of $M$ observations $z_{t_{k}}^{i}$ each one
were available for both state space models with the two values of
$(\delta ,T)$ mentioned above.

For both examples, the order-$1$ innovation estimators
($\widehat{\alpha } _{h,\delta ,T}^{u},\widehat{\sigma }_{h,\delta
,T}^{u}$) and ($\widehat{ \alpha }_{\cdot ,\delta
,T},\widehat{\sigma }_{\cdot ,\delta ,T}$) on uniform $\left( \tau
\right) _{h,T}^{u}\supset \{t\}_{T/\delta ,T}$ and adaptive $\left(
\tau \right) _{\cdot ,T}\supset \{t\}_{T/\delta ,T}$ time
discretizations, respectively, were computed from the two sets of
$100$ data $Z_{\delta ,T}^{i}$ with $T=30$ and $\delta =1,0.1$. The
values of $h$ were set as $h=\delta ,\delta /16,\delta /64$ for the
example 3, and as $h=\delta ,\delta /8,\delta /32$ for the example
4. The tolerances for the adaptive estimators were set as in the
first example. Figures 9 and 11 show the histograms and the
confidence limits for the estimators ($\widehat{ \alpha }_{h,\delta
,T}^{u},\widehat{\sigma }_{h,\delta ,T}^{u}$) and ($ \widehat{\alpha
}_{\cdot ,\delta ,T},\widehat{\sigma }_{\cdot ,\delta ,T}$)
corresponding to each example. For the two examples, the difference
between the order-$1$ innovation estimator ($\widehat{\alpha
}_{h,\delta ,T}^{u}, \widehat{\sigma }_{h,\delta ,T}^{u}$) and the
adaptive one ($\widehat{\alpha }_{\cdot ,\delta ,T},\widehat{\sigma
}_{\cdot ,\delta ,T}$) decreases when $h$ does it. This is,
according Theorem \ref{Innovation convergence theorem}, an expected
result by assuming that the difference between the adaptive and the
exact innovation estimators is negligible for $\left( \tau \right)
_{\cdot ,T}$ thin enough. In addition, Table V and VI show the bias
of the approximate innovation estimators for these examples. Observe
as the adaptive ($\widehat{\alpha }_{\cdot ,\delta
,T},\widehat{\sigma }_{\cdot ,\delta ,T}$) and the order-$1$
innovation estimator ($\widehat{\alpha } _{h,\delta
,T}^{u},\widehat{\sigma }_{h,\delta ,T}^{u}$) with $h<\delta $
provide much less biased estimation of the parameters $(\alpha
,\sigma )$ than the conventional innovation estimator
($\widehat{\alpha }_{\delta ,\delta ,T}^{u},\widehat{\sigma
}_{\delta ,\delta ,T}^{u}$), which is in fact unable to identify the
parameters of the examples. Clearly, this illustrates the usefulness
of the order-$1$ innovation estimator and its adaptive
implementation. However, as it is shown in Table V for $\delta =0.1$
, no always the adaptive estimator ($\widehat{\alpha }_{\cdot
,\delta ,T}, \widehat{\sigma }_{\cdot ,\delta ,T}$) is less unbiased
than the order-$1$ innovation estimator ($\widehat{\alpha
}_{h,\delta ,T}^{u},\widehat{\sigma } _{h,\delta ,T}^{u}$) for some
$h<\delta $. This can happen for one of following reasons: 1) the
bias of the exact innovation\ estimator when the adaptive estimator
is close enough to it, or 2) an insufficient number of accepted
steps of the adaptive estimator for a given tolerance. In our case,
since ($\widehat{\alpha }_{h,\delta ,T}^{u},\widehat{\sigma
}_{h,\delta ,T}^{u}$) converges to ($\widehat{\alpha }_{\cdot
,\delta ,T},\widehat{ \sigma }_{\cdot ,\delta ,T}$) as $h$ decreases
(Figure 9 with $\delta =0.1$) and the average of accepted steps of
the adaptive estimators is acceptable (Figure 10 with $\delta
=0.1$), the first explanation is more suitable. Figures 10 and 12
show the average of accepted and fail steps of the adaptive
estimators at each $t_{k}\in \{t\}_{T/\delta ,T}$ for each example.
Note how the average of accepted steps corresponding to the
estimators from samples with $\delta =0.1$ is ten time lower than
that of the estimators from samples with $\delta =1$, which is an
expected result as well.\newline \newline

\begin{tabular}{|c|c|c|}
\hline $T=30$ & $\alpha$ & $\sigma $ \\ \hline
\begin{tabular}{c}
$h$ \\ \hline $\delta $ \\ \hline $\delta /16$ \\ \hline $\delta
/64$ \\ \hline
$\cdot $%
\end{tabular}
&
\begin{tabular}{cc}
$\delta =1$ & $\delta =0.1$ \\ \hline \multicolumn{1}{r}{$-0.4588$}
& \multicolumn{1}{r}{$-0.1403$} \\ \hline
\multicolumn{1}{r}{$-0.1244$} & \multicolumn{1}{r}{$-0.0026$} \\
\hline \multicolumn{1}{r}{$-0.0336$} & \multicolumn{1}{r}{$0.0041$}
\\ \hline
\multicolumn{1}{r}{$-0.0108$} & \multicolumn{1}{r}{$0.0064$}%
\end{tabular}
&
\begin{tabular}{cc}
$\delta =1$ & $\delta =0.1$ \\ \hline $-0.7240$ &
\multicolumn{1}{r}{$-0.0140$} \\ \hline $-0.2180$ &
\multicolumn{1}{r}{$0.0103$} \\ \hline $-0.1883$ &
\multicolumn{1}{r}{$0.0104$} \\ \hline
$-0.1803$ & \multicolumn{1}{r}{$0.0099$}%
\end{tabular}
\\ \hline
\end{tabular}

{\small Table V: Bias of the approximate innovation estimators for
the equation (\ref{SEa EJ3})-(\ref{SEb EJ3}). }$h=\delta ${\small ,
for the
conventional; }$h=\delta /16,\delta /64,${\small \ for the order-$1$ on }$%
\left( \tau \right) _{h,T}^{u}${\small ; and }$h=\cdot ,${\small \
for the adaptive order-$1$ on }$\left( \tau \right) _{\cdot
,T}${\small .\newline \newline \newline }

\begin{tabular}{|c|c|c|}
\hline $T=30$ & $\alpha$ & $\sigma $ \\ \hline
\begin{tabular}{c}
$h$ \\ \hline $\delta $ \\ \hline $\delta /8$ \\ \hline $\delta /32$
\\ \hline
$\cdot $%
\end{tabular}
&
\begin{tabular}{cc}
$\delta =1$ & $\delta =0.1$ \\ \hline \multicolumn{1}{r}{$-0.8511$}
& \multicolumn{1}{r}{$-0.2740$} \\ \hline
\multicolumn{1}{r}{$-0.2488$} & \multicolumn{1}{r}{$-0.0662$} \\
\hline \multicolumn{1}{r}{$-0.1887$} & \multicolumn{1}{r}{$-0.0472$}
\\ \hline
\multicolumn{1}{r}{$-0.1550$} & \multicolumn{1}{r}{$-0.0373$}%
\end{tabular}
&
\begin{tabular}{cc}
$\delta =1$ & $\delta =0.1$ \\ \hline $-1.0347$ &
\multicolumn{1}{r}{$-0.0239$} \\ \hline $-0.3107$ &
\multicolumn{1}{r}{$0.0071$} \\ \hline $-0.2857$ &
\multicolumn{1}{r}{$0.0072$} \\ \hline
$-0.2805$ & \multicolumn{1}{r}{$0.0084$}%
\end{tabular}
\\ \hline
\end{tabular}

{\small Table VI: Bias of the approximate innovation estimators for
the equation (\ref{SEa EJ4})-(\ref{SEb EJ4}). }$h=\delta ${\small ,
for the
conventional; }$h=\delta /8,\delta /32,${\small \ for the order-$1$ on }$%
\left( \tau \right) _{h,T}^{u}${\small ; and }$h=\cdot ,${\small \
for the adaptive order-$1$ on }$\left( \tau \right) _{\cdot
,T}${\small. \newline}

\subsection{Simulations with noise free observation equations}

In section \ref{Section on complete observations}, the connection
among the innovation and quasi-maximum likelihood estimators was
early mentioned for the identification of models with noise free
complete observations. In this
situation, it is easy to verify that the LL-based innovation estimator (\ref%
{LL-based innovation estimator}) reduces to the LL-based
quasi-maximum likelihood estimator introduced in Jimenez (2012b). In
that paper, the state equations of the four models considered in
Section \ref{Section Test models} were also used as test examples in
simulations. The reader interested in this identification problem is
encouraged to consider these simulations.

\section{Conclusions}

An alternative approximation to the innovation method was introduced
for the parameter estimation of diffusion processes given a time
series of partial and noisy observations. This is based on a
convergent approximation to the first two conditional moments of the
innovation process through approximate continuous-discrete filters
of minimum variance. For finite samples, the convergence of the
approximate innovation estimators to the exact one was proved when
the error between the approximate and the exact linear minimum
variance filters decreases. It was also demonstrated that, for an
increasing number of observations, the approximate estimators are
asymptotically normal distributed and their bias decreases when the
above mentioned error does it. As particular instance, the
order-$\beta $ innovation estimators based on Local Linearization
filters were proposed. For them, practical algorithms were also
provided and their performance in simulation illustrated with
various examples. Simulations shown that: 1) with thin time
discretizations between observations, the order-$1$ innovation
estimator provides satisfactory approximations to the exact
innovation estimator; 2) the convergence of the order-$1$ innovation
estimator to the exact one when the maximum stepsize of the time
discretization between observations decreases; 3) with respect to
the conventional innovation estimator, the order-$1$ innovation
estimator gives much better approximation to the exact innovation
estimator, and has less bias and higher efficiency; 4) with an
adequate tolerance, the adaptive order-$1$ innovation estimator
provides an automatic, suitable and computational efficient
approximation to the exact innovation estimator; and 5) the
effectiveness of the order-$1$ innovation estimator for the
identification of SDEs from a reduced number of partial and noisy
observations distant in time. Further note that new estimators can
also be easily applied to a variety of practical problems with
sequential random measurements or with multiple missing data.
\newline

\textbf{Acknowledgement. }The numerical simulations of this paper
were concluded on July 2012 within the framework of the
Associateship Scheme of the Abdus Salam International Centre for
Theoretical Physics (ICTP), Trieste, Italy. The author thanks to the
ICTP for the partial support to this work.

\section{Appendix}

According\ to Jimenez (2012b), given the filters values $\mathbf{y}%
_{t_{k}/t_{k}}$ and $\mathbf{P}_{t_{k}/t_{k}}$, the predictions $\mathbf{y}%
_{t/t_{k}}$ and $\mathbf{P}_{t/t_{k}}$ of the order-$\beta $ LL
filter are
computed by the recursive formulas%
\begin{equation}
\mathbf{y}_{t/t_{k}}=\mathbf{y}_{\tau _{n_{t}}/t_{k}}+\mathbf{L}_{2}e^{%
\mathbf{M}(\tau _{n_{t}})(t-\tau _{n_{t}})}\mathbf{u}_{\tau
_{n_{t}},t_{k}} \label{ALLF8}
\end{equation}%
and
\begin{equation}
vec(\mathbf{P}_{t/t_{k}})=\mathbf{L}_{1}e^{\mathbf{M}(\tau
_{n_{t}})(t-\tau _{n_{t}})}\mathbf{u}_{\tau _{n_{t}},t_{k}}
\label{ALLF9}
\end{equation}%
for all $t\in (t_{k},t_{k+1}]$ and $t_{k},t_{k+1}\in \{t\}_{M}$,
where
\begin{equation*}
n_{t}=\max \{n=0,1,\ldots :\tau _{n}\leq t\text{ and }\tau _{n}\in
\left( \tau \right) _{h}\},
\end{equation*}%
and the vector $\mathbf{u}_{\tau ,t_{k}}$ and the matrices
$\mathbf{M}(\tau ) $, $\mathbf{L}_{1}$, $\mathbf{L}_{2}$\ are
defined as
\begin{equation*}
\mathbf{M}(\tau )=\left[
\begin{array}{cccccc}
\mathcal{A}(\tau ) & \mathcal{B}_{5}(\tau ) & \mathcal{B}_{4}(\tau )
&
\mathcal{B}_{3}(\tau ) & \mathcal{B}_{2}(\tau ) & \mathcal{B}_{1}(\tau ) \\
\mathbf{0} & \mathbf{C}(\tau ) & \mathbf{I}_{d+2} & \mathbf{0} &
\mathbf{0}
& \mathbf{0} \\
\mathbf{0} & \mathbf{0} & \mathbf{C}(\tau ) & \mathbf{0} &
\mathbf{0} &
\mathbf{0} \\
\mathbf{0} & \mathbf{0} & \mathbf{0} & 0 & 2 & 0 \\
\mathbf{0} & \mathbf{0} & \mathbf{0} & 0 & 0 & 1 \\
\mathbf{0} & \mathbf{0} & \mathbf{0} & 0 & 0 & 0%
\end{array}%
\right] \text{, \ \ }\mathbf{u}_{\tau ,t_{k}}=\left[
\begin{array}{c}
vec(\mathbf{P}_{\tau /t_{k}}) \\
\mathbf{0} \\
\mathbf{r} \\
0 \\
0 \\
1%
\end{array}%
\right] \in
%TCIMACRO{\U{211d} }%
%BeginExpansion
\mathbb{R}
%EndExpansion
^{(d^{2}+2d+7)}
\end{equation*}%
and%
\begin{equation*}
\mathbf{L}_{1}=\left[
\begin{array}{cc}
\mathbf{I}_{d^{2}} & \mathbf{0}_{d^{2}\times (2d+7)}%
\end{array}%
\right] \text{, \ \ \ \ \ \ \ \ \ }\mathbf{L}_{2}=\left[
\begin{array}{ccc}
\mathbf{0}_{d\times (d^{2}+d+2)} & \mathbf{I}_{d} & \mathbf{0}_{d\times 5}%
\end{array}%
\right]
\end{equation*}%
in terms of the matrices and vectors
\begin{equation*}
\mathcal{A}(\tau )=\mathbf{A}(\tau )\mathbf{\oplus A}(\tau
)+\sum\limits_{i=1}^{m}\mathbf{B}_{i}(\tau )\mathbf{\otimes B}%
_{i}^{\intercal }(\tau ),
\end{equation*}%
\begin{equation*}
\mathbf{C(}\tau )=\left[
\begin{array}{ccc}
\mathbf{A}(\tau ) & \mathbf{a}_{1}(\tau ) & \mathbf{A}(\tau )\mathbf{y}%
_{\tau ,t_{k}}+\mathbf{a}_{0}(\tau ) \\
0 & 0 & 1 \\
0 & 0 & 0%
\end{array}%
\right] \in \mathbb{R}^{(d+2)\times (d+2)},
\end{equation*}%
\begin{equation*}
\mathbf{r}^{\intercal }=\left[
\begin{array}{ll}
\mathbf{0}_{1\times (d+1)} & 1%
\end{array}%
\right]
\end{equation*}%
$\mathcal{B}_{1}(\tau )=vec(\mathbf{\beta }_{1}(\tau ))+\beta _{4}(\tau )%
\mathbf{y}_{\tau /t_{k}}$, $\mathcal{B}_{2}(\tau )=vec(\mathbf{\beta }%
_{2}(\tau ))+\mathbf{\beta }_{5}(\tau )\mathbf{y}_{\tau /t_{k}}$, $\mathcal{B%
}_{3}(\tau )=vec(\mathbf{\beta }_{3}(\tau ))$, $\mathcal{B}_{4}(\tau )=%
\mathbf{\beta }_{4}(\tau )\mathbf{L}$ and $\mathcal{B}_{5}(\tau )=\mathbf{%
\beta }_{5}(\tau )\mathbf{L}$ with
\begin{align*}
\mathbf{\beta }_{1}(\tau )& =\sum\limits_{i=1}^{m}\mathbf{b}_{i,0}(\tau )%
\mathbf{b}_{i,0}^{\intercal }(\tau ) \\
\mathbf{\beta }_{2}(\tau )& =\sum\limits_{i=1}^{m}\mathbf{b}_{i,0}(\tau )%
\mathbf{b}_{i,1}^{\intercal }(\tau )+\mathbf{b}_{i,1}(\tau )\mathbf{b}%
_{i,0}^{\intercal }(\tau ) \\
\mathbf{\beta }_{3}(\tau )& =\sum\limits_{i=1}^{m}\mathbf{b}_{i,1}(\tau )%
\mathbf{b}_{i,1}^{\intercal }(\tau ) \\
\mathbf{\beta }_{4}(\tau )& =\mathbf{a}_{0}(\tau )\oplus
\mathbf{a}_{0}(\tau
)+\sum\limits_{i=1}^{m}\mathbf{b}_{i,0}(\tau )\otimes \mathbf{B}_{i}(\tau )+%
\mathbf{B}_{i}(\tau )\otimes \mathbf{b}_{i,0}(\tau ) \\
\mathbf{\beta }_{5}(\tau )& =\mathbf{a}_{1}(\tau )\oplus
\mathbf{a}_{1}(\tau
)+\sum\limits_{i=1}^{m}\mathbf{b}_{i,1}(\tau )\otimes \mathbf{B}_{i}(\tau )+%
\mathbf{B}_{i}(\tau )\otimes \mathbf{b}_{i,1}(\tau ),
\end{align*}%
$\mathbf{L}=\left[
\begin{array}{ll}
\mathbf{I}_{d} & \mathbf{0}_{d\times 2}%
\end{array}%
\right] $, and the $d$-dimensional identity matrix $\mathbf{I}_{d}$.
Here,
\begin{equation*}
\mathbf{A}(\tau )=\frac{\partial \mathbf{f}(\tau ,\mathbf{y}_{\tau /t_{k}})}{%
\partial \mathbf{y}}\text{ \ \ \ \ \ \ and \ \ \ \ \ \ }\mathbf{B}_{i}(\tau
)=\frac{\partial \mathbf{g}_{i}(\tau ,\mathbf{y}_{\tau
/t_{k}})}{\partial \mathbf{y}}
\end{equation*}%
are matrices, and the vectors $\mathbf{a}_{0}(\tau _{n_{t}})$, $\mathbf{a}%
_{1}(\tau _{n_{t}})$, $\mathbf{b}_{i,0}(\tau _{n_{t}})$ and $\mathbf{b}%
_{i,1}(\tau _{n_{t}})$ satisfy the expressions%
\begin{equation*}
\mathbf{a}^{\beta }(t;\tau _{n_{t}})=\mathbf{a}_{0}(\tau _{n_{t}})+\mathbf{a}%
_{1}(\tau _{n_{t}})(t-\tau _{n_{t}})\text{ \ \ \ \ and \ \ \ \ }\mathbf{b}%
_{i}^{\beta }(t;\tau _{n_{t}})=\mathbf{b}_{i,0}(\tau _{n_{t}})+\mathbf{b}%
_{i,1}(\tau _{n_{t}})(t-\tau _{n_{t}})
\end{equation*}%
for all $t\in \lbrack t_{k},t_{k+1}]$ and $\tau _{n_{t}}\in \left(
\tau \right) _{h}$, where
\begin{equation*}
\mathbf{a}^{\beta }(t;\tau )=\left\{
\begin{array}{ll}
\mathbf{f}(\tau ,\mathbf{y}_{\tau /t_{k}})-\frac{\partial \mathbf{f}(\tau ,%
\mathbf{y}_{\tau /t_{k}})}{\partial \mathbf{y}}\mathbf{y}_{\tau /t_{k}}+%
\frac{\partial \mathbf{f}(\tau ,\mathbf{y}_{\tau /t_{k}})}{\partial \tau }%
(t-\tau ) & \text{for }\mathbb{\beta }=1 \\
\mathbf{a}^{1}(t;\tau )+\frac{1}{2}\sum\limits_{j,l=1}^{d}[\mathbf{G}(\tau ,%
\mathbf{y}_{\tau /t_{k}})\mathbf{G}^{\intercal }(\tau
,\mathbf{y}_{\tau /t_{k}})]^{j,l}\text{ }\frac{\partial
^{2}\mathbf{f}(\tau ,\mathbf{y}_{\tau
/t_{k}})}{\partial \mathbf{y}^{j}\partial \mathbf{y}^{l}}(t-\tau ) & \text{%
for }\mathbb{\beta }=2%
\end{array}%
\right.
\end{equation*}%
and%
\begin{equation*}
\mathbf{b}_{i}^{\beta }(t;\tau )=\left\{
\begin{array}{ll}
\mathbf{g}_{i}(\tau ,\mathbf{y}(\tau ))-\frac{\partial \mathbf{g}_{i}(\tau ,%
\mathbf{y}_{\tau /t_{k}})}{\partial \mathbf{y}}\mathbf{y}_{\tau /t_{k}}+%
\frac{\partial \mathbf{g}_{i}(\tau ,\mathbf{y}_{\tau
/t_{k}})}{\partial \tau
}(t-\tau ) & \text{for }\mathbb{\beta }=1 \\
\mathbf{b}_{i}^{1}(t;\tau )+\frac{1}{2}\sum\limits_{j,l=1}^{d}[\mathbf{G}%
(\tau ,\mathbf{y}_{\tau /t_{k}})\mathbf{G}^{\intercal }(\tau \mathbf{,y}%
_{\tau /t_{k}})]^{j,l}\text{ }\frac{\partial ^{2}\mathbf{g}_{i}(\tau ,%
\mathbf{y}(\tau ))}{\partial \mathbf{y}^{j}\partial
\mathbf{y}^{l}}(t-\tau )
& \text{for }\mathbb{\beta }=2%
\end{array}%
\right.
\end{equation*}%
are functions associated to the order-$\beta $ Ito-Taylor expansions
for the drift and diffusion coefficients of (\ref{SS1}) in the
neighborhood of $(\tau \mathbf{,y}_{\tau /t_{k}})$, respectively,
and $\mathbf{G=[g}_{1},\ldots ,\mathbf{g}_{m}]$ is an $d\times m$
matrix function. The symbols $vec$, $\oplus $ and $\otimes $ denote
the vectorization operator, the Kronecker sum and product,
respectively.

From computational viewpoint, each evaluation of\ the formulas (\ref{ALLF8}%
)-(\ref{ALLF9}) at $\tau _{n}$ requires the computation of just one
exponential matrix whose matrix depends of the drift and diffusion
coefficients of (\ref{SS1}) at $(\tau _{n-1},\mathbf{y}_{\tau _{n-1}/t_{k}})$%
. This exponential matrix can the efficiently computed through the
well known Pad\'{e} method (Moler \& Van Loan, 2003) or,
alternatively, by means of the Krylov subspace method (Moler \& Van
Loan, 2003) in the case of high dimensional SDEs. Even more, low
order Pad\'{e} and Krylov methods as suggested in Jimenez \& de la
Cruz (2012) can be used as well for reducing the computation cost,
but preserving the order-$\beta $ of the approximate moments.
Alternatively, simplified formulas for the moments can be used when
the equation to be estimate is autonomous or has additive noise (see
Jimenez, 2012a). All this makes simple and efficient the evaluation
of the
approximate moments $\mathbf{y}_{t_{k+1}/t_{k}}$ and $\mathbf{V}%
_{t_{k+1}/t_{k}}$ required by the innovation estimator
(\ref{LL-based innovation estimator}).

\section{References}

\textbf{Bollerslev T. and Wooldridge J.M.} (1992) Quasi-maximun
likelihood estimation and inference in dynamic models with
time-varying covariances. Econom. Rev., 11, 143-172.

\textbf{Calderon C.P., Harris N.C. and Kiang C.H. and Cox D.D.}
(2009). Analyzing single-molecule manipulation experiments, Journal
of Molecular Recognition, 22, 356-362.

\textbf{Chiarella C., Hung H. and To T.D.} (2009). The volatility
structure of the fixed income market under the HJM framework: A
nonlinear filtering approach. Comput. Stat. Data Anal., 53,
2075-2088.

\textbf{Gitterman M.}{\normalsize , The noisy oscillator, World
Scientific, 2005. }

\textbf{Jimenez J.C.} (2012a) Simplified formulas for the mean and
variance of linear stochastic differential equations. Submitted.
http://arxiv.org/abs/1207.5067.

\textbf{Jimenez J.C.} (2012b) Approximate linear minimum variance
filters for continuous-discrete state space models: convergence and
practical algorithms. Submitted. http://arxiv.org/abs/1207.6023.

\textbf{Jimenez J.C.} (2012c) Approximate discrete-time schemes for
the estimation of diffusion processes from complete observations.
Submitted. http://arxiv.org/abs/1212.1788.

\textbf{Jimenez J.C., Biscay R. and Ozaki T.} (2006) Inference
methods for discretely observed continuous-time stochastic
volatility models: A commented overview, Asia-Pacific Financial
Markets, 12, 109-141.

\textbf{Jimenez J.C. and de la Cruz H.} (2012) Convergence rate of
strong Local Linearization schemes for stochastic differential
equations with additive noise, BIT, 52, 357-382.

\textbf{Jimenez J.C., Shoji I. and Ozaki T.} (1999) Simulation of
stochastic differential equations through the Local Linearization
method. A comparative study, J. Statist. Physics, 94, 587-602.

\textbf{Jimenez, J.C. and Ozaki, T.} (2006) An approximate
innovation method for the estimation of diffusion processes from
discrete data, J. Time Series Analysis, 27, 77-97.

\textbf{Kloeden P.E. and Platen E}. (1999) Numerical Solution of
Stochastic Differential Equations, Springer-Verlag, Berlin, Third
Edition.

\textbf{Ljung L. and Caines P.E.} (1979) Asymptotic normality of
prediction error estimators for approximate system models,
Stochastics, 3, 29-46.

\textbf{Moler C. and Van Loan C.} (2003) Nineteen dubious ways to
compute the exponential of a matrix, SIAM Review, 45, 3-49.

\textbf{Nielsen J.N. and Madsen H.} (2001) Applying the EKF to
stochastic differential equations with level effects, Automatica,
37, 107-112.

\textbf{Nielsen J.N., Madsen H. and Young, P. C.} (2000a) Parameter
estimation in stochastic differential equations: an overview. Annual
Review of Control, 24, 83-94.

\textbf{Nielsen J.N., Vestergaard M. and Madsen H.} (2000b)
Estimation in continuous-time stochastic volatility models using
nonlinear filters, Int. J. Theor. Appl. Finance, 3, 279-308.

\textbf{Nolsoe K., Nielsen J.N. and Madsen H.} (2000)
Prediction-based estimating function for diffusion processes with
measurement noise. Technical Reports 2000, No.10, Informatics and
Mathematical Modelling, Technical University of Denmark.

\textbf{Oppenheim A.V. and Schafer R.W.} (2010) Discrete-Time Signal
Processing, Prentice Hall, Third Edition.

\textbf{Ozaki T.} (1994) The local linearization filter with
application to nonlinear system identification. In Bozdogan H. (ed.)
Proceedings of the first US/Japan Conference on the Frontiers of
Statistical Modeling: An Informational Approach, 217-240. Kluwer
Academic Publishers.

\textbf{Ozaki T. and Iino M.} (2001) An innovation approach to
non-Gaussian time series analysis, J. Appl. Prob., 38A, 78-92.

\textbf{Peng H., Ozaki T. and Jimenez J.C.} (2002) Modeling and
control for foreign exchange based on a continuous time stochastic
microstructure model, in Proceedings of the 41st IEEE Conference on
Decision and Control, Las Vegas, Nevada USA, December 2002,
4440-4445.

\textbf{Riera J.J., Watanabe J., Iwata K., Miura N., Aubert E.,
Ozaki T. and Kawashima R.} (2004) A state-space model of the
hemodynamic approach: nonlinear filtering of BOLD signals.
Neuroimage, 21, 547-567.

\textbf{Schweppe F.} (1965) Evaluation of likelihood function for
Gaussian signals, IEEE Trans. Inf. Theory, 11, 61-70.

\textbf{Shoji I.} (1998) A comparative study of maximum likelihood
estimators for nonlinear dynamical systems, Int. J. Control, 71,
391-404.

\textbf{Singer H.} (2002) Parameter estimation of nonlinear
stochastic differential equations: Simulated maximum likelihood
versus extended Kalman filter and Ito-Taylor expansion, J. Comput.
Graphical Statist.11, 972-995.

\textbf{Valdes P.A., Jimenez J.C., Riera J., Biscay R. and Ozaki T.}
(1999) Nonlinear EEG analysis based on a neural mass model. Biol.
Cyb., 81, 415-424.

\newpage

%% Figure 1
\begin{figure}
\centering
 $\begin{array}{c}
  \includegraphics[width=5in]{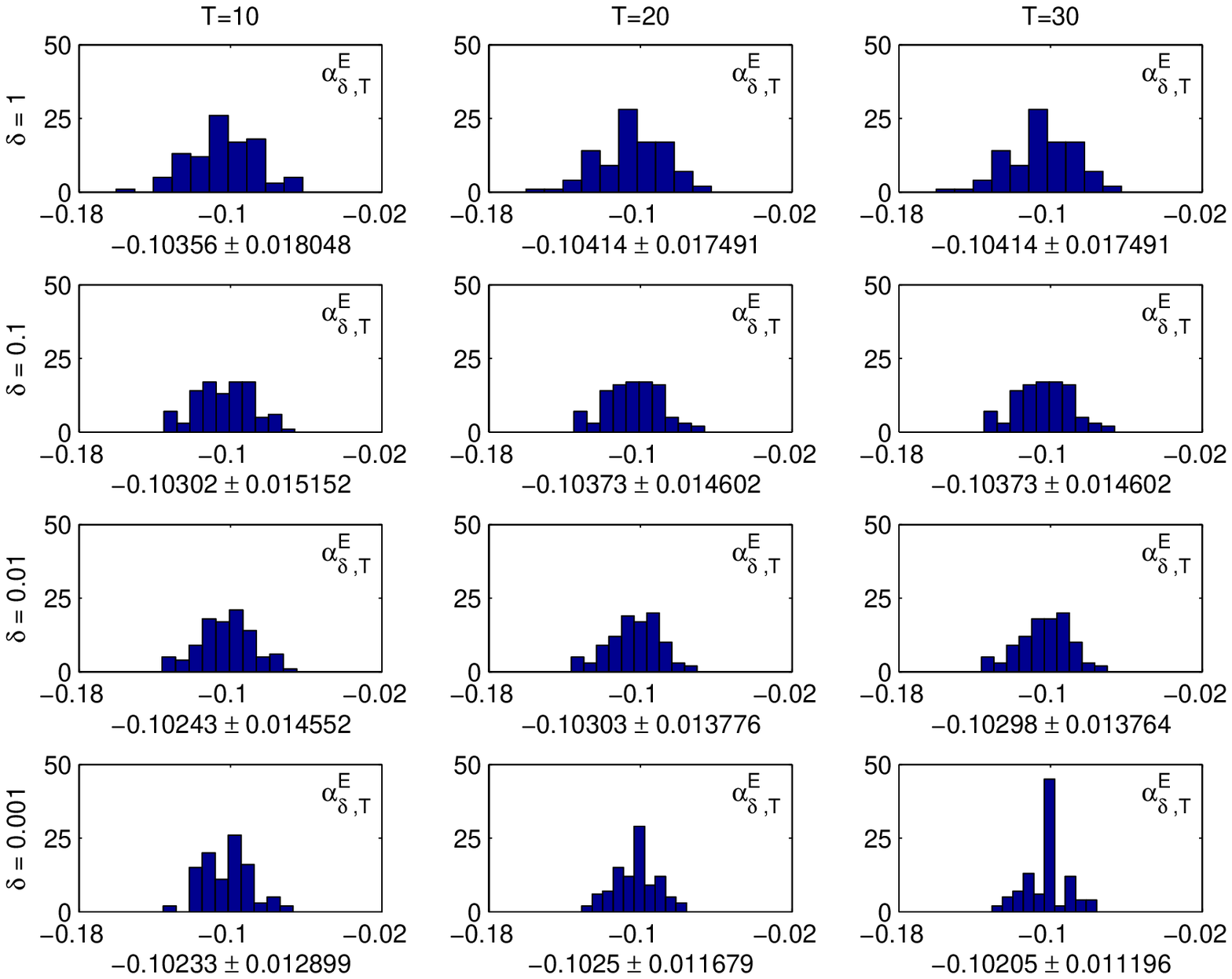} \\
  \includegraphics[width=5in]{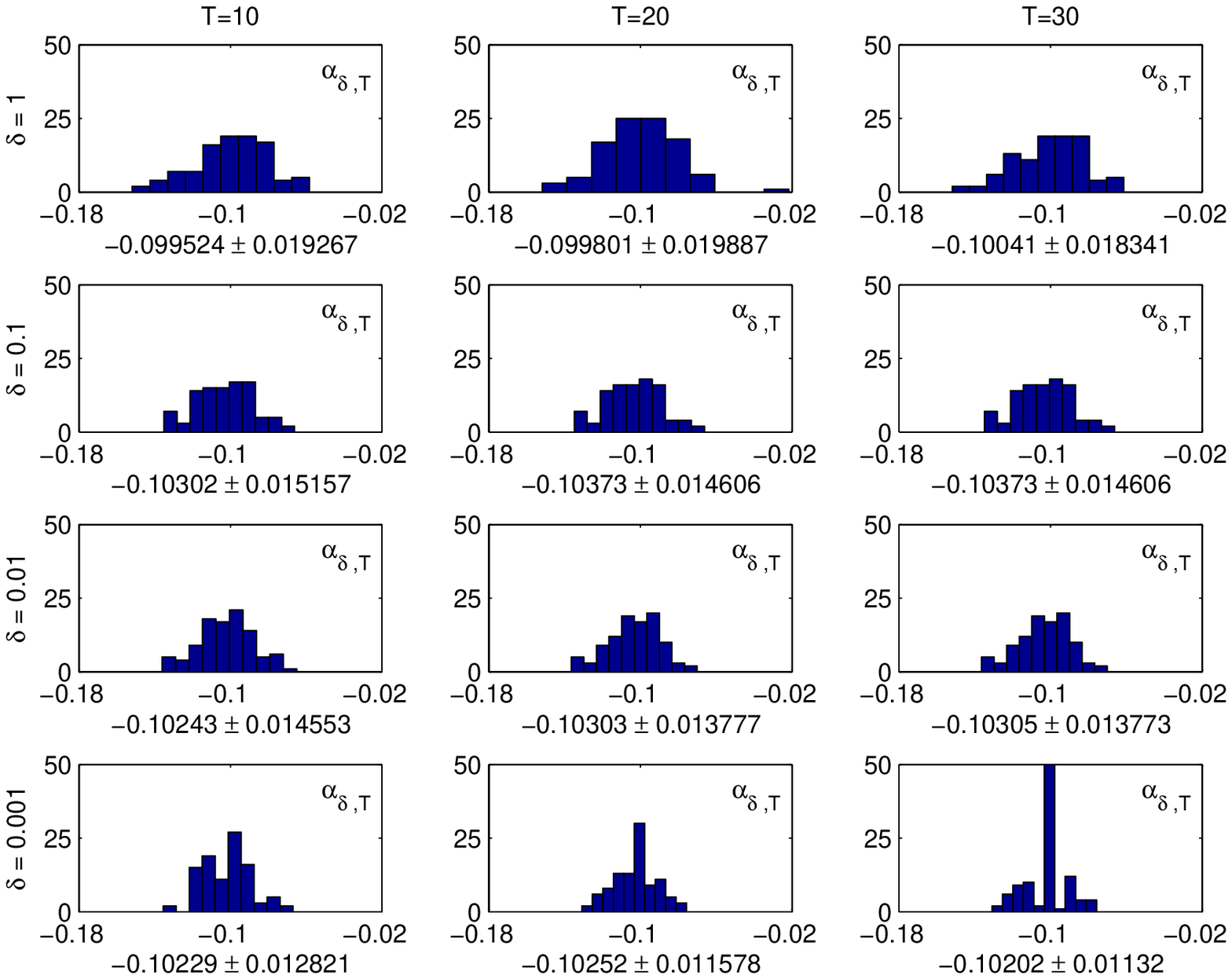}
\end{array}$
\caption{Histograms and confidence limits for the exact ($\widehat{\protect\alpha }_{\protect\delta %
,T}^{E}$) and the conventional ($\widehat{\protect\alpha }_{\protect\delta %
,T}$) innovation estimators of $\protect\alpha $ computed from the
Example 1 data with sampling period $\protect\delta $ and time
interval of length $T$.}
\end{figure}

%% Figure 2
\begin{figure}
\centering
 $\begin{array}{c}
  \includegraphics[width=5in]{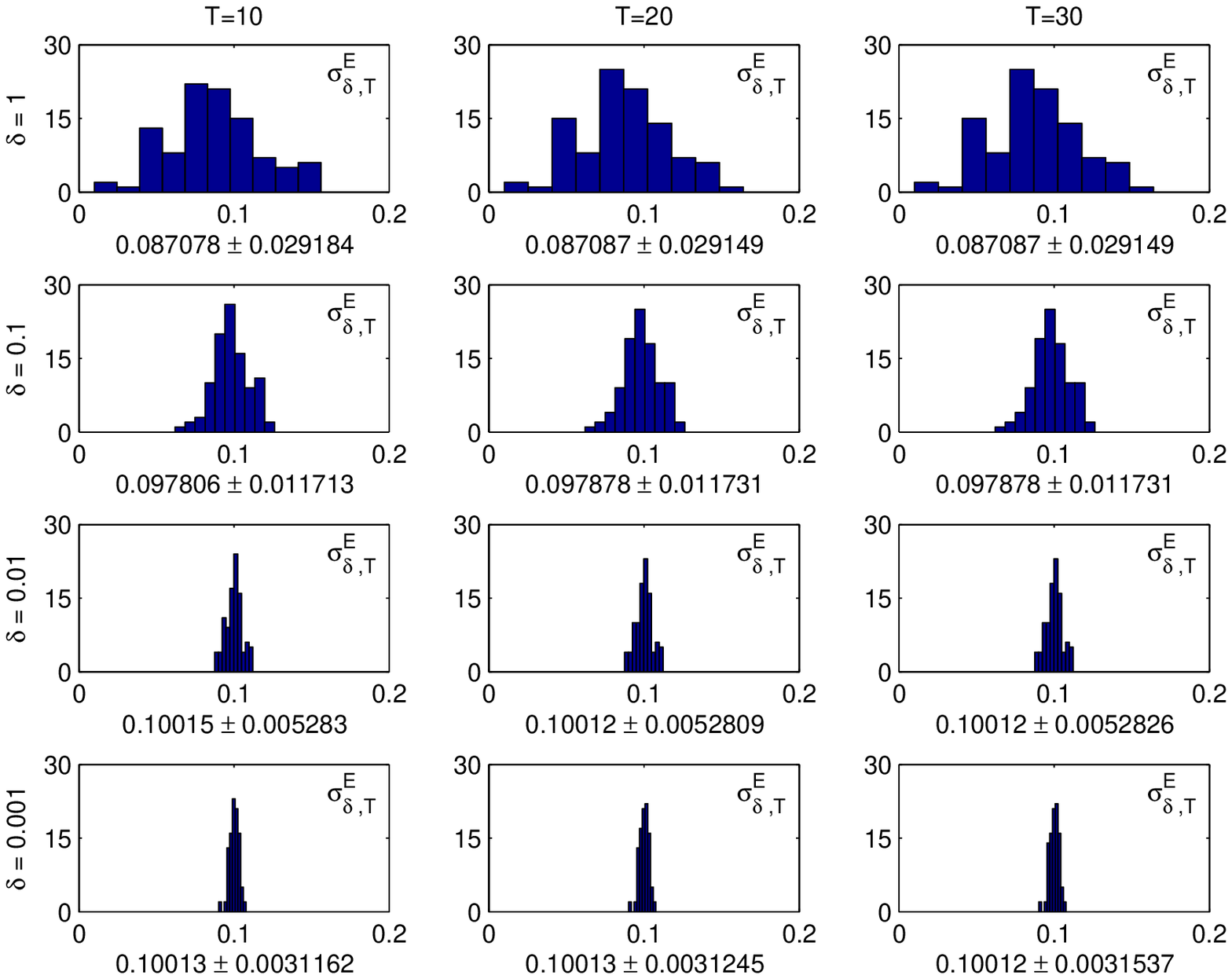} \\
  \includegraphics[width=5in]{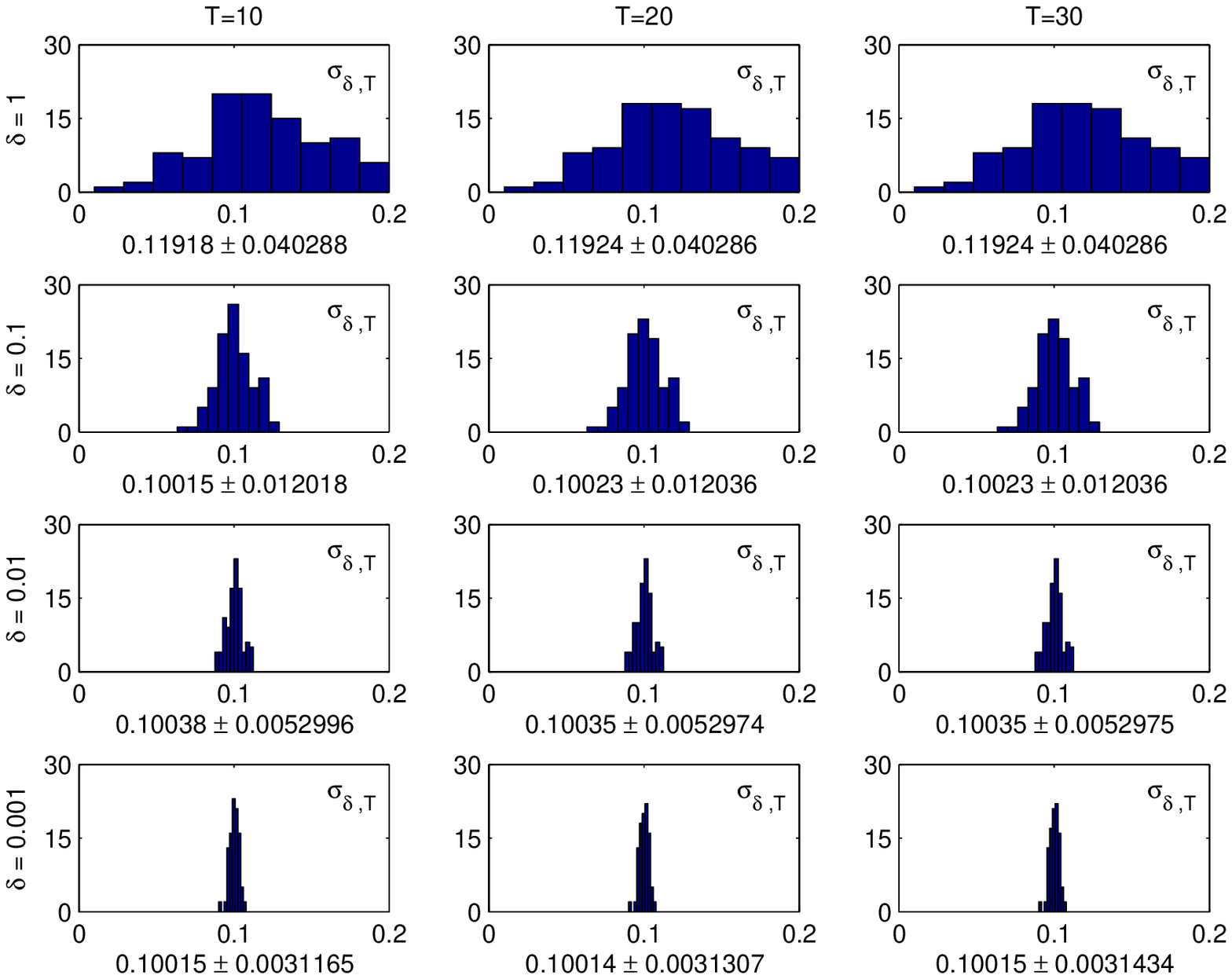}
\end{array}$
\caption{Histograms and confidence limits for the exact ($\widehat{\protect\sigma }_{\protect\delta %
,T}^{E}$) and the conventional ($\widehat{\protect\sigma }_{\protect\delta %
,T}$) innovation estimators of $\protect\sigma $ computed from the
Example 1 data with sampling period $\protect\delta $ and time
interval of length $T$.}
\end{figure}

%% Figure 3
\begin{figure}
\centering
 $\begin{array}{c}
  \includegraphics[width=5in]{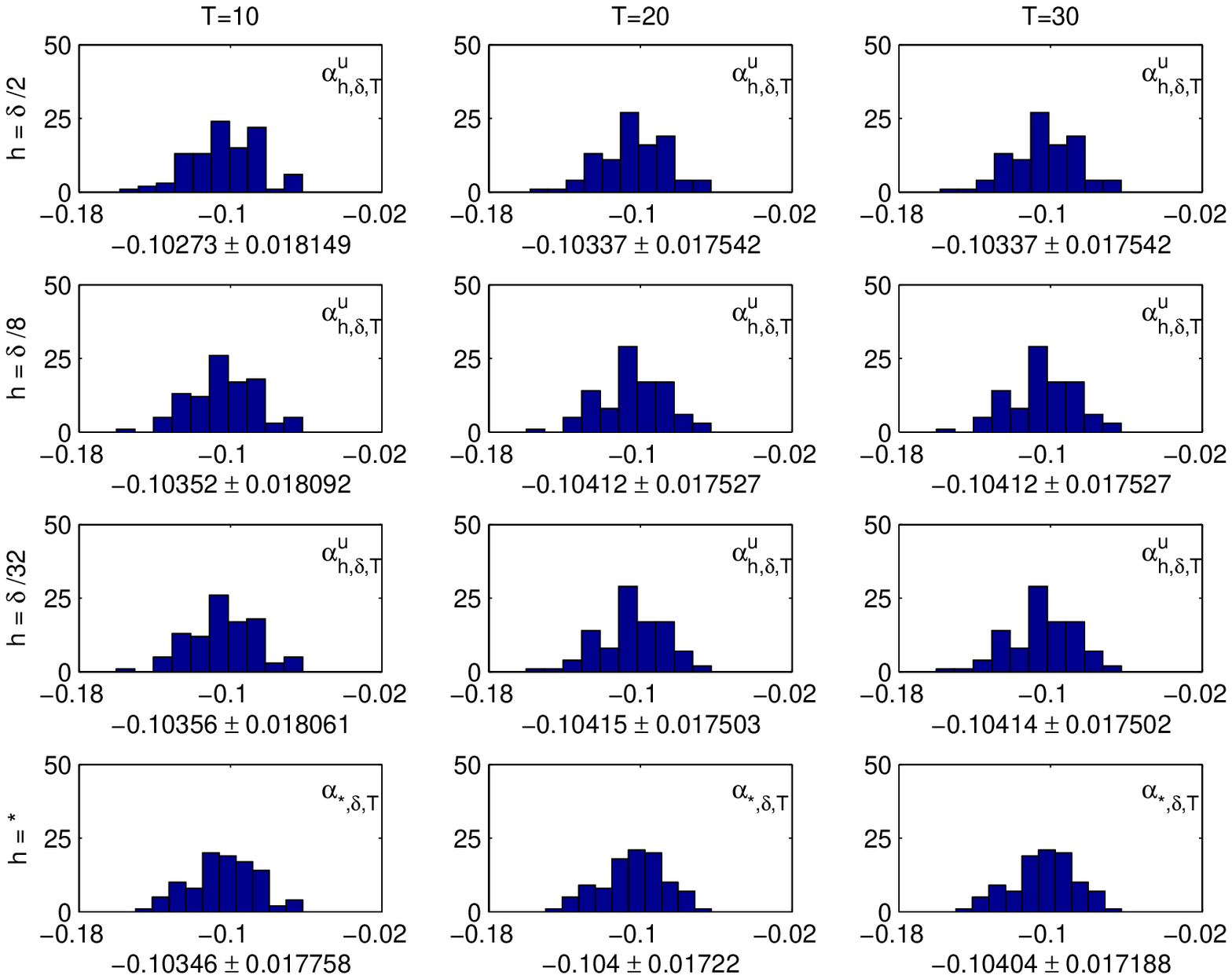} \\
  \includegraphics[width=5in]{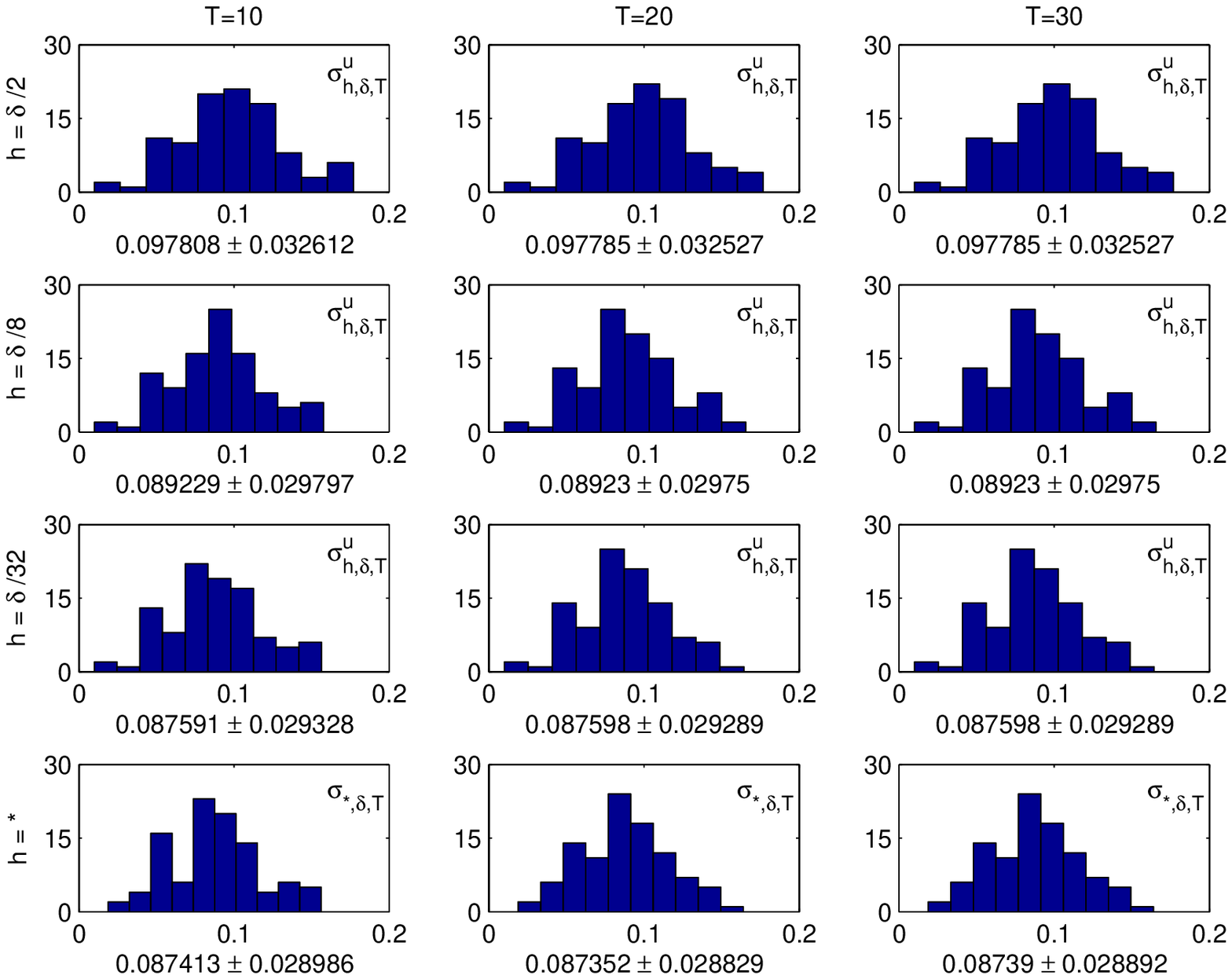}
\end{array}$
\caption{Histograms and confidence limits for the oder-1 innovation estimators of $\protect\alpha $ and $%
\protect\sigma $\ computed on uniform $\left( \protect\tau \right)
_{h,T}^{u}$ and adaptive $\left( \protect\tau \right) _{\cdot ,T}$
time discretizations from the Example 1 data with sampling period
$\protect\delta =1$ and time interval of length $T$.}
\end{figure}

%% Figure 4
\begin{figure}
\centering
 $\begin{array}{c}
  \includegraphics[width=5in]{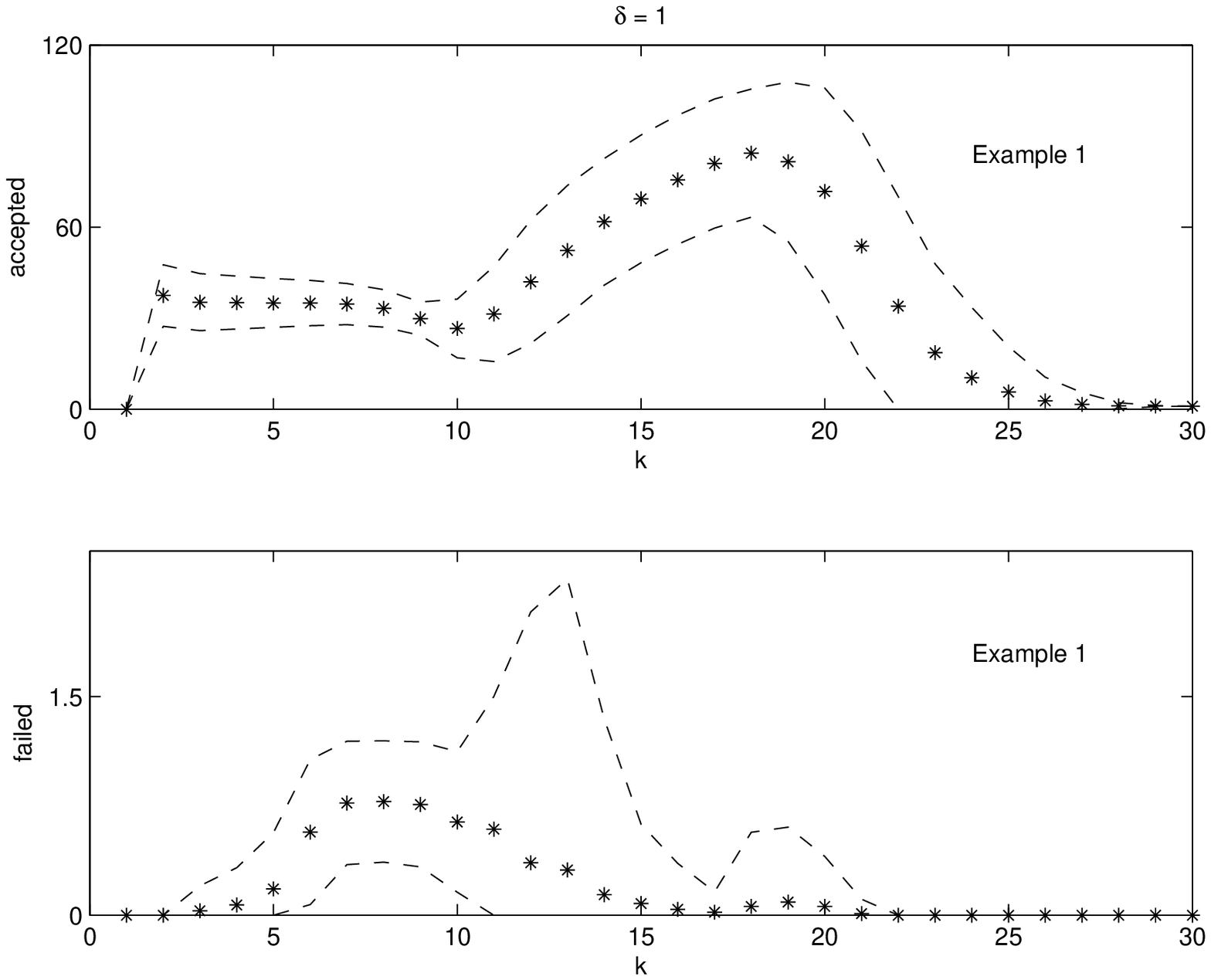} \\
  \includegraphics[width=5in]{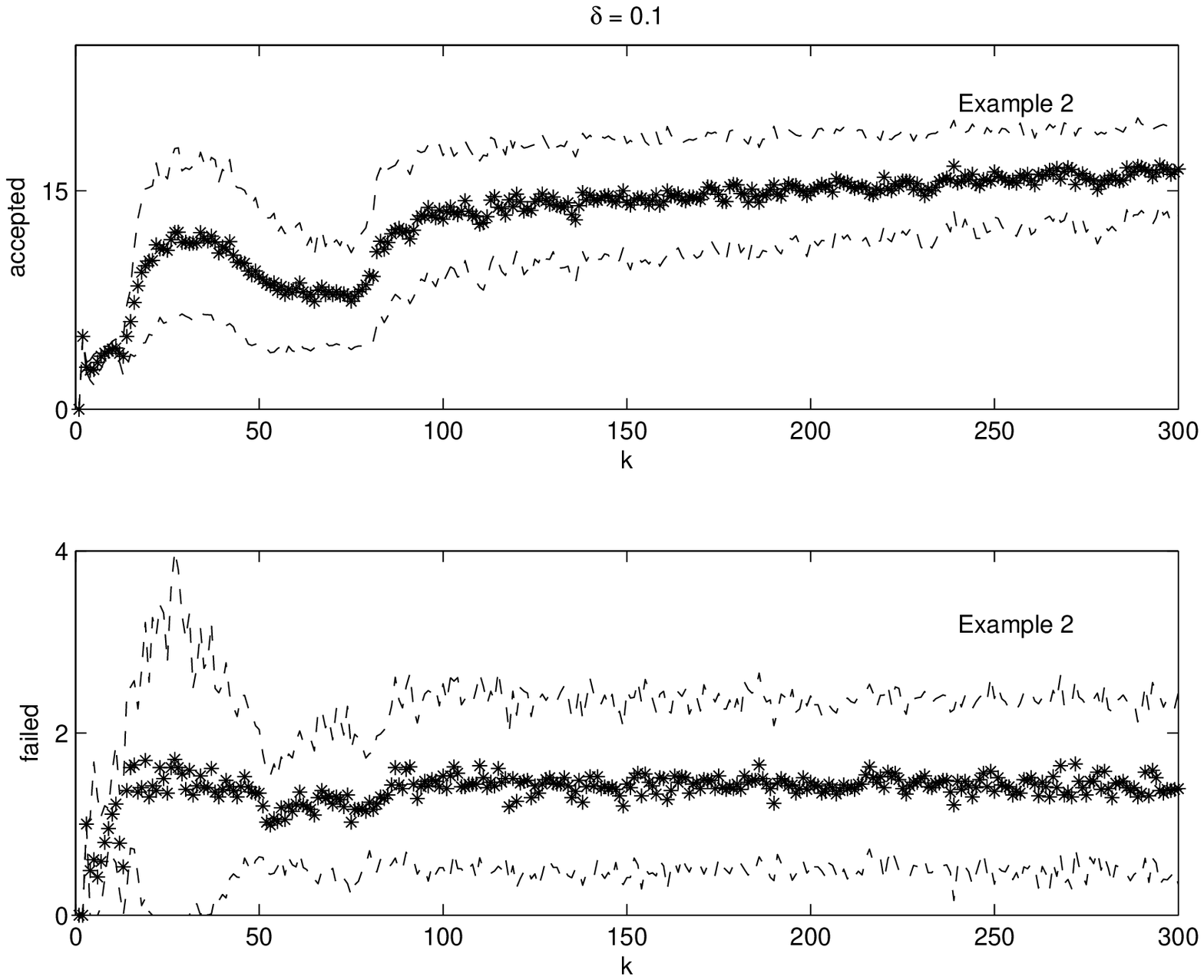}
\end{array}$
\caption{Average (*) and 90\% confidence limits (-) of accepted and
failed steps of the adaptive innovation estimator at each $t_{k}\in
\{t\}_{N}$ in the Examples 1 and 2.}
\end{figure}

%% Figure 5
\begin{figure}
\centering
 $\begin{array}{c}
  \includegraphics[width=5in]{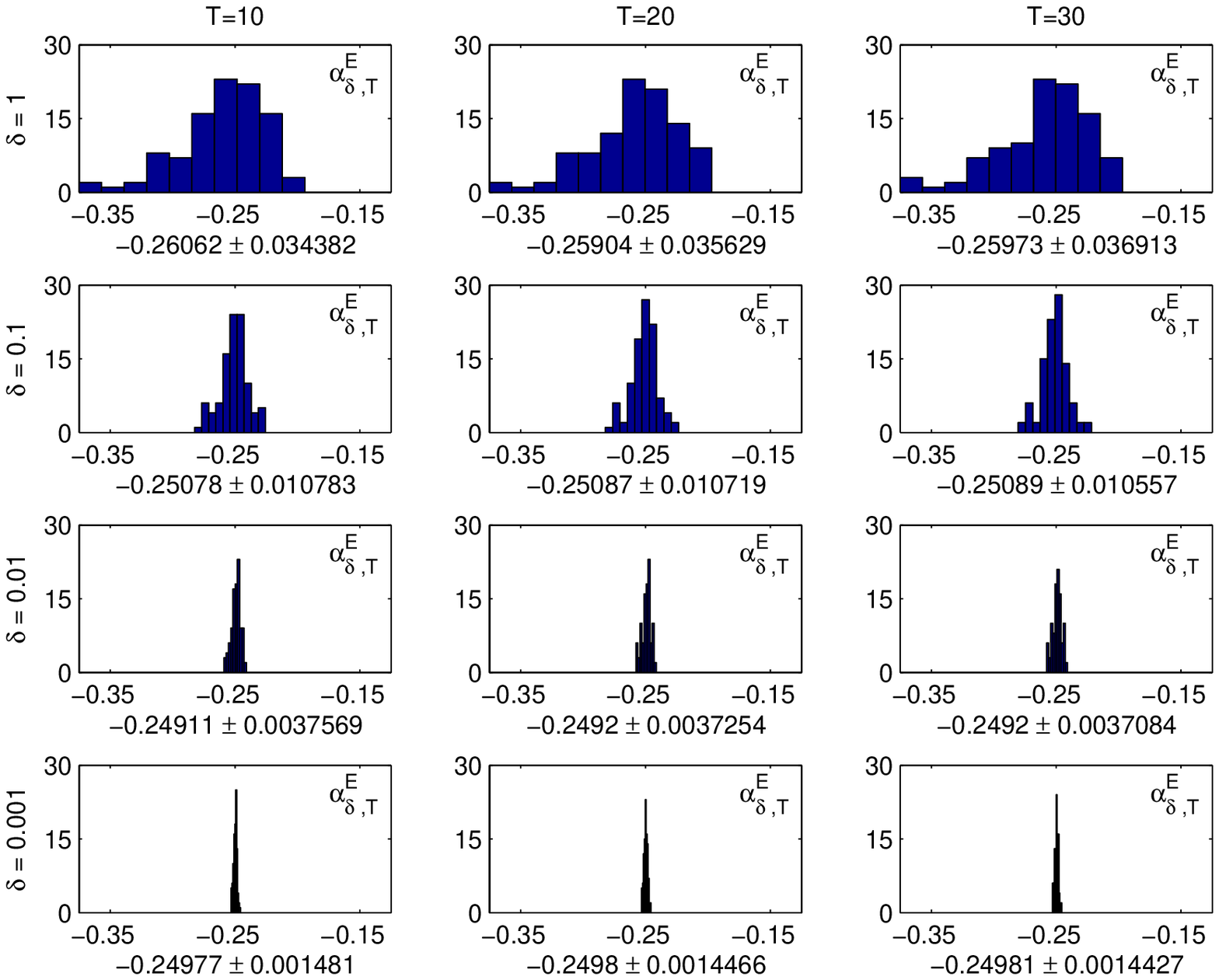} \\
  \includegraphics[width=5in]{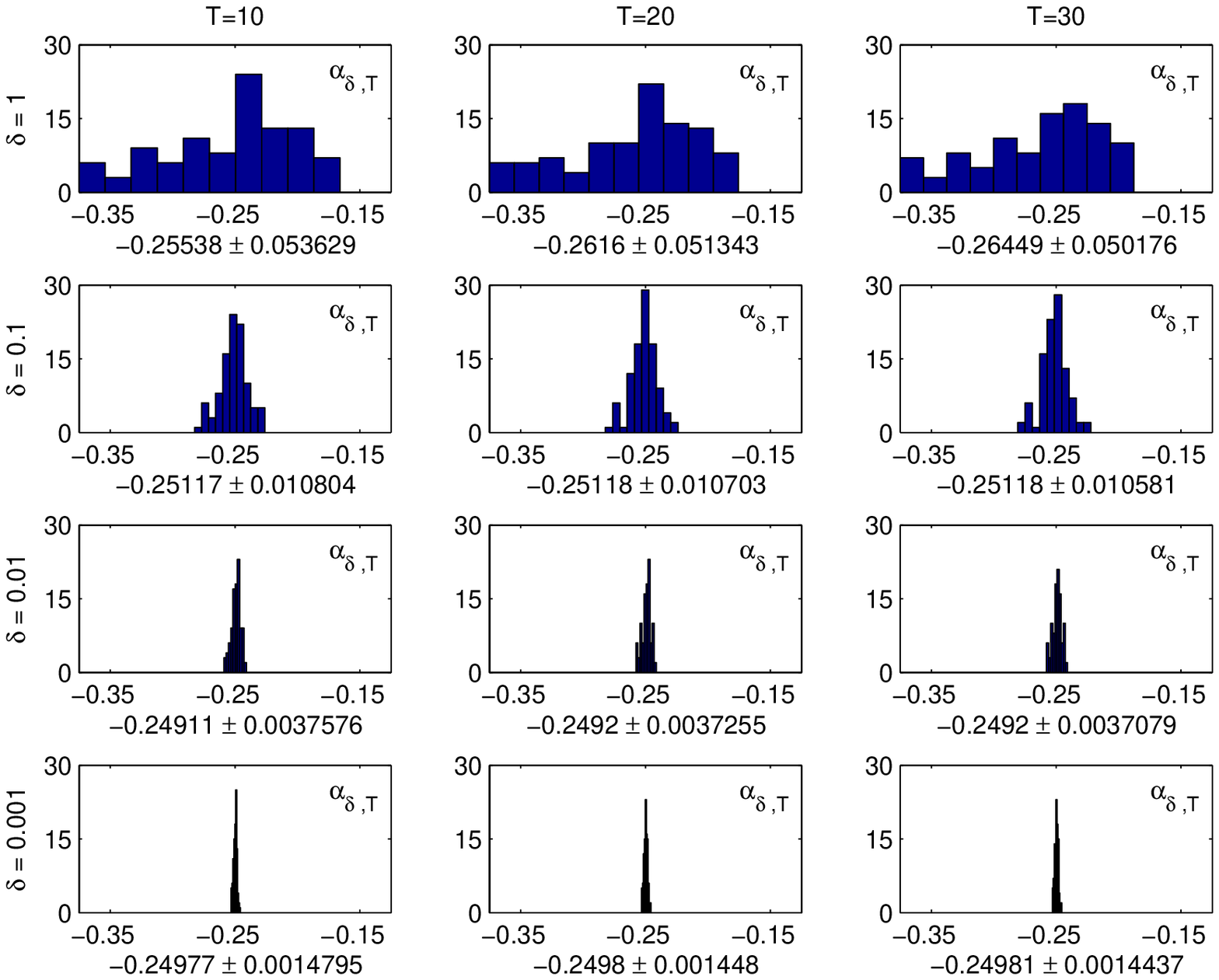}
\end{array}$
\caption{Histograms and confidence limits for the exact ($\widehat{\protect\alpha }_{\protect\delta %
,T}^{E}$) and the conventional ($\widehat{\protect\alpha }_{\protect\delta %
,T}$) innovation estimators of $\protect\alpha $ computed from the
Example 2 data with sampling period $\protect\delta $ and time
interval of length $T$.}
\end{figure}

%% Figure 6
\begin{figure}
\centering
 $\begin{array}{c}
  \includegraphics[width=5in]{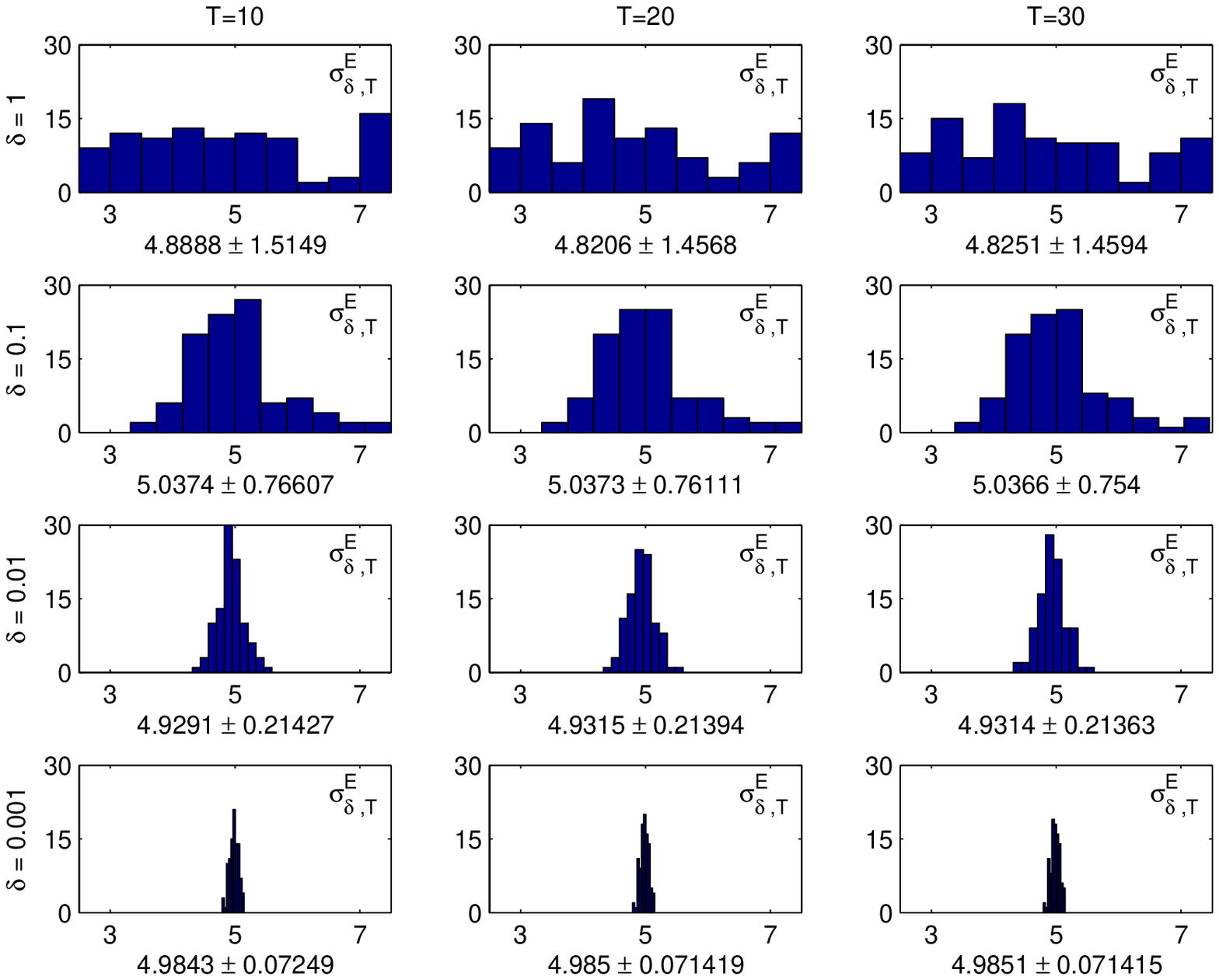} \\
  \includegraphics[width=5in]{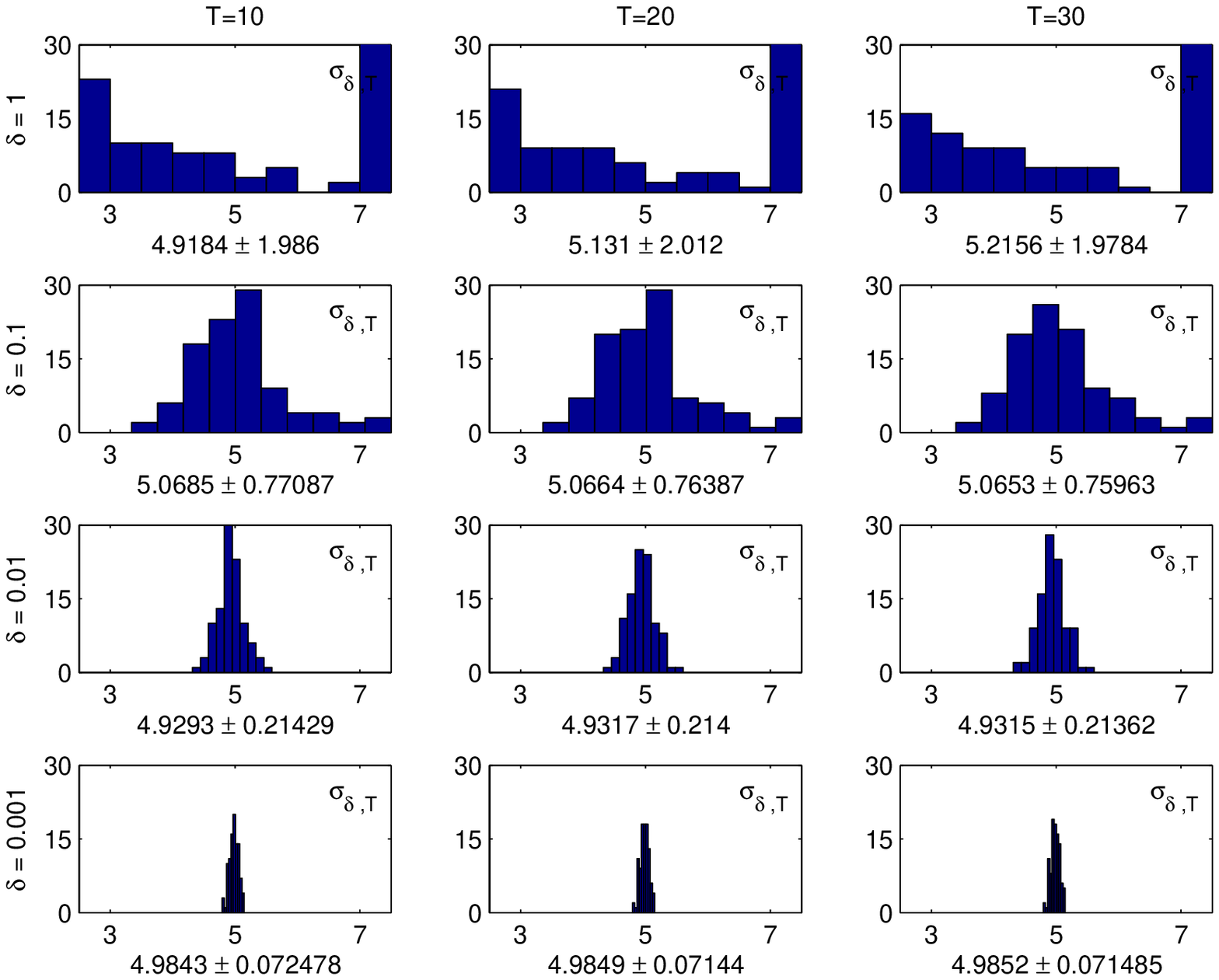}
\end{array}$
\caption{Histograms and confidence limits for the exact ($\widehat{\protect\sigma }_{\protect\delta %
,T}^{E}$) and the conventional ($\widehat{\protect\sigma }_{\protect\delta %
,T}$) innovation estimators of $\protect\sigma $ computed from the
Example 2 data with sampling period $\protect\delta $ and time
interval of length $T$.}
\end{figure}

%% Figure 7
\begin{figure}
\centering
 $\begin{array}{c}
  \includegraphics[width=5in]{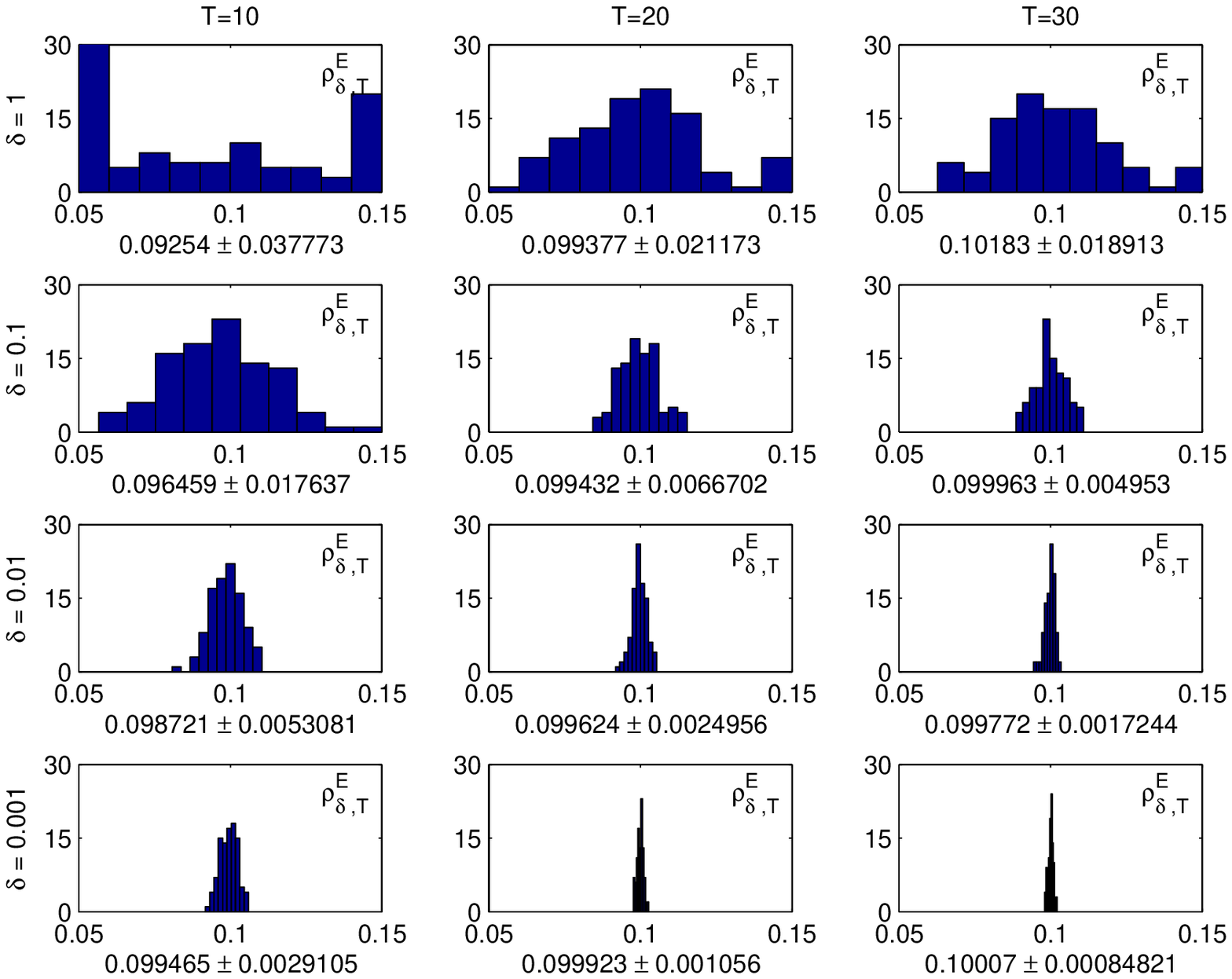} \\
  \includegraphics[width=5in]{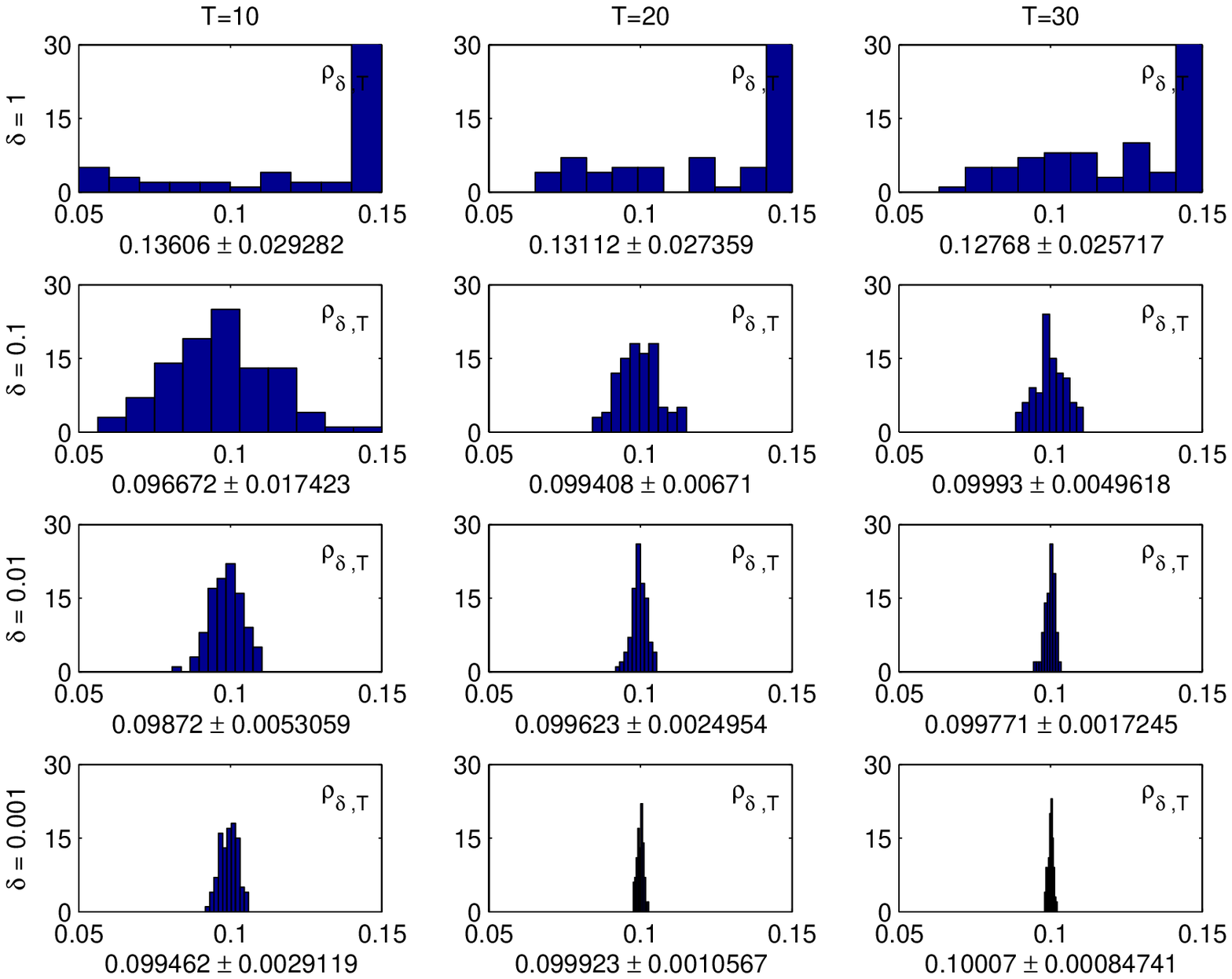}
\end{array}$
\caption{Histograms and confidence limits for the exact ($\widehat{\protect\rho }_{\protect\delta %
,T}^{E}$) and the conventional ($\widehat{\protect\rho }_{\protect\delta ,T}$%
) innovation estimators of $\protect\rho $ computed from the Example
2 data with sampling period $\protect\delta $ and time interval of
length $T$.}
\end{figure}

%% Figure 8a
\subfiguresbegin
\begin{figure}
\centering
 $\begin{array}{c}
  \includegraphics[width=5in]{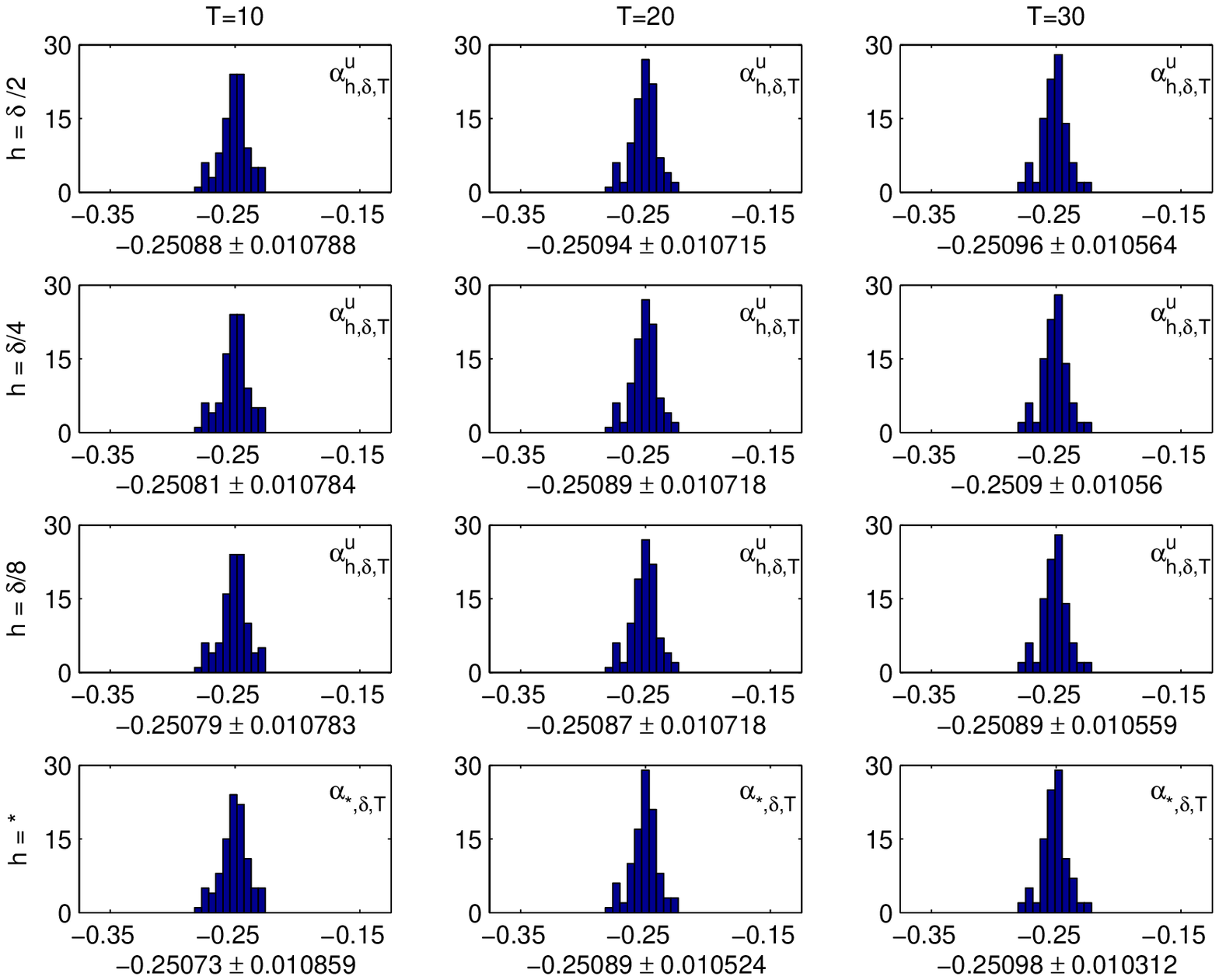} \\
  \includegraphics[width=5in]{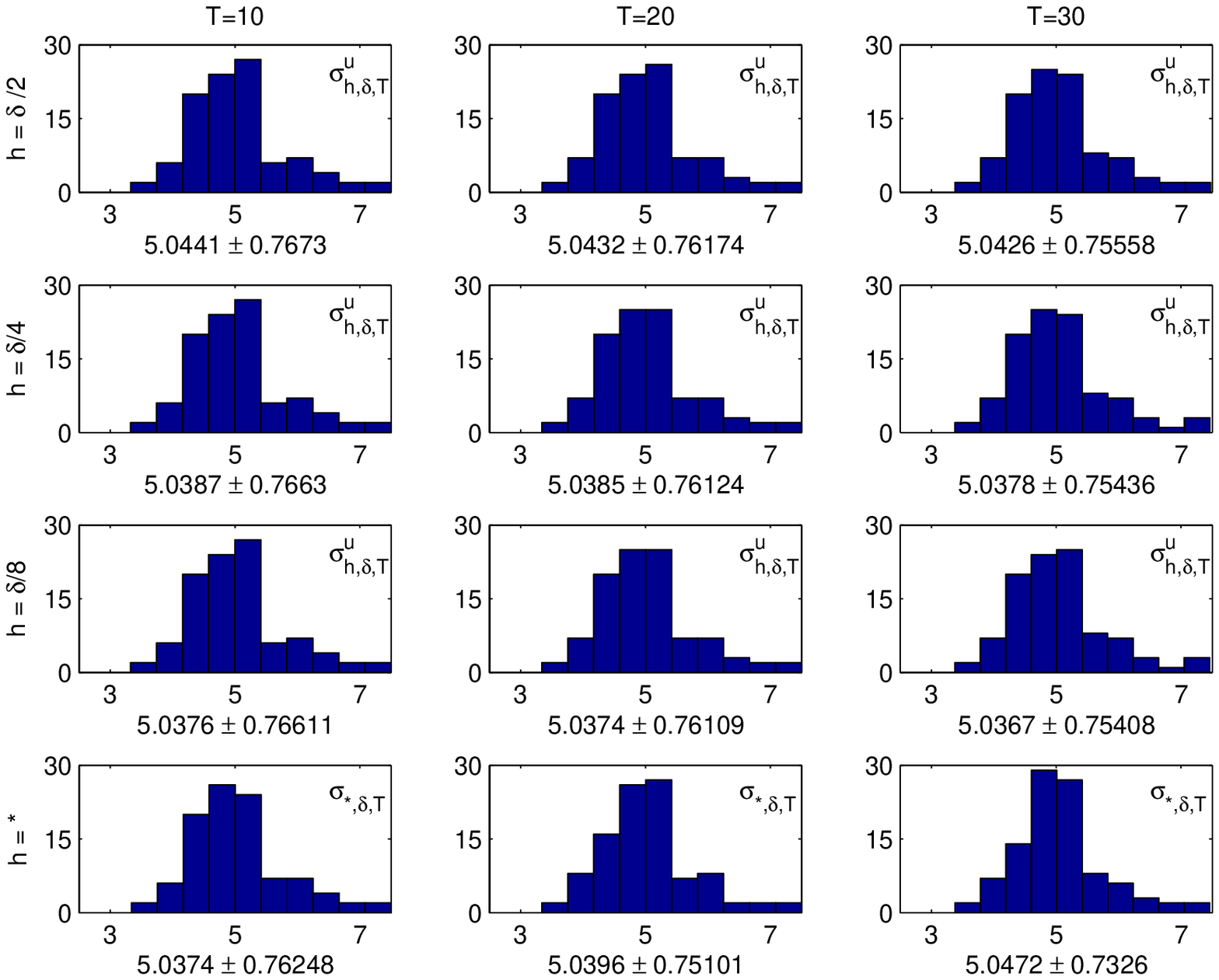}
\end{array}$
\caption{Histograms and confidence limits for the oder-1 innovation estimators of $\protect\alpha $ and $%
\protect\sigma $\ computed on uniform $\left( \protect\tau \right)
_{h,T}^{u}$ and adaptive $\left( \protect\tau \right) _{\cdot ,T}$
time discretizations from the Example 2 data with sampling period
$\protect\delta =0.1$ and time interval of length $T$.}
\end{figure}

%% Figure 8b
\begin{figure}
\centering
\includegraphics[width=5in]{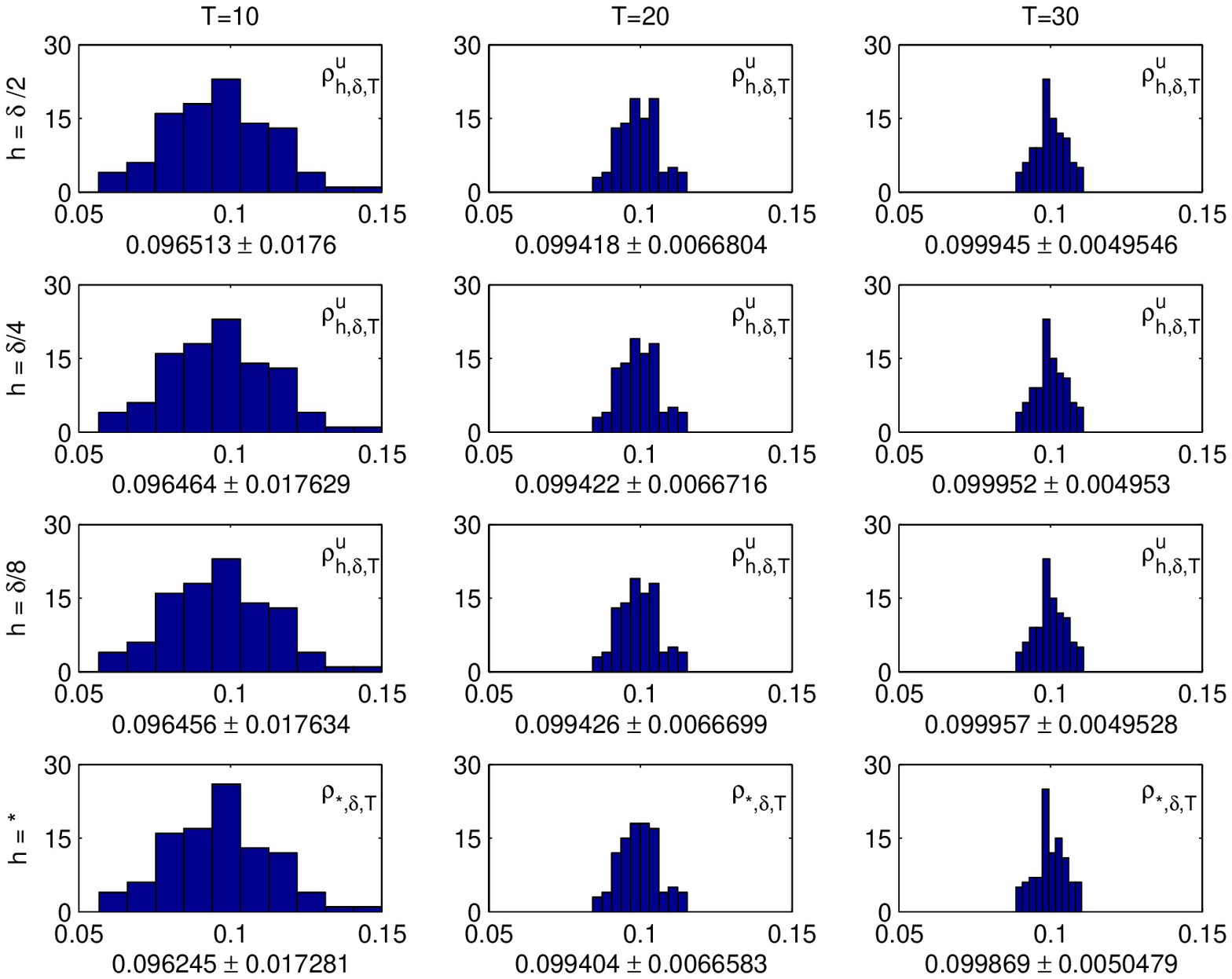}
\caption{Histograms and confidence limits for the oder-1 innovation estimators of $\protect\rho $\ computed on uniform $\left( \protect\tau \right) _{h,T}^{u}$\ and adaptive $%
\left( \protect\tau \right) _{\cdot ,T}$ time discretizations from
the Example 2 data with sampling period $\protect\delta =0.1$ and
time interval of length $T$.}
\end{figure}
\subfiguresend

%% Figure 9
\begin{figure}
\centering
 $\begin{array}{c}
  \includegraphics[width=5in]{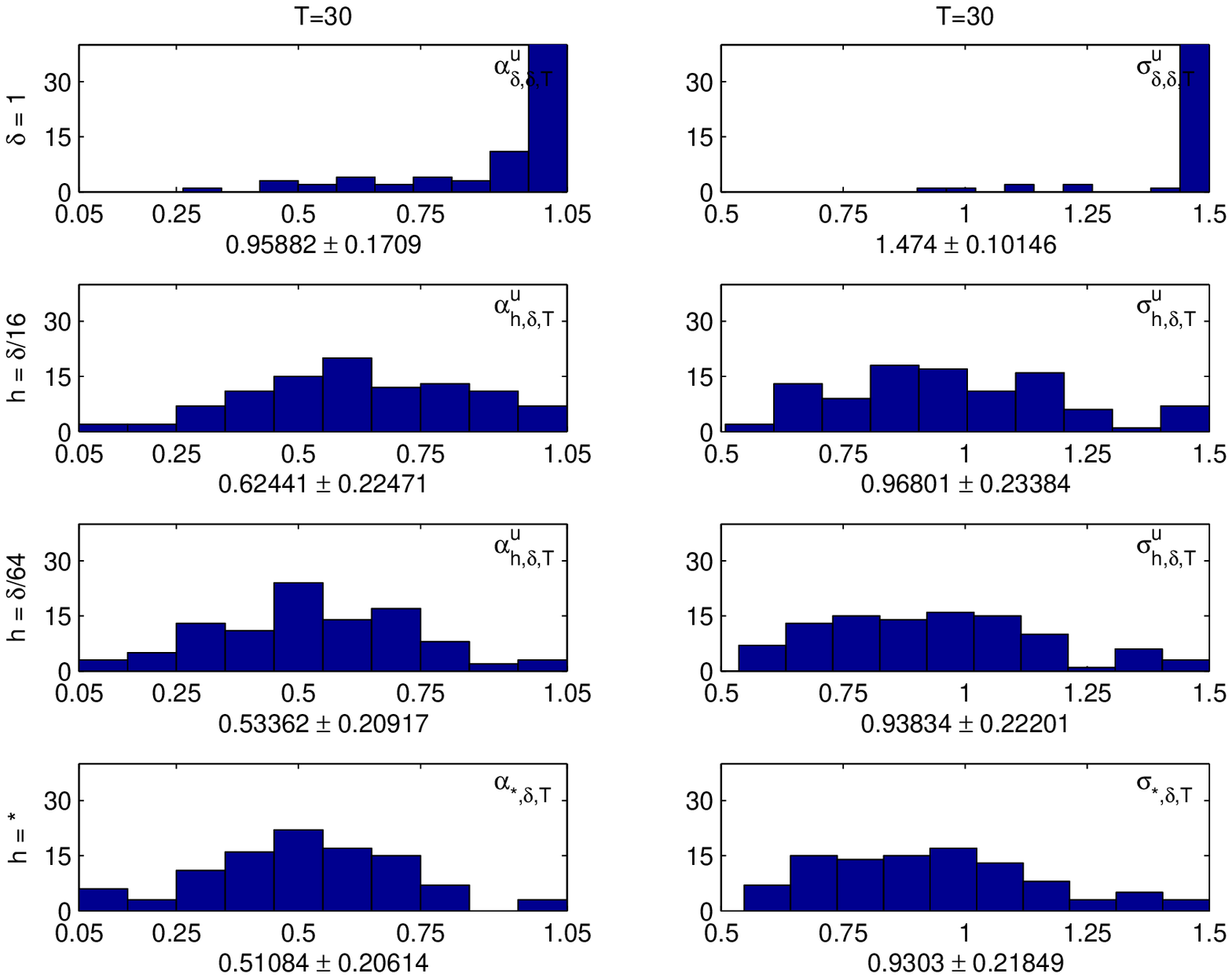} \\
  \includegraphics[width=5in]{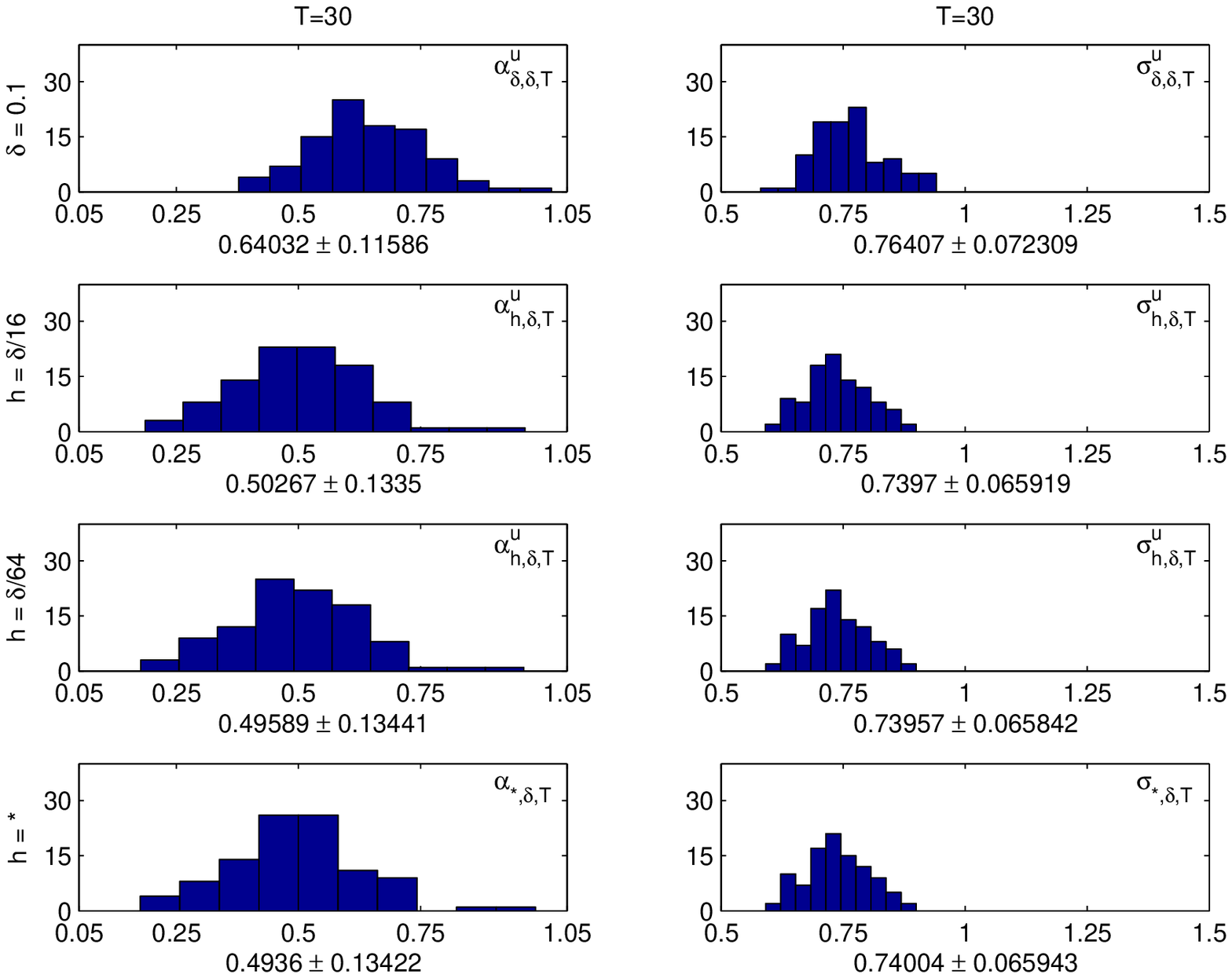}
\end{array}$
\caption{Histograms and confidence limits for the oder-1 innovation estimators of $\protect\alpha $ and $%
\protect\sigma $\ computed on uniform $\left( \protect\tau \right)
_{h,T}^{u}$ and adaptive $\left( \protect\tau \right) _{\cdot ,T}$
time discretizations from the Example 3 data with sampling period
$\protect\delta$ and time interval of length $T=30$.}
\end{figure}

%% Figure 10
\begin{figure}
\centering
 $\begin{array}{c}
  \includegraphics[width=5in]{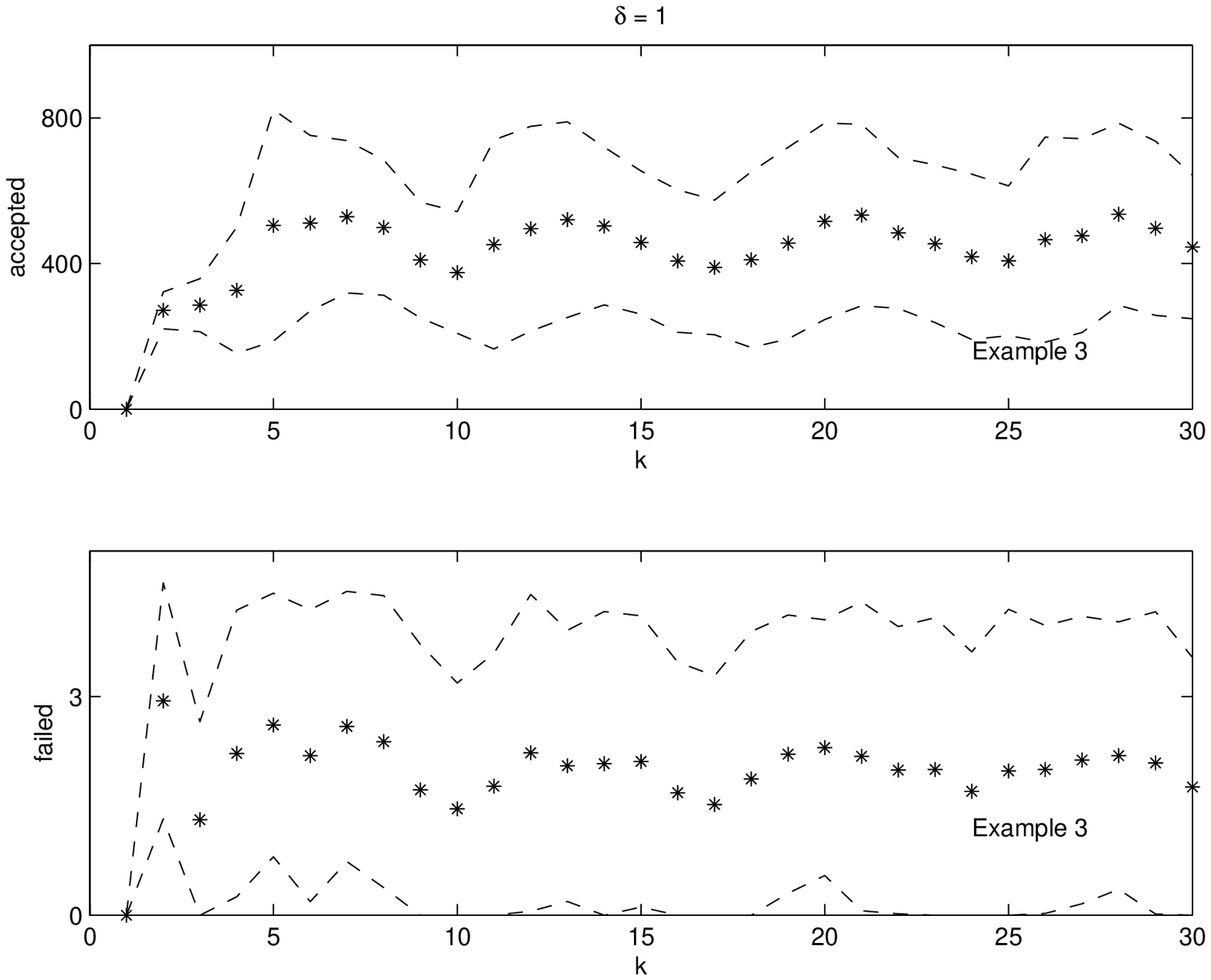} \\
  \includegraphics[width=5in]{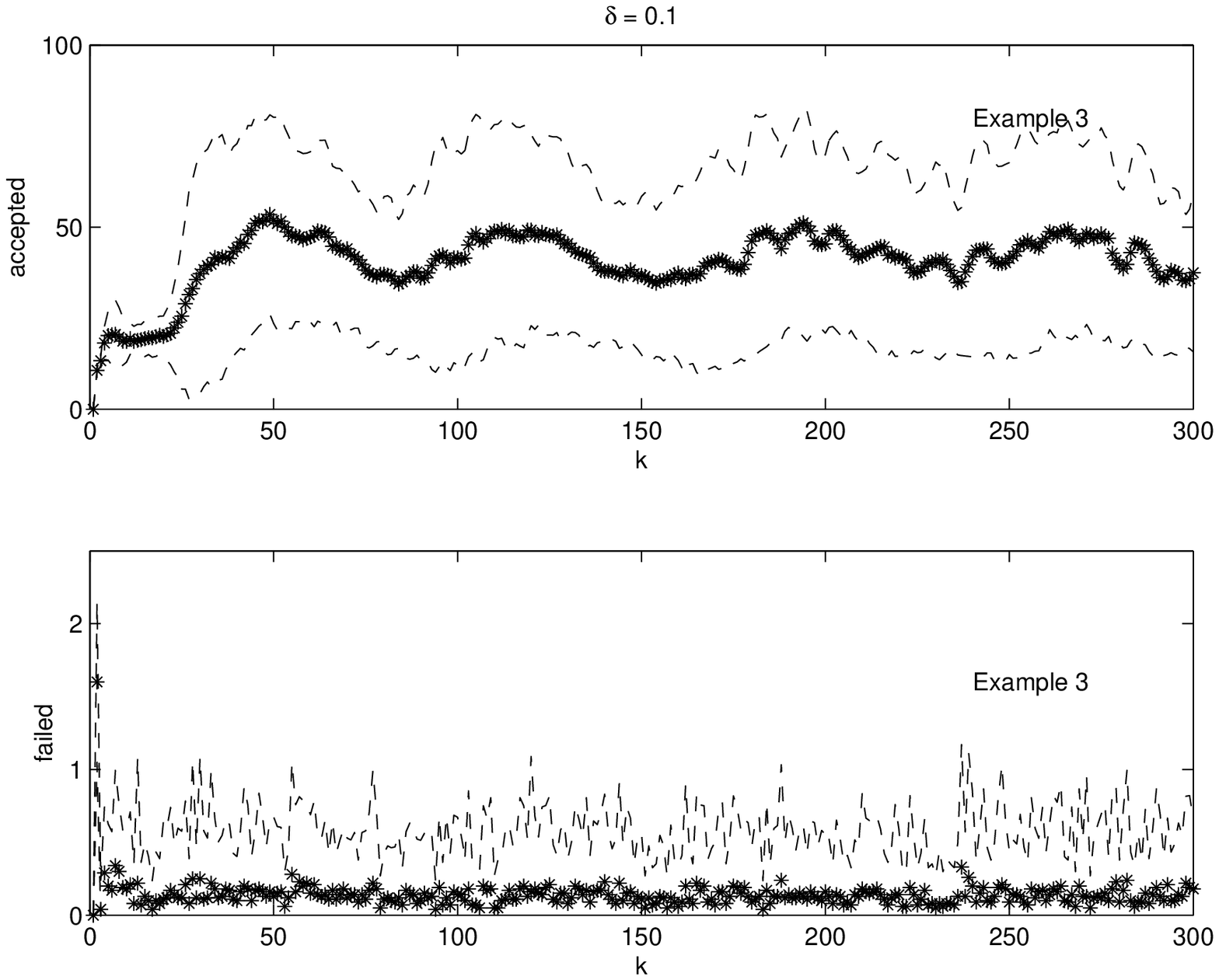}
\end{array}$
\caption{Average (*) and 90\% confidence limits (-) of accepted and
failed steps of the adaptive innovation estimator at each $t_{k}\in
\{t\}_{N}$ in the Example 3.}
\end{figure}

%% Figure 11
\begin{figure}
\centering
 $\begin{array}{c}
  \includegraphics[width=5in]{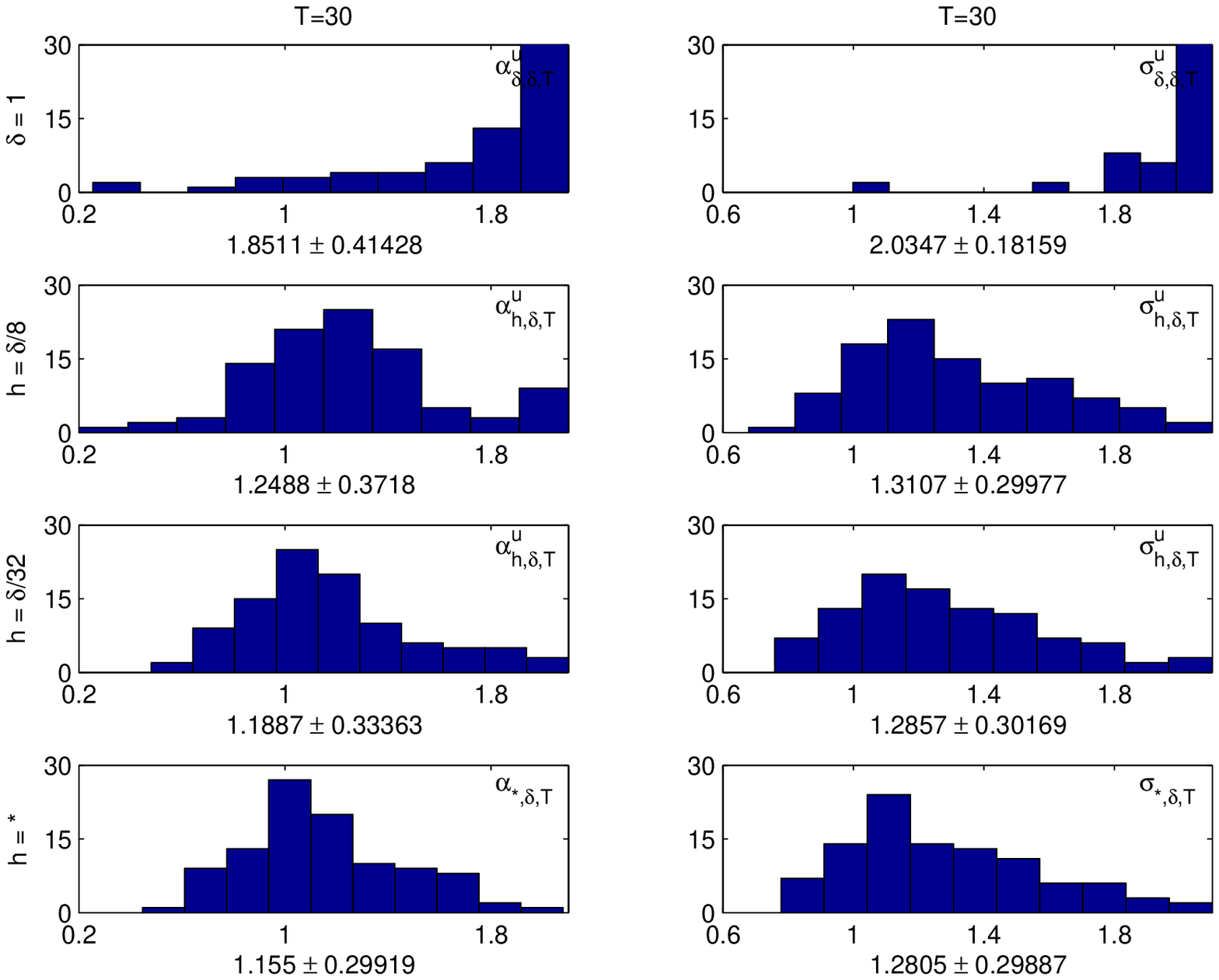} \\
  \includegraphics[width=5in]{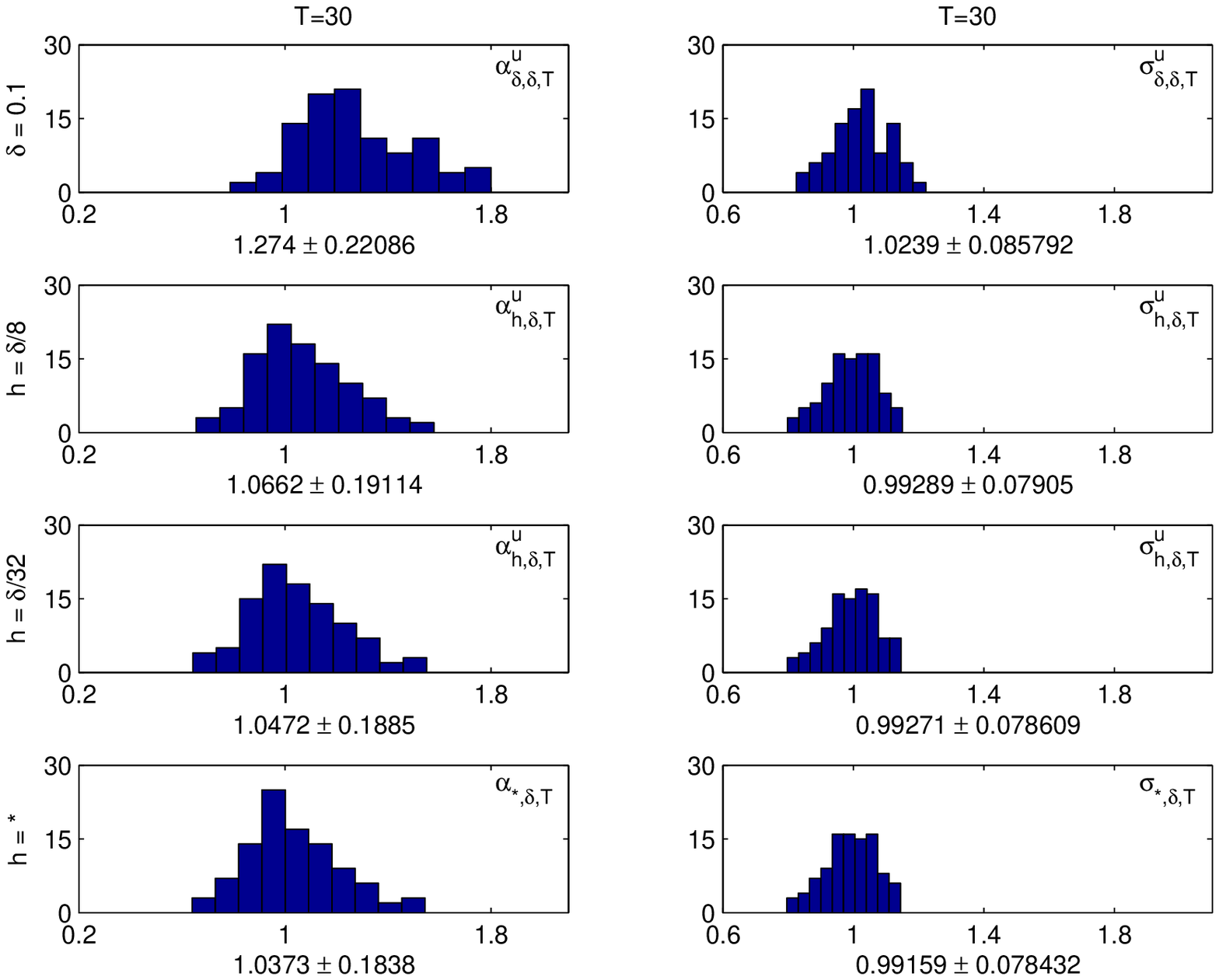}
\end{array}$
\caption{Histograms and confidence limits for the oder-1 innovation estimators of $\protect\alpha $ and $%
\protect\sigma $\ computed on uniform $\left( \protect\tau \right)
_{h,T}^{u}$ and adaptive $\left( \protect\tau \right) _{\cdot ,T}$
time discretizations from the Example 4 data with sampling period
$\protect\delta$ and time interval of length $T=30$}
\end{figure}

%% Figure 12
\begin{figure}
\centering
 $\begin{array}{c}
  \includegraphics[width=5in]{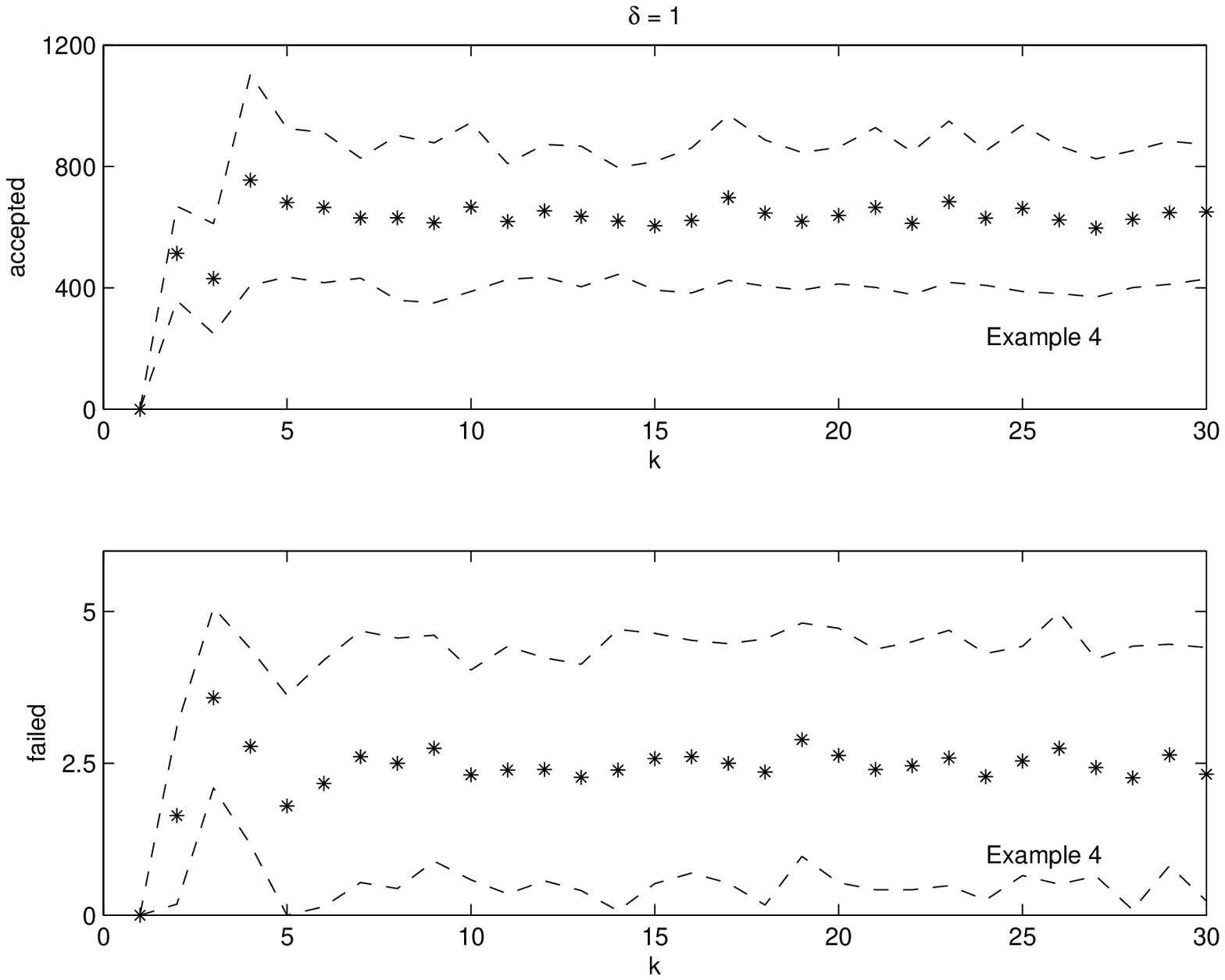} \\
  \includegraphics[width=5in]{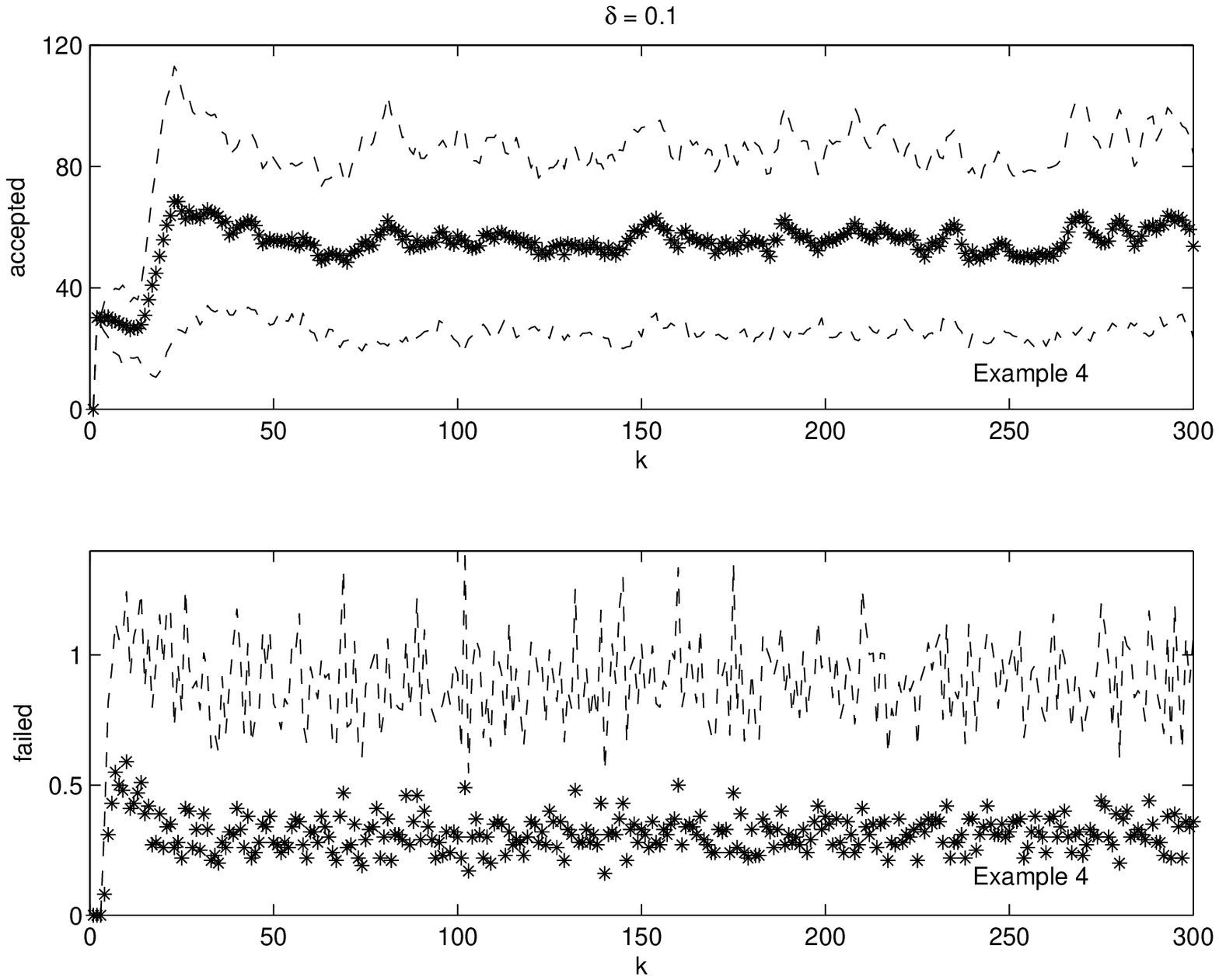}
\end{array}$
\caption{Average (*) and 90\% confidence limits (-) of accepted and
failed steps of the adaptive innovation estimator at each $t_{k}\in
\{t\}_{N}$ in the Example 4.}
\end{figure}

\end{document}